%% file: KTTT.tex
\documentclass[a4paper]{amsart}
\input{macro-KTTT}

\usepackage[british]{babel}
\usepackage{csquotes}

\author[W.~Kim]{Wansu Kim}
\address{Department of Mathematical Sciences, KAIST
291 Daehak-ro, Yuseong-gu, Daejeon, 34141, Republic of Korea}
\email{wansu.math@kaist.ac.kr}

\author[K.-S.~Tan]{Ki-Seng Tan}
\email{kisengtan@gmail.com}

\author[F.~Trihan]{Fabien Trihan}
\address{Sophia University,
Department of Information and Communication Sciences
7-1 Kioicho, Chiyoda-ku, Tokyo 102-8554, JAPAN}
\email{f-trihan-52m@sophia.ac.jp}

\author[K.-W. Tsoi]{Kwok-Wing Tsoi}
\address{Department of Mathematics, National Taiwan University, Taipei 10764, Taiwan}
\email{kwokwingtsoi@ntu.edu.tw}

\title[Refined BSD conjecture]{On a Birch and Swinnerton-Dyer type conjecture for the Hasse-Weil-Artin $L$-functions in characteristic $p>0$}

\begin{document}
\keywords{Conjecture of Birch and Swinnerton-Dyer over global function fields, Hasse--Weil--Artin $L$-functions, Geometric equivariant Tamagawa number conjecture, equivariant Riemann--Roch theorem}

\subjclass[2020]{11G40, 14G10, 11J95, 14C40.}
\begin{abstract}
Given an abelian variety $A$ over a global function field $K$ of characteristic $p>0$ and an irreducible complex continuous representation $\psi$ of the absolute Galois group of $K$, we obtain a BSD-type formula for the leading term of Hasse--Weil--Artin $L$-function for $(A,\psi)$ at $s=1$ under certain technical hypotheses. The formula we obtain can be applied quite generally; for example, it can be applied to the $p$-part of the leading term even when $\psi$ is weakly wildly ramified at some place under additional hypotheses. 

Our result is the function field analogue of the work of D.~Burns and D.~Macias~Castillo \cite{BurnsMaciasCastillo:EquivBSD}, built upon the work on the equivariant refinement of the BSD conjecture by D.Burns, M.~Kakde and the first-named author \cite{BurnsKakdeKim:EquivBSDTame}.  To handle the $p$-part of the leading term, we need the Riemann--Roch theorem for equivariant vector bundles on a curve over a finite field generalising the work of S.~Nakajima \cite{Nakajima:GalMod-tame}, B.~K\"ock \cite{Koeck:EquivRR}, and H.~Fischbacher-Weitz and B.~K\"ock \cite{FischbacherWeitz-Koeck:EquivRR}, which is of independent interest. 
\end{abstract}

\maketitle

\tableofcontents

\section{Introduction}
In \cite{Tate:BourbakiBSD}, Tate gave a uniform formulation of the conjecture of Birch and Swinnerton-Dyer (or the \emph{BSD conjecture}) for abelian varieties over global fields of any characteristic. Furthermore, for a Jacobian of a curve over a global function field of characteristic $p>0$, the ``prime-to-$p$ part'' of the full conjecture (including the leading term formula up to $p$-power ambiguity) was obtained assuming finiteness of a certain object closely related to the Tate--Shafarevich group; see \cite[Theorem~5.2]{Tate:BourbakiBSD} for further details. It is now known that for an abelian variety $A$ over a global function field $K$ the full BSD conjecture follows from the finiteness of $\Sha(A/K)\{\ell\}$, the $\ell$-primary part of the Tate--Shafarevich group for some prime $\ell$; see Kato--Trihan \cite{KatoTrihan:BSD} for the precise result, and its introduction for the history. 
We note that the \emph{$p$-part} of the argument in \cite{KatoTrihan:BSD} heavily relies on the theory of $p$-adic cohomology, as anticipated by Tate \cite[p.~438]{Tate:BourbakiBSD}. 

In the work of D.~Burns, M.~Kakde and the first-named author \cite{BurnsKakdeKim:EquivBSDTame}, we formulated an \emph{equivariant refinement} of the BSD conjecture for abelian varieties over a global function field of characteristic $p>0$ in the spirit of the Equivariant Tamagawa Number Conjecture for motives over a number field. Given an abelian variety $A$ over a global function field $K$ and a finite Galois extension $L/K$, the equivariant BSD conjecture \cite[Conjecture~4.3]{BurnsKakdeKim:EquivBSDTame} refines the BSD conjecture for $A/L$ by predicting that the ``\emph{derived} Galois module structure'' of arithmetic invariants of $A/L$ should encode the (suitably normalised) leading terms of Hasse--Weil--Artin $L$-functions at $s=1$ attached to $A$ and complex irreducible characters of $G\coloneqq\Gal(L/K)$, as well as certain ``algebraic relations'' thereof.\footnote{See \cite[Proposition~4.8]{BurnsKakdeKim:EquivBSDTame} for an example of algebraic relations implied by the equivariant BSD conjecture when $L/K$ is a $p$-extension.} (To forumlate the conjecture we need certain \emph{perfect complexes} of integral Galois modules, which we refer to as the ``\emph{derived} Galois module structure''.) 

In this paper, we consider the following natural question. 
\begin{ques}\label{q:main}
    Let $A$ be an abelian variety over a global function field $K$, and we choose a complex irreducible character $\psi$ of the absolute Galois group of $K$. Assuming finiteness of the relevant Tate--Shafarevich group, can we get an \emph{explicit} formula of the (suitably normalised) the leading term of the Hasse--Weil--Artin $L$-function attached to $(A,\psi)$ at $s=1$ (possibly imposing additional hypotheses that are not too restrictive)?
\end{ques}
Note that the equivariant refinement of the BSD conjecture is known up to torsion in relative $\Krm_0$ assuming a suitable finiteness condition on the Tate--Shafarevich group; \emph{cf} \cite[Theorem~4.10]{BurnsKakdeKim:EquivBSDTame}. In fact, ignoring such torsion ambiguity does \emph{not} affect individual leading terms, though we may lose algebraic relations among them. However, \emph{loc.~cit.} does not completely resolve the question; in fact, the resulting formula depends on the ``derived Galois-module structure'' of certain perfect complexes, so it is not explicit enough. (See Corollary~\ref{cor:BKK} for the precise statement.) Nonetheless, one take \cite[Theorem~4.10]{BurnsKakdeKim:EquivBSDTame} and Corollary~\ref{cor:BKK} as a starting point, and manage to extract some non-trivial and explicit formula on the leading term at $s=1$ of $L$-function attached to $(A,\psi)$ under some additional technical hypotheses; \emph{cf.} Assumption~\ref{ass:BC}.

Let us briefly indicate the nature of Assumption~\ref{ass:BC}. Recall that the equivariant BSD conjecture \cite[Conjecture~4.3]{BurnsKakdeKim:EquivBSDTame} involves two perfect complexes: a kind of Selmer complex for $A/L$, and a coherent cohomology of a certain equivariant vector bundle. Our additional hypotheses are mainly to simplify the ``Selmer complex term'', clearly inspired by the number field analogue of our result obtained by D.~Burns and D.~Macias~Castillo (\emph{cf.}~\cite{BurnsMaciasCastillo:EquivBSD}, especially the set of hypotheses at the beginning of \S6). Our main work is to control the ``coherent cohomology term'' (or rather, the \emph{ramification correction} to the local volumes, so to speak) under a \emph{mild} hypothesis -- namely, Assumption~\ref{ass:BC}(\ref{ass:BC:p}) -- and thereby obtain a formula for the \emph{$p$-part} of the leading term in a satisfying generality. If $A$ has semistable reduction at all places of $K$, then our main result can be applied if $\psi$ has tame ramification at worst (or even, we allow ``shallow wild ramification'') assuming finiteness of a suitable Tate--Shafarevich group.

Let us set up the notation for more detailed introduction. In the setting of Question~\ref{q:main}, let $L/K$ be a finite Galois extension with $G\coloneqq\Gal(L/K)$ such that $\psi$ factors through $G$. Suppose that $\psi$ can be defined over a number field $E\subset \CC$, and we fix the underlying $E$-vector space $V_\psi$ for the representation $\psi$. Let $Z$ be the set of places of $K$ consisting exactly of the places ramified in $L/K$ and the bad reduction places for $A$.

We consider the Hasse--Weil--Artin $L$-function $L_U(A,\psi,s)$ without Euler factors at $Z$ as in \eqref{eq:L-ftn}, where $U$ denotes the set of places of $K$ away from $Z$. We normalise its leading term $\Lscr_U(A,\psi)$ at $s=1$ as \eqref{eq:rk-LT} so that we have $\Lscr_U(A,\psi)\in E^\times$. In particular, for any place $\lambda$ of $E$ it makes sense to consider the $\lambda$-adic valuation $v_\lambda\big(\Lscr_U(A,\psi)\big)$.

If $\lambda$ is a place over $\ell$ coprime to $|G|$ (which applies to all but finitely many places of $E$), then $v_\lambda\big(\Lscr_U(A,\psi)\big)$ is quite easy to describe since the group ring $\ZZ_\ell G$ is rather simple in terms of homological algebra; \emph{cf.} Proposition~\ref{prop:main-c-t}. For a place $\lambda$ of $E$ over a prime $\ell$ dividing $|G|$ however, one cannot really expect an explicit description of $v_\lambda\big(\Lscr_U(A,\psi)\big)$ without imposing an additional assumption for $(A,\psi)$ to simplify the homological algebra involved. And unsurprisingly, it requires much harder extra work to handle places $\lambda$ over $p$, the characteristic of $K$, when $p$ divides $|G|$. 
\begin{thm}[{\emph{Cf.}~Theorem~\ref{th:main}}]\label{th:intro}
    In addition to the above setting, suppose that $\Sha(A/L)$ is finite. Choose a prime $\ell$ that does not divide any of $|A(L)_{\tors}|$, $|A^t(L)_{\tors}|$ and $|\Acal_L(k_w)|$, where $\Acal_L$ is the N\'eron model of $A/L$ and $k_w$ is the residue field at a place $w$ of $L$ above some $v\in Z$. If $\ell = p$ then we assume that $A$ has semistable reduction at all places of $K$ and the extension $L/K$ is at worst weakly ramified at each place in the sense of Definition~\ref{def:weak-ram}. Then for any place $\lambda$ of $E$ above $\ell$ we have 
        \[
        \Lscr_U(A,\psi)\Ocal_{E,\lambda} = \vol_Z(A/K)^{\deg\psi}\cdot \lo_{Z_L}(A,\psi) \cdot \frac{\Reg^\psi_\lambda}{|G|^{r_{\alg}(\psi)}}\cdot\Char_{\psi}\big(\Sha^\vee_{\psi,\lambda}(A/L)\big),
        \]
        where $\Ocal_{E,\lambda}$ is the $\lambda$-adic completion of $\Ocal_E$, $\Reg^\psi_\lambda$ is the $\psi$-twisted regulator for $A$ (\emph{cf.} Definition~\ref{def:psi-ht}) and $r_{\alg}(\psi)$ is the rank of the ``$\psi$-part of $A(L)$'' \eqref{eq:psi-MW-rk}. Lastly, $\Char\big(\Sha^{\vee}_{\psi,\lambda}(A/L)\big)$ is the characteristic ideal \eqref{eq:char-ideal} of $\Sha^\vee_{\psi,\lambda}(A/L)$ \eqref{eq:Sha-psi}, and $\vol_Z(A/K)$ and $\lo_{Z_L}(A,\psi)$ are $p$-power integers defined in Theorem~\ref{th:main}.
\end{thm}

    We actually obtain a result in a more general setting where $A$ has semistable reduction at all places of $L$ (instead of $K$),  $L/K$ is at worst weakly ramified at each place, and $L/K$ is tamely ramified at all places of $K$ where $A$ has non-semistable reduction. The formula becomes  more complicated in this generality, and we refer to the main body of the text.

    If we choose $L=K$ (so $\psi$ is the trivial character and $E=\QQ$), then Theorem~\ref{th:intro} is compatible with the $\ell$-part of the classical BSD formula \cite[(1.8.1)]{KatoTrihan:BSD}; indeed, in the setting of of Theorem~\ref{th:main}  we have $\lo_{Z_L}(A,\psi) = 1$ and  $\vol_{Z}(A/K)$ is the $p$-part of $\vol\big(\prod_{v\in Z}A(K_v)\big)$, using the notation of \emph{loc.~cit.} In general, the $p$-power integer $\lo_{Z_L}(A,\psi)$, given by an explicit local formula, can be thought of as the ``ramification correction'' to the volume term. In fact, we have $\lo_{Z_L}(A,\psi)=1$ if the ramification index of $L/K$ at each place is a power of $p$ (including the case where $L/K$ is unramified everywhere). If the ramification index of $L/K$ at each place divides $p-1$ and $\deg\psi = 1$ then we have 
    \[
    \log_p\big(\lo_{Z_L}(A,\psi)\big)= \frac{\dim A}{|G|}\sum_w j_{w,\psi}[k_w:\FF_p],
    \]
    where $w$ runs through all places of $L$ ramified over $K$, and $j_w$ is determined so that the inertia subgroup at $w$ acts on $\mfr_w^{j_w}/\mfr_w^{j_w+1}$ via  the restriction of $\psi$,  where $\mfr_w$ is the maximal ideal of $\Ocal_{X_L}$ corresponding to $w$. (See Remarks~\ref{rmk:main-thm-simplified} and \ref{rmk:ra-psi} for further details.) We believe that the explicit formula for $\lo_{Z_L}(A,\psi)$ is new even when $L/K$ is cyclic and tame.
    
    We note that for abelian varieties defined over a number field, an analogous result was obtained by D.~Burns and D.~Macias~Castillo (\emph{cf.} \cite[Proposition~7.3]{BurnsMaciasCastillo:EquivBSD}), which clearly inspired our result.

    Let us now list some of the main ingredients of the proof. Let $\pi\colon X_L\to X$ denote the covering of smooth projective curves corresponding to $L/K$, and let $Z_L\subset X_L$ be the closed subset consisting of places over $Z$. Let $\Acal_L$ be the N\'eron model over $X_L$ of $A/L$. 
    

    Let us specialise to the case where $\ell = p$, which is the main case of interest. To apply \cite[Theorem~4.10]{BurnsKakdeKim:EquivBSDTame} one needs to choose a suitable $G$-stable subbundle $\Lcal\subseteq\Lie(\Acal_L)(-Z_L)$ such that $\RGamma(X_L,\Lcal)$ is a perfect $\FF_p G$-complex. 
    The choice of such $\Lcal$ could a priori be very inexplicit, but we show that we can take $\Lcal = \Lie(\Acal_L)(-Z_L)$ provided that the following conditions are satisfied; \emph{cf.} Assumption~\ref{ass:BC}(\ref{ass:BC:p}).
    \begin{ass}\label{ass:intro}
    \begin{itemize}
        \item $L/K$ is weakly ramified at all places (\emph{cf.} Definition~\ref{def:weak-ram}), and
        \item if $L/K$ is wildly ramified at a place $v$ of $K$ then $A$ has semistable reduction at $v$.
    \end{itemize}    
    \end{ass}
    See Corollary~\ref{cor:descent-Neron}(\ref{cor:descent-Neron:perf}) and Proposition~\ref{prop:c-t-SC}(\ref{prop:c-t-SC:p}) for the precise statement. Note that the main ingredient of the proof is K\"ock's local integral normal basis theorem for weakly ramified extensions \cite[Theorem~1.1]{Koeck:EquivRR}, which we recall in Theorem~\ref{th:int-normal-basis}.

    Suppose that Assumption~\ref{ass:intro} is valid, and set $\Lcal = \Lie(\Acal_L)(-Z_L)$. Then we get a $p$-adic Selmer complex with support condition  $\SC_{Z_L,p}(A,L/K)\in \Dperf(\ZZ_p G)$ following the construction of \cite[Proposition~3.7(i)]{BurnsKakdeKim:EquivBSDTame}, and under the assumption of Theorem~\ref{th:intro} one can compute its ``$\psi$-isotypic parts'' by the same argument as \cite[Proposition~7.3(ii)]{BurnsMaciasCastillo:EquivBSD}; \emph{cf.} Proposition~\ref{prop:BC}.

     By the equivariant BSD conjecture modulo torsion, it remains to compute the $\psi$-isotypic part of $\RGamma(X_L,\Lie(\Acal_L)(-Z_L))^\vee$. This can be achieved using the following result.
    \begin{thm}[\emph{Cf.} {Theorem~\ref{th:EquivRR}, Corollary~\ref{cor:descent-Neron}}]\label{th:Intro-RR}
    Set $\Lcal\coloneqq \Lie(\Acal_L)(-Z_L)$ and view it as a $G$-equivariant vector bundle on $X_L$.
    If Assumption~\ref{ass:intro} is satisfied,  then $\RGamma(X_L,\Lcal)$ can be represented by a two-term complex $[C^0\to C^1]$ for some projective $\FF_pG$-modules $C^0,C^1$. Furthermore, we have an explicit formula for 
        \[[C^0]-[C^1] - \chi(\Lcal^G)[\FF_pG] \in \Krm_0(\FF_p G)\]
    in terms of the inertia action on the completed stalk $\widehat\Lcal_w$ at each $w\in Z_L$. (Here, $\chi(\Lcal^G) = \log_p(|\Hrm^0(X,\Lcal^G)| / |\Hrm^1(X,\Lcal^G)|)$.)
    \end{thm}
    The precise formula is quite complicated, and we refer to the main body of the text.

    The statement can be divided into two steps. 
    By analysing the N\'eron models over $X_L$ and $X$, we deduce a certain local property of $\Lcal$ in terms of ramification at each $w\in Z_L$ (\emph{cf.} Corollary~\ref{cor:descent-Neron}). And for $G$-equivariant vector bundles on $X_L$ satisfying the same local property satisfied by $\Lcal$, we prove a kind of ``equivariant Riemann--Roch theorem''; \emph{cf.} Theorem~\ref{th:EquivRR}.
    When $X_L\to X$ is a \emph{tame} $G$-cover of curves over an algebraically closed field $k$, then the Euler characteristic of a $G$-equivariant vector bundle $\Lcal$ in $\Krm_0(kG)$ was computed modulo $[kG]$  by S.~Nakajima \cite[Theorem~2]{Nakajima:GalMod-tame}.
    The rank-$1$ case of the equivariant Riemann--Roch theorem was obtained by Fischbacher-Weitz and K\"ock \cite[\S3, Theorem~12]{FischbacherWeitz-Koeck:EquivRR} (built upon the case of curves over algebraically closed field \cite[Theorem~4.5]{Koeck:EquivRR}). We give a common generalisation of these arguments to obtain the equivariant Riemann--Roch theorem sufficient for the proof of our main result, Theorem~\ref{th:intro}.

    By Theorem~\ref{th:Intro-RR} we can compute the $\psi$-isotypic part of $\RGamma(X_L,\Lcal)^\vee$, and compare it with the volume term and $\lo_{Z_L}(A,\psi)$. In case where $A$ admits non-semistable reduction at some place of $K$, the volume term needs to be corrected by analysing the behaviour of N\'eron models over tame extensions; \emph{cf.} Proposition~\ref{prop:descent-Neron} and Theorem~\ref{th:main}.

    Let us outline the contents of the paper. In \S\ref{sec:local} we collect various results for semilinear representations of decomposition groups, including K\"ock's local integral normal basis theorem. In \S\ref{sec:EquivRR} we formulate and prove the ``Riemann--Roch theorem for equivariant vector bundles'' (\emph{cf.} Theorem~\ref{th:EquivRR}). In \S\ref{sec:K-thy} we review the relative $K_0$-groups and re-interpret Theorem~\ref{th:EquivRR} using relative $K_0$-group. In \S\ref{sec:Neron} we collect various results on N\'eron models (including the behaviour under tame ramification) and show that the equivariant Riemann--Roch theorem can be applied to $\Lie(\Acal_L)(-Z_L)$ under Assumption~\ref{ass:intro}. In \S\ref{sec:BKK} we review the equivariant refinement of the BSD conjecture in \cite{BurnsKakdeKim:EquivBSDTame}, and in \S\ref{sec:main} we give a proof of the main theorem (\emph{cf.} Theorem~\ref{th:main}). In \S\ref{sec:exa} we give some examples in which our main theorem can be applied unconditionally.

\begin{notconv}For any commutative ring $R$ (necessarily with $1$) and for any group $G$, we let $RG$ denote the group ring of $G$ over $R$. We may write $R[G]$ for $RG$ if there is any risk of confusion.

By a \emph{$G$-representation} $\psi$, we mean a finite-dimensional $\CC$-linear $G$-representation. Let $V_\psi$ denote the (left) $\CC G$-module underlying $\psi$. As a standard fact, there exists an $EG$-module $V_{\psi,E}$ for some number field $E\subset\CC$ such that we have a $\CC G$-isomorphism $\CC\otimes_E V_{\psi,E}\cong V_\psi$. We will also use $V_\psi$ to refer to $V_{\psi,E}$ if there is no risk of confusion. 

By $T_\psi = T_{\psi,\Ocal_E}$, we denote a (chosen) $G$-stable $\Ocal_E$-lattice in $V_\psi$. By abuse of notation, we also let $\psi$ denote the ring homomorphisms $\ZZ G\to \End_{\Ocal_E}(T_\psi)$ and $\QQ G\to \End_E(V_\psi)$ defined by the $G$-action on $T_\psi$ and $V_\psi$.

We write $(-)^\ast$ for the linear dual, and $(-)^\vee$ for the Pontryagin dual. For $T_\psi$ as above, we regard $T_\psi^\ast$ and $T_\psi^\vee$ as \emph{right} $\Ocal_EG$-modules. We write $\check\psi$ for the contragredient of $\psi$.

For any ring $A$ (necessarily with $1$ but not necessarily commutative), we let $\Drm(A)$ denote the derived category of complexes of $A$-modules, and $\Dperf(A)$ for the triangulated full subcategory of perfect complexes of $A$-modules. 
For any $C^\bullet \in \Dperf(A)$, we define its \emph{Euler characteristic} as follows:
\begin{equation}\label{eq:Euler-K0}
    \chi_A(C^\bullet)\coloneqq \sum_i(-1)^i[C^i] \in \Krm_0(A),
\end{equation}
where $\Krm_0(A)$ is the Grothendieck group of the category of finitely generated projective $A$-modules and $[C^i]\in \Krm_0(A)$ denote the class of $C^i$.
\end{notconv}

\section{Review of local integral normal basis theorems}\label{sec:local}
In this section, we collect various standard results on lattices in semilinear Galois modules for finite extensions of local fields, following Chinburg \cite{Chinburg:dR-tame} and K\"ock \cite{Koeck:EquivRR}.

Let $K_v$ be a complete discrete valuation field with perfect residue field $k_v$ of characteristic $p>0$. Let $\Ocal_v$ and $\mfr_v$ respectively denote the valuation ring and its maximal ideal.
We fix a finite Galois extension $L_w/K_v$ with valuation ring $\Ocal_w$, maximal ideal $\mfr_w$, and residue field $k_w$. Set $G_w\coloneqq\Gal(L_w/K_v)$, and write $I_w$ and $P_w$ for the inertia and wild inertia subgroups, respectively. (Although the results in this section are purely local, we will later apply them in the setting where $L_w/K_v$ arises from some global extension $L/K$ via completing at $w \mid v$.)

By \emph{semilinear $G_w$-representation over $\Ocal_w$}, we mean a finite free $\Ocal_w$-module $W_w$ equipped with \emph{semilinear} $G_w$-action.
\begin{lem}\label{lem:loc-freeness}
    For a semilinear $G_w$-representation $W_w$ over $\Ocal_w$ the following are equivalent.
    \begin{enumerate}
        \item\label{lem:loc-freeness:free} $W_w$ is free as an $\Ocal_v[G_w]$-module;
        \item\label{lem:loc-freeness:proj} $W_w$ is projective as an $\Ocal_v[G_w]$-module;
        \item\label{lem:loc-freeness:ct} $W_w$ is cohomologically trivial for $G_w$ (i.e., the Tate cohomology $\widehat\Hrm^i(G,W_w)$ is trivial for each degree $i$).
    \end{enumerate}
\end{lem}
\begin{proof}
    Note that $W_w\otimes_{\Ocal_w}L_w$ is free as an $K_v[G_w]$-module by standard Galois descent, so the equivalence of (\ref{lem:loc-freeness:free}) and (\ref{lem:loc-freeness:proj}) follows from \cite[Corollary~6.4]{Swan:Induced}. The equivalence between (\ref{lem:loc-freeness:proj}) and (\ref{lem:loc-freeness:ct}) is standard as $W_w$ is projective as $\Ocal_w$-module.\footnote{The proof of the $\ZZ G$-projectivity criterion \cite[Chap~IX, \S5, Theorem~7]{Serre:LocalFields} can be repeated to show the $RG$-projectivity criterion for any Dedekind domain $R$. This is also implicitly proved in \cite[Proposition~4.1]{Chinburg:dR-tame}.}
\end{proof}

Let us now recall the following ``higher-rank version'' of the local integral normal basis theorem in the tame setting, which is essentially due to Chinburg.
\begin{prop}\label{prop:Noether}
    Let $W_w$ be a semilinear $G_w$-representation over $\Ocal_w$. Then $W_w$ is free as an $\Ocal_v[G_w]$-module if and only if it is cohomologically trivial for $I_w$. In particular, if $L_w/K_v$ is tame then any semilinear $G_w$-representation over $\Ocal_w$ is free as an $\Ocal_v[G_w]$-module.
\end{prop}
\begin{proof}
    The case when $L_w/K_v$ is \emph{unramified} is standard; \emph{cf.} \cite[\S2, Lemma~1]{Nakajima:GalMod-Var}. To handle the general case, it suffices, by Lemma~\ref{lem:loc-freeness}, to show that $W_w$ is cohomologically trivial for $G_w$ if and only if it is cohomologically trivial for $I_w$. And since by the unramified case $W_w^{I_w}$ is cohomologically trivial for $G_w/I_w$ (being a semilinear $G_w/I_w$-lattice over $\Ocal_w^{I_w}$), the desired claim follows from the inflation-restriction sequence for the Tate cohomology. The claim for the tame case now follows since any $\Ocal_v[I_w]$-module cohomologically trivial for $I_w$ when $|I_w|$ is prime to $p$.
\end{proof}

We next classify the inertia action on any semilinear $G_w$-representation over $\Ocal_w$ in the tame case. For this let us first describe all the mod~$p$ absolutely irreducible representations of $I_w$, without assuming tameness.
\begin{defn}\label{def:theta}
    Let $\theta_w\colon I_w\to k_w^\times$ be the character corresponding to the natural $I_w$-action on $\mfr_w/\mfr_w^2$; in other words, choosing a uniformiser $\varpi_w\in\mfr_w$ we have
\[\theta_w(g) \equiv g\varpi_w/\varpi_w \bmod{m_w}\quad\forall g\in I_w.\]
The $I_w$-action on $\mfr_w^n/\mfr_w^{n+1}$ is given by $\theta_w^n$.
\end{defn}

\begin{rmk}\label{rmk:simple}
    Note that $\theta_w$ induces an inclusion $I_w/P_w\hookrightarrow k_w^\times$, so the order of $\theta_w$ is $|I_w/P_w|$. Furthermore, since $P_w$ acts trivially on any simple $k_w[I_w]$-module,  any simple $k_w[I_w]$-module is isomorphic to exactly one of $\mfr_w^n/\mfr_w^{n+1}$ for $n\in\ZZ / (|I_w/P_w|)$.
\end{rmk}

\begin{lem}\label{lem:tame-inertial-action}
    Suppose that $L_w/K_v$ is tame, and 
    let $W_w$ be a rank-$d$ semilinear $G_w$-representation over $\Ocal_w$.    
    Then there exist integers $n_{w,1},\cdots,n_{w,d}\in\{0,\cdots,|I_w|-1\}$, unique up to ordering, such that we have a $k_w[I_w]$-module isomorphism
    \[
    W_w \otimes_{\Ocal_w} k_w \cong \bigoplus_{i=1}^{d} (\mfr_w^{-n_{w,i}}/\mfr_w^{-n_{w,i}+1}).
    \]
    Furthermore, the above isomorphism can be lifted to an isomorphism 
    \[
    W_w \cong  \bigoplus_{i=1}^{d} \mfr_w^{-n_{w,i}}.
    \]
    of semilinear $I_w$-representations over $\Ocal_w$.
\end{lem}

\begin{proof}
    By tameness, the group ring $k_w[I_w]$ is semi-simple and its simple modules are described in Remark~\ref{rmk:simple}. Therefore, one can find a $k_w$-basis $\bar e_1,\cdots,\bar e_d$ of $W_w\otimes_{\Ocal_w}k_w$ such that $I_w$ acts on $\bar e_i$ via $\theta_w^{-n_{w,i}}$ for $0\leqslant n_{w,i}<|I_w|$. We choose a lift $e_i\in W_w$ of $\bar e_i$ for each $i$, and set 
    \[
    e_i'\coloneqq \frac{1}{|I_w|}\sum_{g\in I_w}\theta_w^{n_{w,i}}(g)\cdot(g e_i).
    \]
    Then each $e_i'\in W_w$ lifts $\bar e_i$ and satisfies $ge_i' = \theta_w^{-n_{w,i}}(g)e_i'$ for any $g\in I_w$. Therefore, $W_w$ can be written as a direct sum of $I_w$-stable $\Ocal_w$-submodules $\Ocal_w e_i'$, which is isomorphic to $\mfr_w^{-n_{w,i}}$ as a semilinear $I_w$ representation over $\Ocal_w$.
\end{proof}

If $L_w/K_v$ is wildly ramified, then the Galois module structure of a semilinear $G_w$-representation over $\Ocal_w$ could be quite complicated in general. Instead, we focus on the case where $W_w$ is of rank~$1$. To proceed, we need the following definition.

\begin{defn}\label{def:weak-ram}
We say that $L_w/K_v$ is \emph{weakly ramified} if the second lower-index ramification subgroup $I_{w,2}$ is trivial. 
\end{defn}
Recall that for any non-negative integer $s$, we set
\[I_{w,s}\coloneqq \{g\in I_w \mid g\varpi_w\equiv \varpi_w\bmod{\mfr_w^{s+1}}\}\]
for some (or equivalently, any) uniformiser $\varpi_w\in\mfr_w$. Note that $I_w = I_{w,0}$ and $P_w = I_{w,1}$.

Clearly, unramified or tamely ramified extensions are weakly ramified. Much less obvious examples of weakly ramified extensions are those obtained by the completion of a finite Galois cover $\pi\colon X_L\to X$ of \emph{ordinary}\footnote{A curve over a field of characteristic $p>0$ is defined to be \emph{ordinary} if the genus and the $p$-rank coincides.} curves over a perfect field of characteristic $p>0$; \emph{cf.} \cite[Theorem~2(i)]{Nakajima:p-rank}.

Being weakly ramified imposes a strong condition on the inertia group $I_w$ as follows.
\begin{lem}\label{lem:Cw}
    For any finite Galois extension $L_w/K_v$ we have 
    \[I_w = P_w\rtimes C_w\]
    where $P_w$ is a $p$-group and $C_w$ is a cyclic group of prime-to-$p$ order. Furthermore, if $L_w/K_v$ is weakly ramified, then $P_w$ is an elementary $p$-group and the conjugation action of $C_w$  on $P_w\setminus\{1\}$ is faithful.
\end{lem}
\begin{proof}
    The properties can be deduced from Proposition~9 and the corollaries of Proposition~7 in \cite[Chap~IV, \S2]{Serre:LocalFields}.
\end{proof}

\begin{rmk}\label{rmk:Cw-choices}
    The choice of the lift $C_w$ of $I_w/P_w$ is far from canonical if $P_w$ is a proper non-trivial subgroup of $I_w$, but different choices of $C_w$ are conjugate to each other. Indeed, by direct computation we have
\begin{equation}
    C_w\cap (g C_w g^{-1}) = \{1\} \quad \forall g\in P_w\setminus\{1\};
\end{equation}
\emph{cf.} the proof of Lemma~4.2 in \cite{Koeck:EquivRR}. By simple counting we obtain $I_w\setminus P_w = \bigsqcup_{g\in P_w}(gC_wg^{-1}\setminus\{1\})$, so in particular any lift of $I_w/P_w$ in $I_w$ is of the form $g C_w g^{-1}$ for a unique $g\in P_w$.
\end{rmk}

\begin{lem}\label{lem:proj-cover}
    For any finite Galois extension $L_w/K_v$, any indecomposable projective $k_w[I_w]$-module is isomorphic to exactly one of 
    \[M_w(j)\coloneqq \Ind_{C_w}^{I_w}\left((\mfr_w^j/\mfr_w^{j+1})|_{C_w}\right) \quad \text{for }j\in \ZZ/(|C_w|).\] 
    Furthermore, $M_w(j)$ is a $k_w[I_w]$-projective cover of $\mfr_w^j/\mfr_w^{j+1}$, so it does not depend on the choice of $C_w$ up to isomorphism.

    If $L_w/K_v$ is weakly ramified, then we have 
    \[M_w(j)|_{C_w}\cong (\mfr_w^j/\mfr_w^{j+1})\oplus k_w[C_w]^{\oplus\frac{|P_w|-1}{|C_w|}}.\]
\end{lem}
\begin{proof}
    Note that the radical $\rad(k_w[I_w])$ of $k_w[I_w]$ is generated by the augmentation ideal  of $k_w[P_w]$. Furthermore, we have an $k_w[I_w]$-module isomorphism
    \[M_w(j)/\rad(k_w[I_w]) = M_{w}(j)_{P_w} \cong \mfr_w^j/\mfr_w^{j+1}.\]
    Then essentially by the Nakayama lemma, $M_w(j)$ is a $k_w[I_w]$-projective cover of $\mfr_w^j/\mfr_w^{j+1}$; \emph{cf.} \cite[Theorem~(6.23)]{CurtisReiner:Methods1}. Indecomposability of $M_w(j)$ follows from being a projective cover of a simple $k_w[I_w]$-module. Since $k_w[I_w] \cong\bigoplus_j M_w(j)$, any non-zero projective $k_w[I_w]$-module contains a copy of some $M_w(j)$. Finally, the last claim is proved in \cite[Lemma~4.2]{Koeck:EquivRR}.
\end{proof}

Let us now recall the \emph{local integral normal basis theorem} due to K\"ock:
\begin{thm}[K\"ock {\cite[Theorem~1.1]{Koeck:EquivRR}}]\label{th:int-normal-basis}
    The local fractional ideal $\mfr_w^{-n}$ for $n\in \ZZ$ is free of rank~$1$ as an $\Ocal_v[G_w]$-module if and only if $L_w/K_v$ is weakly ramified and $n\equiv-1\bmod{|P_w|}$.
\end{thm}
If $L_w/K_v$ is tame (i.e., we have $|P_w|=1$) then the theorem asserts that any fractional ideal $\mfr_w^{-n}$ is projective as an $\Ocal_v[G_w]$-module, which is consistent with Proposition~\ref{prop:Noether}. 

\begin{rmk}\label{rmk:higher-rk-normal-basis}
    The higher-rank generalisation of Theorem~\ref{th:int-normal-basis} (or rather, the wildly ramified analogue of Lemma~\ref{lem:tame-inertial-action}) could be quite complicated. To illustrate, let  $L_w/K_v$ be any finite Galois extension (not necessarily weakly ramified) and choose a semilinear $G_w$-representation $W'_w$ over $\Ocal_w$. (We do \emph{not} require $W'_w$ to be projective as $\Ocal_v[G_w]$-module.) Then $\Ocal_w[G]\otimes_{\Ocal_w}W_w'$, with $G$ acting diagonally, is a semilinear $G_w$-representation over $\Ocal_w$ that is free as a $\Ocal_v[G_w]$-module; indeed, the following $\Ocal_w$-linear isomorphism 
    \begin{align*}
        \Ind_1^{G_w}W_w' \cong \Ocal_w[G]\otimes_{\Ocal_w}W_w' &\riso \Ocal_w[G]\otimes_{\Ocal_w}W_w'\\
        (\sum_{g\in G_w}a_gg)\otimes x &\mapsto  \sum_{g\in G_w}(a_gg)\otimes gx
    \end{align*}
    is $G_w$-equivariant.
\end{rmk}

\section{Equivariant Riemann--Roch for weakly ramified covering}\label{sec:EquivRR}
Let $k$ be a perfect field of characteristic $p>0$. Let $X$ be a smooth projective geometrically connected curve over $k$, with its function field denoted by $K$. For any finite extension $L$ of $K$, we write $X_{L}$ denote the normalisation of $X$ in $\Spec L$ equipped with the covering map $\pi\colon X_{L}\to X$. (We do \emph{not} require $X_L$ to be geometrically connected over $k$.) Let $|X|$ and $|X_L|$ respectively denote the set of closed points of $X$ and $X_L$. 

From now on, suppose that $L/K$ is \emph{Galois} with group $G$, so $\pi$ is a $G$-covering. Choosing $v\in |X|$ and $w\in \pi^{-1}(v)$, we obtain the Galois extension $L_w/K_v$ via completion. We employ the same notation as in \S\ref{sec:local}.

%

We next study the ``equivariant Euler characteristic'' of the cohomology of $G$-equivariant vector bundles on $X_L$; i.e., a locally free $\Ocal_{X_L}$-module with semilinear $G$-action. Given such $\Ecal$, we may represent $\RGamma(X_L,\Ecal)$ as a complex of $kG$-modules (eg, by choosing a $G$-stable \v{C}ech covering). \emph{If} we have $\RGamma(X_L,\Ecal) \in \Dperf(kG)$, then we write
\begin{equation}
    \chi_{kG}(\Ecal)\coloneqq \chi_{kG}(\RGamma(X_L,\Ecal))\in \Krm_0(kG),
\end{equation}
where the right hand side is defined in \eqref{eq:Euler-K0}.

Given a finite Galois cover $\pi\colon X_L\to X$, we let $Z_L^{\ram}$ (respectively, $Z_L^{\wild}$) denote the locus in $X_L$ where $\pi$ is ramified (respectively, wildly ramified).

We are now ready to state the \emph{equivaiant Riemann--Roch theorem}, which  generalises the results of Nakajima \cite{Nakajima:GalMod-tame},  K\"ock \cite{Koeck:EquivRR} and Fischbacher-Weitz and K\"ock \cite{FischbacherWeitz-Koeck:EquivRR}.
\begin{thm}\label{th:EquivRR}
\setcounter{equation}{\value{equation}-1}
\begin{subequations}
    Let $\pi\colon X_L\to X$ be a $G$-cover that is weakly ramified everywhere. Let $\Ecal$  be a $G$-equivariant vector bundle on $X_L$, and suppose that for any $w\in Z_L^{\wild}$ we have 
\begin{equation}\label{eq:Ew}
    \widehat\Ecal_w \cong \bigoplus_{i=1}^{\rk\Ecal} \mfr_w^{-n_{w,i}}
    \quad \text{where }n_{w,i}\equiv -1 \bmod{|P_w|}\ \text{for any } i
\end{equation}
as a semilinear $I_w$-representation over $\Ocal_w$. 
Then the following properties hold.
\begin{enumerate}
    \item\label{th:EquivRR:perf} We can represent $\RGamma(X_L,\Ecal)$ by a complex of finitely generated projective $kG$-modules in degrees $[0,1]$. In particular, we have $\RGamma(X_L,\Ecal)\in \Dperf(kG)$. 
    \item\label{th:EquivRR:formula} For any $w\in Z_L^{\ram}$ and $i\in\{1,\cdots,\rk\Ecal\}$, define $l_{w,i}$ to be the unique integer satisfying 
    \begin{equation}\label{eq:lw}
        l_{w,i}\equiv \frac{1+n_{w,i}}{|P_w|}-1\bmod{|I_w/P_w|} \quad\text{and} \quad  0\leqslant l_{w,i}< |I_w/P_w|,
    \end{equation}
    where $n_{w,i}$'s are as in \eqref{eq:Ew} for $w\in Z_L^{\wild}$, and as in  Lemma~\ref{lem:tame-inertial-action} if $w\notin Z_L^{\wild}$.\footnote{Note that for $w\in Z_L^{\ram}\setminus Z_L^{\wild}$, we have $l_{w,i} = n_{w,i}$ as  we have $|P_w|=1$ and $0\leqslant n_{w,i}< |I_w|$.} 
%
 
    Then we have the following equality in $\Krm_0(kG)\otimes\QQ$
    \begin{equation}\label{eq:EquivRR}
        \chi_{kG}(\Ecal) = -(\rk\Ecal)[N(\pi)] + [W_G(\Ecal)]+ \Ind_1^G\big(\chi_k(\Ecal^G)\big),
    \end{equation}
    where $\Ecal^G$ is the $G$-invariants of $\Ecal$, which is a vector bundle on $X$, and    
    \begin{align*}
        [N(\pi)]  & \coloneqq
          \frac{1}{|G|}\sum_{w\in Z_L^{\ram}}|P_w|\sum_{j=1}^{|I_w/P_w|-1} j\cdot \left[\Ind_{I_w}^G \big(M_w(j)\big)\right]\ \text{and}\\
        [W_G(\Ecal)]& \coloneqq \sum_{w\in Z_L^{\ram}}\frac{1}{[G:I_w]}\sum_{i=1}^{\rk\Ecal}\sum_{j=1}^{l_{w,i}}\left[\Ind_{I_w}^G \big(M_w(-j)\big)\right]. 
    \end{align*}
    Here, $M_w(j)$ is defined in Lemma~\ref{lem:proj-cover}.
\end{enumerate}
\end{subequations}
\end{thm}

Note that formula \eqref{eq:EquivRR} is generalises the rank-$1$ case stated in \cite[\S3, Theorem~12]{FischbacherWeitz-Koeck:EquivRR}, which is built upon \cite[Theorem~4.5]{Koeck:EquivRR}.
When $k=\bar k$ and $\pi$ is tame, then S.~Nakajima \cite[Theorem~2]{Nakajima:GalMod-tame} obtained \eqref{eq:EquivRR} modulo $\Ind_1^G\big(\chi_k(\Ecal^G)\big)$.

Before we give a proof, let us make a few remarks.
\begin{rmk}
    As the notation suggests, $[N(\pi)]$ and $[W_G(\Ecal)]$ respectively come from finitely generated projective $kG$-modules $N(\pi)$ and $W_G(\Ecal)$. Indeed, this is a byproduct of the rank-$1$ case of the formula; \emph{cf.} Theorem~11 and Theorem~12(a) in \cite[\S3]{FischbacherWeitz-Koeck:EquivRR}. In particular, formula~\eqref{eq:EquivRR} holds in $\Krm_0(kG)$ since $\Krm_0(kG)$ is torsion-free (being a free abelian group). In the intended application, we only need formula~\eqref{eq:EquivRR} in $\Krm_0(kG)\otimes\QQ$.
\end{rmk}

\begin{rmk}\label{rmk:EquivRR-p-group}
    In the setting of Theorem~\ref{th:EquivRR}, if we have $I_w = P_w$ for all $w\in Z^{\ram}_L$ (e.g., if $G$ is a $p$-group), then clearly both $[N(\pi)]$ and $[W_G(\Ecal)]$ are trivial so formula~\eqref{eq:EquivRR} reduces to
    \[\chi_{kG}(\Ecal) = \Ind_1^G\big(\chi_k(\Ecal^G)\big).\] 
\end{rmk}

\begin{rmk}
    Applying the Hirzeburch--Riemann--Roch theorem to compute $\chi_k(\Ecal^G)$ (\emph{cf.} Theorem~(4.11) and Exa~4.1.1 in \cite[Appendix~A]{hartshorne}), we obtain \[\Ind_1^G\big(\chi_k(\Ecal^G)\big) = \left((\rk\Ecal)\cdot(1-\gen_K) + \deg(\Ecal^G)\right)[kG],\] where $\gen_K$ is the genus of $X$. (Note also that $\rk(\Ecal^G) = \rk\Ecal$.)
\end{rmk}

\begin{exa}
\setcounter{equation}{\value{equation}-1}
\begin{subequations}
    The local assumption \eqref{eq:Ew} may look artificial, but it is a common generalisation of the tame case (\emph{cf.} Lemma~\ref{lem:tame-inertial-action}) and the following class of examples.
    Let $\Fcal$ be a rank-$d$ vector bundle on $X$, and let $D_L=\sum_{w\in |X_L|}n_w w$ be a $G$-equivariant divisor of $X_L$. Then $\Ecal\coloneqq(\pi^\ast\Fcal)(D_L)$ is a $G$-equivariant vector bundle and we clearly have a $G_w$-equivariant $\Ocal_w$-linear isomorphism 
\begin{equation}\label{eq:nw-div-twist}
    \widehat\Ecal_w \cong (\mfr_w^{-n_w})^{\oplus d},
\end{equation}
    as we have a natural $G_w$-equivariant isomorphism $(\pi^\ast\Fcal)\widehat{_w} \cong \widehat\Fcal_{\pi(w)}\otimes_{\Ocal_{\pi(w)}}\Ocal_w$. Furthermore, 
    we have $\Ecal^G \cong \Fcal(D_K)$ where $D_K$ satisfies $\big(\Ocal_{X_L}(D_L)\big)^G = \Ocal_X(D_K)$. More explicitly, we can apply \cite[Lemma~1.4(a)]{Koeck:EquivRR} to obtain $D_K = \sum_{v\in |X|} n_v v$ where for each $v\in|X|$ we set $n_v = -1+\lceil\frac{n_{\tilde v}+1}{|I_{\tilde v}|} \rceil$ for some (or equivalently, any) $\tilde v\in \pi^{-1}(v)$.

    Now, assume that $\pi$ is weakly ramified everywhere. Then Theorem~\ref{th:EquivRR} can be applied to $\Ecal=(\pi^\ast\Fcal)(D_L)$ provided that the coefficient $n_w$ at each $w\in Z_L^{\wild}$ satisfies $n_w\equiv -1\bmod{|P_w|}$, in which case we have $n_{w,i}=n_w$ for all $i$. 

    Let us further specialise to the case when $\Ecal = (\pi^\ast\Fcal)(-Z_L)$ for some reduced $G$-stable closed subscheme $Z_L\subset X_L$ containing $Z_L^{\ram}$. In that case, we have $l_{w,i} = |I_w/P_w|-1$ for any $w\in Z_L^{\ram}$ and $i$, and $\Ecal^G = \Fcal(-Z)$ where $Z\subset X$ is the reduced image of $Z_L$. Therefore, formula~\eqref{eq:EquivRR} reduces to the following
    \begin{align}\label{eq:EquivRR-intended-case}
            \chi_{kG}\big((\pi^\ast\Fcal)(-Z_L)\big) = & \frac{\rk\Fcal}{|G|}\sum_{w\in Z_L^{\ram}}\sum_{j=1}^{|I_w/P_w|-1}j|P_w|\cdot\left[\Ind_{I_w}^G \big(M_w(-j)\big)\right] \\ 
            &+ \left((\rk\Fcal)\cdot(1-\gen_K-\deg(Z)) + \deg(\Fcal)\right)[kG].\notag
    \end{align}
\end{subequations}
 \end{exa}

%
%

For the rest of the section we give a proof of Theorem~\ref{th:EquivRR}. 

\begin{proof}[\textbf{\emph{Proof of Claim~(\ref{th:EquivRR:perf}) of Theorem~\ref{th:EquivRR}}}]
The proof is essentially contained in the proof of Theorem~2.1(a) in \cite{Koeck:EquivRR}, which is the rank-$1$ case of our statement. Indeed, by K\"ock's theorem (Theorem~\ref{th:int-normal-basis}) our assumption \ref{eq:Ew} implies that $\widehat\Ecal_w$ is cohomologically trivial for $I_w$ for any $w\in |X_L|$, so by Proposition~\ref{prop:Noether} it follows that the following completed stalk
\[
    (\pi_\ast\Ecal)\widehat{_v} \cong \bigoplus_{w\in\pi^{-1}(v)} \widehat\Ecal_w \cong \Ind_{G_{\tilde v}}^G \widehat \Ecal_{\tilde v},
\]
is $\Ocal_v[G]$-free at any $v\in |X|$. (Here, $\tilde v\in\pi^{-1}(v)$ is any point above $v$.) By the standard result (\emph{cf.} \cite[Corollary~(76.9)]{CurtisReiner:AMSbook}), the Zariski stalk $(\pi_\ast\Ecal)_v$ is $\Ocal_{X,v}[G]$-free, so we may apply \cite[Theorem~1.1]{Chinburg:dR-tame} to obtain $kG$-perfectness of  $\RGamma(X_L,\Ecal)\cong \RGamma(X,\pi_\ast\Ecal)$. Finally, $\RGamma(X_L,\Ecal)$ can be represented by a two-term perfect $kG$-complex thanks to the cohomology vanishing outside degrees $[0,1]$.
\end{proof}

\subsection*{Digressions to Brauer characters}
Even if $|G|$ is not invertible in $k$, there is a version of character theory for finitely generated $\bar kG$-modules; namely, \emph{Brauer characters}. (Here, $\bar k$ denotes the algebraic closure of $k$.) We will recall the minimal background needed in the proof, and for further details we refer to \cite[\S18]{CurtisReiner:Methods1} or \cite[Ch~18]{Serre:LinRepFinGp}.

We choose a complete discrete valuation field $E$ of characteristic $0$ with residue field $\bar k$. For any subset $S \subset G$ stable under conjugation, set
\begin{equation}
    \Cl_E(S)\coloneqq \{\phi\colon S\to E \mid \phi \text{ is invariant under conjugation}\}.
\end{equation}

Let $G^{\preg}$ denote the set of elements of $G$ with prime-to-$p$ order, which is stable under conjugation. Then for any finitely generated $\bar kG$-module $M$, let $\Bch_M\in\Cl_E(G^{\preg})$ denote the \emph{Brauer character} in the sense of \cite[Def~(17.4)]{CurtisReiner:Methods1}. It is easy to check that $\Bch_M(1)=\dim_{\bar k}(M)$ and $\Bch_{\bar kG}(g) = 0$ for any $g\ne 1$. If $M$ admits a lift to an $EG$-module $\widetilde M_E$ (e.g., if $M$ is projective), then we have $\Bch_M = \ch_{\widetilde M_E}|_{G^{\preg}}$, which we take as a working definition; \emph{cf.} \cite[Proposition~(17.5)(iv)]{CurtisReiner:Methods1}.\footnote{One may refer to \cite[\S18.1]{Serre:LinRepFinGp} instead, where $\Bch_M$ is called the \emph{modular character}.} 

Let us now recall a few basic properties. The construction of Brauer characters naturally extends to the Grothendieck group $\Gbf_0(\bar kG)$ of finitely generated $\bar kG$-modules, inducing an isomorphism $\Gbf_0(\bar kG)\otimes E\to \Cl_E(G^{\preg})$; \emph{cf.} \cite[Theorem~(17.9)]{CurtisReiner:Methods1}. There is a natural homomorphism $c\colon \Krm_0(\bar kG)\to\Gbf_0(\bar kG)$, which turns out to be injective with finite cokernel; \emph{cf.} \cite[Theorem~(21.22)]{CurtisReiner:Methods1}. By abuse of notation, for $[M]\in\Krm_0(\bar kG)$ we let $\Bch_{[M]}$ denote $\Bch_{c([M])}$. As $\Krm_0(\bar kG)$ and $\Gbf_0(\bar kG)$ are free abelian groups, we obtain the following:
\begin{prop}\label{prop:inj-cartan}
    The homomorphism $\Bch_{(-)}\colon\Krm_0(\bar kG)\to \Cl_E(G^{\preg})$ is injective.
\end{prop}

We record the following lemma, which should be well known.
\begin{lem}\label{lem:st-eq}
    For finitely generated \emph{projective} $\bar kG$-modules $M$ and $M'$, we have $\Bch_M(g)=\Bch_{M'}(g)$ for any $g\ne 1$ if and only if $[M]-[M']$ is an integer multiple of $[\bar kG]$ in $\Krm_0(\bar kG)$.
\end{lem}

\begin{proof}
    Suppose that an element $[N]\in\Krm_0(\bar kG)$ satisfies $\Bch_{[N]}(g) = 0$ for any $g\ne 0$. We first claim that $|G|$ divides $\Bch_{[N]}(1)=\dim_{\bar k}([N])$. The lemma easily follows from this claim for $[N] = [M]-[M']$ via Proposition~\ref{prop:inj-cartan}.

    For each prime divisor $\ell$ of $|G|$, choose a Sylow $\ell$-subgroup $G_\ell$ of $G$. Note that the restriction $[N]|_{G_\ell}$ defines an element in $\Krm_0(\bar k[G_\ell])$ as the restriction preserves projectivity, and we have $\Bch_{[N]|_{G_\ell}} = (\Bch_{[N]})|_{G_\ell}$.
    
    If $\ell \ne p$ then we have $G_\ell\subset G^{\preg}$ and $\Bch_{[N]}|_{G_\ell} = \ch_{[\widetilde N_E]}|_{G_\ell}$ where $[\widetilde N_E]\in \Krm_0(EG)$ is the lift of $[N]$. By the standard character theory in characteristic~$0$, we have $\dim_E([\widetilde N_E]^{G_\ell})  = \dim_{\bar k}([N])/|G_\ell|$, which is an integer. If $\ell = p$ then $\Krm_0(\bar k[G_p])$ is the free abelian group generated by $[\bar k[G_p]]$ as $\bar k[G_p]$ is a local ring. In particular, $|G_p|$ divides $\dim_{\bar k}([N])$. This shows that $|G|$ divides $\dim_{\bar k}([N])$.
\end{proof}

\begin{rmk}\label{rmk:st-eq}
    Lemma~\ref{lem:st-eq} cannot be extended to $\Gbf_0(\bar kG)$ in general. For example, if $G$ admits a \emph{normal} Sylow $p$-subgroup $P\ne \{1\}$, then we have $\Bch_{\bar k[G/P]}(g) = 0$ for any $g\in G^{\preg}\setminus\{1\}$. Then the Brauer characters of $\bar k[G/P]^{\oplus|P|}$ and $\bar kG$ coincide, but $\bar k[G/P]$ is not even projective as a $\bar kG$-module.
\end{rmk}

\subsection*{Proof of Claim~(\ref{th:EquivRR:formula}) of Theorem~\ref{th:EquivRR}: The case where $k$ is algebraically closed}
In the setting of Theorem~\ref{th:EquivRR}, assume that $k=\bar k$ so we have $G_w = I_w$. Choose $H =\langle g \rangle$ for some $g\in G^{\preg}\setminus\{1\}$, a cyclic subgroup of prime-to-$p$ order.  Let $\pi_H\colon X_L\to X_{L^H}$ denote the natural projection, which is a everywhere tame $H$-covering. 

The following lemma generalises the tame case \cite[p~120]{Nakajima:GalMod-tame}, and the main idea of proof can be read off from the proof of Theorem~3.1 and Theorem~4.3 of \cite{Koeck:EquivRR}.

\begin{lem}\label{lem:Nakajima}
    In the above setting, we have the following elements in $\Krm_0(kH)$
    \[[N(\pi)]|_H - [N(\pi_H)] \ \text{and} \ [W_G(\Ecal)]|_H - [W_H(\Ecal)] \]
    are integer multiples of $[kH]$.
\end{lem}
\begin{proof}
    For any $w\in Z_L^{\ram}$, let $H_w\coloneqq I_w\cap H$.  Then by the Mackey formula \cite[Theorem~(44.2)]{CurtisReiner:AMSbook}\footnote{When all the modules involved are projective (as in our setting), one may alternatively obtain the mod~$p$ Mackey formula by lifting to characteristic~$0$.}, we have
\begin{align*}
    [\Ind_{I_w}^G M_w(j)]|_H &= \sum_{s\in H\backslash G/I_w} \left[\Ind^H_{H_{sw}} \big(M_{sw}(j)|_{H_{sw}}\big) \right]
    \\&= \frac{1}{|H|}\sum_{s\in G/I_w} |H_{sw}| \cdot \left[\Ind^H_{H_{sw}} \big(M_{sw}(j)|_{H_{sw}}\big)\right].
\end{align*}
Choosing $C_{sw}$ to contain $H_{sw}$, which is possible by Remark~\ref{rmk:Cw-choices}, it follows from Lemma~\ref{lem:proj-cover} that 
\begin{equation}\label{eq:proj-cover}
    M_{sw}(j)|_{H_{sw}} \cong (\theta_{sw}^H)^j \oplus (k[H_{sw}])^{\oplus\frac{|P_{sw}|-1}{|C_{sw}|}[C_{sw}:H_{sw}]},
\end{equation}
where $\theta^H_{sw} = \theta_{sw}|_{H_{sw}}$.

Let $n_{w,i}^H = l_{w,i}^H$ be the integers defined in Lemma~\ref{lem:tame-inertial-action} for $\Ecal$ viewed as an $H$-equivariant vector bundle. Then clearly we have 
\begin{equation}\label{eq:lw-H}
    l^H_{w,i}\equiv l_{w,i}\equiv n_{w,i} \bmod{|H_w|} 
\end{equation}
possibly up to reordering $l^H_{w,i}$'s. Therefore we have
\begin{align*}
    [W_G(\Ecal)]|_H & = \sum_{w\in Z_L^{\ram}}\frac{1}{[G:I_w]}\sum_{i=1}^{\rk\Ecal}\sum_{j=1}^{l_{w,i}}\left[\Ind_{I_w}^G \big(M_w(-j)\big)\right]|_H \\
    & =  \sum_{w\in Z_L^{\ram}}\frac{1}{[H:H_w]}\sum_{i=1}^{\rk\Ecal}\sum_{j=1}^{l_{w,i}}\left[\Ind_{H_w}^H \big(M_w(-j)|_{H_w}\big)\right] \\
    &\equiv \sum_{w\in Z_L^{\ram}}\frac{1}{[H:H_w]}\sum_{i=1}^{\rk\Ecal}\sum_{j=1}^{l^H_{w,i}}\left[\Ind_{H_w}^H \big((\theta_w^H)^{-j}\big)\right] = W_H(\Ecal) \bmod{[kH]},
\end{align*}
where the last congruence uses \eqref{eq:proj-cover}, \eqref{eq:lw-H} and $k[H_w]\cong\bigoplus_{j\in \ZZ/(|H_w|)}(\theta_w^H)^j$.

The computation of $[N(\pi)]|_H$ is quite similar except that we use
\[
\sum_{\substack{0\leqslant j <|I_w/P_w|\\ j\equiv j_0\bmod{|H_w|}}}j|P_w|\cdot \left[\Ind_{I_w}^G(M_w(j))\right]|_H \equiv j_0|I_w| \cdot \left[\Ind_{H_w}^H \big((\theta^H_w)^{j_0}\big)\right]\bmod{[kH]}
\]
for any $0\leqslant j_0<|H_w|$. 
\end{proof}

\begin{cor}\label{cor:EquivRR-alg-clos}
    Theorem~\ref{th:EquivRR}(\ref{th:EquivRR:formula}) holds if $k = \bar k$.
\end{cor}
\begin{proof}
     We first show that $\chi_{kG}(\Ecal) \equiv -\rk(\Ecal)[N(\pi)]+[W_G(\Ecal)] + c[kG]$ for some $c\in\ZZ$. 
     By Lemma~\ref{lem:st-eq}, we can proceed by comparing the values of Brauer characters at each $g\in G^{\preg}\setminus\{1\}$. Now applying Lemma~\ref{lem:Nakajima} to $H\coloneqq \langle g\rangle$ for any $g\in G^{\preg}\setminus\{1\}$, it suffices to prove the claim for $\pi_H\colon X_L\to X_{L^H}$ as we have $\chi_{kG}(\Ecal)|_{H} = \chi_{kH}(\Ecal)$. However, the claim for $\pi_H$ is already obtained in \cite[Theorem~2]{Nakajima:GalMod-tame}.

     Note that $\left(\chi_{kG}(\Ecal)\right)^G = \chi_k(\Ecal^G)$ in $\Krm_0(k)$. We next claim that $[N(\pi)]^G = [W_G(\Ecal)]^G = 0$; indeed,  we have by the Frobenius reciprocity that
     \[
    \left( \Ind_{I_w}^G(M_w(j))\right)^G \cong \left( \Ind_{C_w}^G(\theta_w^j|_{C_w})\right)^G \cong \Hom_{C_w}(\triv,\theta_w^j|_{C_w}),
     \]
     which is zero if and only if $j\not\equiv 0 \bmod{|I_w/P_w|}$.
     This shows that $c[kG] = [\Ind_1^G(\chi_k(\Ecal^G))]$, which implies formula \eqref{eq:EquivRR}.
\end{proof}
We are now ready to conclude the proof.
\begin{proof}[\textbf{\emph{Proof of Claim~(\ref{th:EquivRR:formula}) of Theorem~\ref{th:EquivRR}: The general case}}]
We now allow $k$ to be any perfect field. By injectivity of the scalar extension map $\Krm_0(kG)\to\Krm_0(\bar kG)$, we may verify formula \eqref{eq:EquivRR} in $\Krm_0(\bar kG)$.

Let $\bar \pi\colon \overline X_L\to\overline X$ denote the base change of $\pi$ to $\bar k$, and let $\overline \Ecal$ denote the pull back of $\Ecal$ to $\overline X_L$. We choose a connected component $\overline X_L^\circ$ of $\overline X_L$, with  $\bar\pi^\circ\coloneqq \bar\pi|_{\overline X_L^\circ}$. Let $G^\circ$ denote the stabiliser of $\overline X_L^\circ$. By Corollary~\ref{cor:EquivRR-alg-clos}, we have the following formula in $\Krm_0(\bar k[G^\circ])$
\[\chi_{\bar k[G^\circ]}\big(\overline\Ecal|_{\overline X_L^\circ}\big) = -(\rk\Ecal)\cdot[N(\bar\pi^\circ)] + [W_{G^\circ}\big(\overline\Ecal|_{\overline X_L^\circ}\big)] + \Ind_1^{G^\circ}\big(\chi_{\bar k}(\overline\Ecal^G)\big).\]
(Note that $\overline\Ecal^G = (\overline\Ecal|_{\overline X_L^\circ})^{G^\circ}$.) We next claim that $\Ind^G_{G^\circ}$ of the above formula coincides with the scalar extension to $\bar k$ of formula \eqref{eq:EquivRR} for $\Ecal$.

Firstly, the following equality holds in $\Krm_0(\bar kG)$
\[\chi_{kG}\big(\RGamma(X_L,\Ecal)\big)\otimes _k \bar k = \chi_{\bar kG}\big(\RGamma(\overline X_L,\overline \Ecal)\big) = \Ind_{G^\circ}^G\left(\chi_{\bar kG^\circ}\big(\RGamma(\overline X_L^\circ,\overline\Ecal|_{\overline X_L^\circ})\big)\right).\]
Similarly, we have $\Ind^G_{G^\circ}\Ind_1^{G^\circ}\big(\chi_{\bar k}(\overline\Ecal^G)\big) = \Ind_1^G\big(\chi_k(\Ecal^G)\otimes_k\bar k\big)$ in $\Krm_0(\bar k)$.
Lastly, it remains to show
\[
[N(\pi)]\otimes_k\bar k = \Ind_{G^\circ}^G[N(\bar\pi^\circ)]\quad\text{and}\quad [W_G(\Ecal)]\otimes_k\bar k = \Ind_{G^\circ}^G[W_{G^\circ}\big(\overline\Ecal|_{\overline X_L^\circ}\big)],
\]
which can be deduced from $(\mfr_w^j/\mfr_w^{j+1})\otimes_k\bar k \cong \bigoplus_{\bar w}\mfr_{\bar w}^j/\mfr_{\bar w}^{j+1}$ where $\bar w$ runs through the closed points in $\overline X_L$ over $w$. (More details can be found in the proof of Theorem~11 and Theorem~12 in \cite[\S3]{FischbacherWeitz-Koeck:EquivRR}.) Hence, formula \eqref{eq:EquivRR} is valid after the scalar extension to $\bar k$, as desired.
\end{proof}

\begin{rmk}\label{rmk:FWK}
Assume that $k=\bar k$, and let $\pi\colon X_L\to X$ be a connected $G$-cover over $k$, not necessarily weakly ramified everywhere. Then for any $G$-equivariant vector bundle $\Ecal$ (for which $\RGamma(X_L,\Ecal)$ may not be a perfect $kG$-complex), K\"ock \cite[Theorem~3.1]{Koeck:EquivRR} showed\footnote{This formula is generalised to the case where $k$ is any perfect field in \cite[\S4, Theorem~16]{FischbacherWeitz-Koeck:EquivRR}, but there is an obvious typo in the coefficient of $\Bch_kG$ that makes it inconsistent with \cite[Theorem~3.1]{Koeck:EquivRR} when $k=\bar k$.}
\[
\sum_i(-1)^i\Bch_{\Hrm^i(X_L,\Ecal)} = c'\cdot\Bch_{kG} - \frac{1}{|G|}\sum_{w\in|X|}|P_w|\sum_{j=1}^{|I_w/P_w|-1}j\cdot\Ind_{I_w}^G\big(\Bch_{\mfr_w^j\widehat\Ecal_w/\mfr_w^{j+1}\widehat\Ecal_w}\big),
\]
where \[c' = 1+\gen_K + \frac{1}{|G|}\deg(\Ecal)-\frac{\rk\Ecal}{2|G|}\sum_{w\in|X_L|}\left(([I_w:P_w]-1)(|P_w|+1)+\sum_{s\geqslant2}(|I_{w,s}|-1)\right).\]
By Proposition~\ref{prop:inj-cartan}, the verification of formula \eqref{eq:EquivRR} over $\bar k$ reduces to comparing with the Brauer character of the right hand side of \eqref{eq:EquivRR} with the above formula. (This is how the rank-$1$ case of Theorem~\ref{th:EquivRR}(\ref{th:EquivRR:formula}) is proved when $k=\bar k$ in \cite[Theorem~4.3]{Koeck:EquivRR}.) In our proof of Theorem~\ref{th:EquivRR}(\ref{th:EquivRR:formula}), we study the value of $\Bch_{\chi_{kG}(\Ecal)}(g)$ at any $g\ne 1$ and concluded the proof via a simple and conceptual argument using $(\chi_{kG}(\Ecal))^G = \chi_k(\Ecal^G)$ and Lemma~\ref{lem:st-eq}. This proof avoids evaluating the Brauer character at $g=1$. Alternatively, one can explicitly compare the Brauer character value at $g=1$ for both sides of \eqref{eq:EquivRR}, using the theorems of Hirzeburch--Riemann--Roch and Riemann--Hurwitz, as well as the following computation
\[\deg(\Ecal)-|G|\deg(\Ecal^G) = \sum_{w\in Z_L^{\ram}}\sum_{i=1}^{\rk\Ecal}\left(|P_w|(l_{w,i}+1)-1\right),\]
which equals $\sum_{w\in Z_L^{\ram}}\sum_{i=1}^{\rk\Ecal}n_{w,i}$ if we have chosen $0\leqslant n_{w,i}<|I_w|$ for all $w\in Z_L^{\ram}$ and $i$.
This way, one can also show the compatibility of \eqref{eq:EquivRR} with \cite[Theorem~3.1]{Koeck:EquivRR} over $\bar k$. (We note that the coefficient $c'$ of $\Bch_{kG}$ in the above formula may not coincide with the coefficient $c$ of $[kG]$ in $\Krm_0(kG)$ due to the subtlety explained in Remark~\ref{rmk:st-eq}.)

%
\end{rmk}

\section{Relative $\Krm$-theory of group rings}\label{sec:K-thy}
In this section, we review the Euler characteristics in the relative $\Krm_0$-group (\emph{cf.}~Def~\ref{def:Euler-Char-Rel-K0}), and obtain some explicit computation for equivariant vector bundles (\emph{cf.} Corollary~\ref{cor:psi-part}).

Let us consider the following general setup. Let $E$ be either a number field or a (possibly archimedean) local field of characteristic~$0$, and fix a Dedekind domain $R$ contained in $E$. (We usually set $E=\Frac(R)$, but we also allow $E=\RR$ and $R=\ZZ$.) We fix a finite-dimensional semi-simple $E$-algebra $A$ together with an $R$-order $\Afr\subset A$. (In the intended setting, we consider $\Afr = RG$ and $A = EG$ for a finite group $G$.) Thanks to the assumption on $E$, the reduced norm map $\Krm_1(A)\to \zeta(A)^\times$ is injective, where $\zeta(A)$ is the centre of $A$. We  view $\Krm_1(A)$ as a subgroup of $\zeta(A)^\times$; \emph{cf.} \cite[Theorem~(45.3)]{CurtisReiner:Methods2}. \\[-6pt]

We recall the explicit description of the abelian group $\Krm_0(\Afr,A)$ in terms of generators and relations, following \cite[p~215]{Swan:K-thy}:
\begin{description}
    \item[generators] the isomorphism classes of triples $[P,\theta,Q]$, where $P$ and $Q$ are finitely generated projective $\Afr$-modules, $\theta\colon E\otimes_R P \riso E\otimes_R Q$ is an isomorphism of $A$-modules, with respect to the obvious notion of isomorphisms (\emph{cf.} \cite[p~214]{Swan:K-thy});
    \item[relations] generated by the following
\begin{align}\label{eq:Rel-K0-relations}
    [P,\theta,Q]+[P',\theta',Q'] &=[P\oplus P',\theta\oplus\theta',Q\oplus Q'] \quad\text{and}\\
    [P, \theta, P']+[P',\theta',P''] &= [P, \theta'\circ\theta,P'']. \notag
\end{align} 
\end{description}

\begin{subequations}
The relative $K_0$-group fits into the natural \emph{localisation sequence} (\emph{cf.} \cite[Theorem~15.5]{Swan:K-thy}; in particular, we have the following connecting homomorphism
\begin{equation}\label{eq:conn-hom}
    \partial_{\Afr,E}\colon \xymatrix@1{\Krm_1(A) \ar[r]& \Krm_0(\Afr,A)},
\end{equation}
which, in our setting, turns out to be the restriction of the following map:
\begin{equation}\label{eq:conn-hom-delta}
    \delta_{\Afr, E}\colon \xymatrix@1{\zeta(A)^\times \ar[rr]^-{a\mapsto[\Afr,a,\Afr]} && \Krm_0(\Afr,A)}.
\end{equation}
\end{subequations}

\begin{exa}\label{exa:rel-K0-PID}
    Suppose that $\Afr$ is a PID and $A=\Frac\Afr$ is a finite extension of $E$. Then the natural map $A^\times\to \Krm_0(\Afr,A)$, sending $a\in A^\times$ to $[\Afr,a,\Afr]$, induces an isomorphism
    \[
    \xymatrix@1{A^\times/\Afr^\times \ar[r]^-{\cong} & \Krm_0(\Afr,A)}.
    \]
    (This essentially follows from the structure theorem of finitely generated modules over PID. In fact, given $[P,\theta,Q]$ we choose an $\Afr$-basis of $P$ and $Q$ so that $\theta$ can be represented by a diagonal matrix, and the relations \eqref{eq:Rel-K0-relations} yield $[P,\theta,Q]=[\Afr,a,\Afr]$ where $a$ is the determinant of the matrix representation of $\theta$, showing surjectivity. Injectivity can be seen by keeping track of the effect of relations on ``determinants''; \emph{cf.} \cite[Lemma~15.8]{Swan:K-thy}.) 
%
\end{exa}


Now, let $E'$ be another global or local field containing $E$, and choose a Dedekind subdomain $R'\subset E'$ containing $R$. 
We choose an $E$-algebra homomorphism
\[\psi\colon A \to A'\coloneqq\End_{E'}(V_\psi)\]
for some finite-dimensional $E'$-vector space $V_\psi$. Fix an $R'$-lattice $T_\psi$ of $V_\psi$ such that $\psi$ restricts to a map $\Afr\to\Afr'\coloneqq\End_{R'}(T_\psi)$, which we also denote by $\psi$. 

For a (left) $\Afr$-module $M$, we define the following $R'$-module
\begin{equation}\label{eq:psi-part-integral}
    [M]_\psi\coloneqq T_\psi^\ast \otimes_{\Afr} M,
\end{equation}
where we view the $R'$-linear dual $T_\psi^\ast$ as a $(R',\Afr)$-bimodule via $\psi$. Clearly, if $P$ is a projective $\Afr$-module then $[P]_\psi$ is also a projective $R'$-module, so we get a group homomorphism 
\begin{equation}\label{eq:rho-psi}
    \rho^\psi\colon\xymatrix@R=0pt@C=1em{\Krm_0(\Afr,A) \ar[r]&**[r] \Krm_0(R',E');&\\
    [P,\theta,Q] \ar@{|->}[r]&**[r] [P,\theta,Q]_\psi & = \big[[P]_\psi,[\theta]_\psi,[Q]_\psi\big].
    } 
\end{equation}
Alternatively, this map can be obtained the following composition
\[\begin{tikzcd}
    \Krm_0(\Afr,A) \arrow[r, "\psi_\ast"] & \Krm_0(\Afr',A') \arrow[r, "\cong"] & \Krm_0(R',E')
\end{tikzcd} ,\]
where $\psi_\ast$ is induced by the scalar extension $\Afr'\otimes_\Afr(-)$, and the second arrow by the Morita equivalence; \emph{cf.} \cite[\S3.5]{BurnsFlach:ETNC-1}.

Let us record the following basic properties of $\rho^\psi$.
\begin{lem}\label{lem:rho-psi}
Choose an $E$-algebra map $\psi\colon A\to A' = \End_{E'}(V_\psi)$ and an $R'$-lattice $T_\psi\subset V_\psi$ such that $\psi(\Afr) $ is contained in $\Afr'\coloneqq \End_{R'}(T_\psi)$, as above.
\begin{enumerate}
    \item\label{lem:rho-psi:indep} The map $\rho^\psi\colon \Krm_0(\Afr,A) \to \Krm_0 (R',E')$ depends only on $\psi\colon A\to A' = \End_{E'}(V_\psi)$, not on the choice of $R'$-lattice $T_\psi\subset V_\psi$.
    \item\label{lem:rho-psi:conn-hom} We have the following commutative diagram
\[
    \xymatrix@C=1em{
    \Krm_1(A) \ar[d]_-{\partial_{\Afr,E}}\ar[rr]^-{\Nrd^\psi}&& 
     E'^\times  & \ar[l]_-{\cong} \Krm_1(E')\ar[d]^-{\partial_{R',E'}}\\
    \Krm_0(\Afr,A) \ar[rrr]_-{\rho^\psi}&& & \Krm_0(R',E') 
    },
\]
where $\Nrd^\psi$ is the map $\Krm_1(A) \hookrightarrow \zeta(A)^\times\to\zeta(A')^\times = E'^\times$ induced by $\psi$.
\end{enumerate}
\end{lem}
\begin{proof}
    Let $T'_\psi$ be another $R'$-lattice of $V_\psi$ such that $\End_{R'}(T'_\psi)$ contains $\psi(\Afr)$, and choose $\alpha\in A'^\times$ such that $T'_\psi = \alpha(T_\psi)$. Then for any $\Afr$-module $M$ we have an isomorphism 
    \[
    \begin{tikzcd}[column sep=large]
        T_\psi'^\ast\otimes_{\Afr}M\arrow[r, "\alpha^\ast\otimes\id_M"] &
        T_\psi^\ast\otimes_{\Afr}M  
    \end{tikzcd},
    \]
    inducing an isomorphism of triples $T_\psi'^\ast\otimes_{\Afr}(P,\theta,Q)\riso T_\psi^\ast\otimes_{\Afr}(P,\theta,Q)$ for any triple $(P,\theta,Q)$ representing an element of $\Krm_0(\Afr,A)$.

%
    Claim~(\ref{lem:rho-psi:conn-hom}) also follows from the straightforward diagram chasing.
\end{proof}

%

    For $C^\bullet\in \Dperf(\Afr)$, a \emph{trivialisation} over $E$  (or an \emph{$E$-trivialisation})  means an $A$-linear isomorphism
\begin{equation}\label{eq:ht-pairing}
        h\colon \bigoplus_{i \text{ even}}\Hrm^i(E\otimes_RC^\bullet) \riso \bigoplus_{i \text{ odd}}\Hrm^i(E\otimes_RC^\bullet).
\end{equation}
By semi-simplicity of $A$, one can naturally extend $h$ to an isomorphism $\bigoplus_{i \text{ even}}E\otimes_RC^i \riso \bigoplus_{i \text{ odd}}E\otimes_RC^i$, also denoted by $h$.

\begin{defn}\label{def:Euler-Char-Rel-K0}
    Given a perfect $\Afr$-complex with $E$-trivialisation $(C^\bullet, h)$, we define its \emph{Euler characteristic} $\chi_{\Afr,E}(C^\bullet,h) \in \Krm_0(\Afr,A)$ as follows:
\[
\chi_{\Afr,E}(C^\bullet,h) \coloneqq \left[\big(\bigoplus_{i \text{ even}}C^i\big), h, \big(\bigoplus_{i \text{ odd}}C^i\big)\right].
\]
Immediately, $\chi_{\Afr,E}(C^\bullet,h)$ only depends on the isomorphism class of $C^\bullet$ in $\Dperf(\Afr)$.

If $C^\bullet$ is a perfect $\Afr$-complex such that $E\otimes_RC^\bullet$ is acyclic, then we let $\chi_{\Afr,E}(C^\bullet,0)$ denote the Euler characteristic with respect to the \emph{unique} $E$-trivialisation on $C^\bullet$ (i.e., the zero map between the zero modules).
\end{defn}
The formation of $\chi_{\Afr,E}(-)$ is functorial with respect to $\rho^\psi$ \eqref{eq:rho-psi}; indeed,  given a perfect $\Afr$-complex with $E$-trivialisation $(C^\bullet,h)$, we have
\begin{equation}
    \rho^\psi\big(\chi_{\Afr,E}(C^\bullet,h)\big) = \chi_{R',E'}([C^\bullet]_\psi,[h]_\psi), 
\end{equation}
where $[h]_\psi$ is the $E'$-trivialisation induced on $[C^\bullet]_\psi\coloneqq T_\psi^\ast\otimes_\Afr C^\bullet$ via $h$.

\begin{exa}\label{exa:Euler-char-tor}
     If $\Afr$ is a DVR with normalised valuation $v_A$, then let $v_A$ also denote the following isomorphism
    \[ v_A\colon\xymatrix@1{\Krm_0(\Afr,A) \ar[r]^-{\cong}   &  \ZZ} \quad \text{sending } [\Afr,a,\Afr]\mapsto v_A(a).\]
    In that case, for any $C^\bullet\in \Dperf(\Afr)$ with all $\Hrm^i(C^\bullet)$ torsion we have
    \[v_A\left(\chi_{\Afr,E}(C^\bullet,0)\right)  = \sum_{i}(-1)^{i+1}\lth_\Afr\big(\Hrm^i(C^\bullet)\big).\]
\end{exa}

Now, let $\pfr$ be a non-zero \emph{principal} prime ideal of $R$, and write $k\coloneqq R/\pfr$.
Given $C^\bullet_k\in\Dperf(\Afr/\pfr\Afr)$, we abusively write 
\[\chi_{\Afr,E}(C^\bullet_k,0)\coloneqq \chi_{\Afr,E}(C^\bullet_\Afr,0)\]
where  $C^\bullet_\Afr$ is a perfect $\Afr$-complex quasi-isomorphic to $C^\bullet_k$. (Note that $E\otimes_R C^\bullet_\Afr$ is clearly acyclic.) In this case, we have another notion of Euler characteristic; namely, $\chi_{\Afr/\pfr\Afr}(C^\bullet_k)\in\Krm_0(\Afr/\pfr\Afr)$, which can be related to $\chi_{\Afr,E}(C^\bullet_k,0)$ as follows. Define a homomorphism
\[\jmath_{\Afr/\pfr\Afr}\colon\xymatrix@1{\Krm_0(\Afr/\pfr\Afr)\ar[r]&\Krm_0(\Afr,A)}\]
by sending a finitely generated projective $\Afr/\pfr\Afr$-module $P_k$, viewed also as a complex concentrated in degree~$0$, to
\[\jmath_{\Afr/\pfr\Afr}(P_k)\coloneqq \chi_{\Afr,E}(P_k,0) = [\widetilde P,\varpi^{-1},\widetilde P], \]
where $\widetilde P$ is a projective $\Afr$-module lifting $P_k$ and $\varpi$ is a generator of $\pfr$; indeed, $P_k$ is quasi-isomorphic to a two-term complex $[\widetilde P\xrightarrow{\varpi}\widetilde P]$ concentrated in degree $[-1,0]$, so its Euler characteristic is as above. One can show (by straightforward computation) that this extends to any $C^\bullet_k\in\Dperf(\Afr/\pfr\Afr)$; that is,
\[\chi_{\Afr,E}(C^\bullet_k,0) = \jmath_{\Afr/\pfr\Afr}\big(\chi_{\Afr/\pfr\Afr}(C^\bullet_k)\big).\]
Furthermore, given $\psi\colon\Afr\to \End_{R'}(T_\psi)$ as before such that $R'$ is a DVR whose maximal ideal $\pfr'$ contains $\pfr$, we also have
\begin{equation}\label{eq:psi-part-torsion}
    \rho^\psi\big(\chi_{\Afr,E}(C^\bullet_k,0)\big) = \chi_{R',E'}(T^\ast_\psi\otimes^{\Lrm}_\Afr C_k^\bullet,0 ) 
    = \jmath_{R'/\pfr'}\big(\chi_{R'/\pfr'}(T_\psi^\ast/\pfr T_\psi^\ast \otimes_{\Afr/\pfr\Afr}C_k^\bullet) \big).
\end{equation}

\begin{lem}\label{lem:psi-part-torsion}
In the above setting, choose an algebraic closure $\bar k$ of $k$ and a $k$-embedding $R'/\pfr'\hookrightarrow \bar k$. Set $\overline T^\ast_{\psi}\coloneqq T_\psi^\ast\otimes_{R'}\bar k$ and define 
\[
\mrm_{\psi}\colon \xymatrix@1{\Krm_0(\Afr/\pfr\Afr) \ar[rr]^-{\overline T_{\psi}^\ast\otimes_{\Afr/\pfr\Afr}(-)}&& \Krm_0(\bar k) \ar[r]^-{\dim_{\bar k}}_-{\cong}& \ZZ}.
\]
Then for any $C^\bullet_k\in\Dperf(\Afr/\pfr\Afr)$ we have
\[
v_{E'}\left(\rho^\psi\big(\chi_{\Afr,E}(C^\bullet_k,0)\big)\right) = - e_{\pfr'|\pfr}\cdot\mrm_{\psi}\left(\chi_{\Afr/\pfr\Afr}(C^\bullet_k) \right),\]  
where $v_{E'}(-)$ is the normalised valuation on $E'\coloneqq\Frac R'$ and $e_{\pfr'|\pfr}$ is the ramification index.
\end{lem}
\begin{proof}
    By Example~\ref{exa:Euler-char-tor} and \eqref{eq:psi-part-torsion}, we have
    \begin{align*}
        v_{E'}\left(\rho^\psi\big(\chi_{\Afr,E}(C^\bullet_k,0)\big)\right) 
        & = \sum_i(-1)^{i+1}\lth_{R'}\left(\Hrm^i\big(T_{\psi}^\ast/\pfr T_\psi^\ast\otimes_{\Afr/\pfr\Afr}C^\bullet_k\big)\right) \\
        & = \sum_i(-1)^{i+1}\lth_{R'}\left(T_{\psi}^\ast/\pfr T_\psi^\ast\otimes_{\Afr/\pfr\Afr}C^i\right) \\
        &=  e_{\pfr'|\pfr}\cdot\sum_i(-1)^{i+1}\dim_{\bar k}\left(\overline T_{\psi}^\ast\otimes_{\Afr/\pfr\Afr}C^i\right)= 
        -e_{\pfr'|\pfr}\cdot\mrm_{\psi}\left(\chi_{\Afr/\pfr\Afr}(C^\bullet_k) \right),
    \end{align*}
    as desired.
\end{proof}

\begin{rmk}
    In  Lemma~\ref{lem:psi-part-torsion}, the normalised valuation of $\rho^\psi\big(\chi_{\Afr,E}(C^\bullet_k,0)\big)$ changes if we replace $T_\psi$ by the scalar extension under a finite ramified ring extension, but the formation of the fractional ideal $\afr$ whose normalised valuation equals that of $\rho^\psi\big(\chi_{\Afr,E}(C^\bullet_k,0)\big)$ commutes with any finite scalar ring extensions.
\end{rmk}

Now consider a $G$-cover $\pi\colon X_L\to X$ and a $G$-equivariant vector bundle $\Ecal$ in $X_L$ that satisfies the conditions in Theorem~\ref{th:EquivRR}. We also assume that the constant field $k$ of $X$ is \emph{finite} with characteristic $p$. For simplicity we write
\begin{equation}\label{eq:Euler-Char-Rel-K0-VB}
\chi^G_p(\Ecal)\coloneqq \chi_{\ZZ_p G,\QQ_p}(\RGamma(X_L,\Ecal)^\vee,0) \in\Krm_0(\ZZ_p G,\QQ_p G),    
\end{equation}
which makes sense as $\RGamma(X_L,\Ecal)^\vee \in \Dperf(\Fp G)$. Here, $\RGamma(X_L,\Ecal)^\vee$ denotes the Pontryagin dual (or equivalently, the $\FF_p$-linear dual) with the contragredient $G$-action. To motivate this choice, see Theorem~\ref{th:BKK}.
For $\ell\ne p$, we set $\chi^G_{\ell}(\Ecal)=0$ in $\Krm_0(\ZZ_\ell G,\QQ_\ell G)$.

Let $E'$ be a finite extension of $\QQp$  with valuation ring $R'=\Ocal_{E'}$ and residue field $k'$. Pick a $G$-stable $\Ocal_{E'}$-lattice $T_\psi$ in a finitely generated $E'G$-module $V_\psi$, so we get a map $\psi\colon \ZZ_p G \to \End_{R'}(T_\psi)$. We now apply Lemma~\ref{lem:psi-part-torsion} to $\Afr=\ZZp G$ and $\pfr = (p)$ to obtain the following corollary.
\begin{cor}\label{cor:psi-part}
    Let $\pi\colon X_L\to X$ be a $G$-cover that is weakly ramified everywhere. For each $w\in Z^{\ram}_L$, choose $C_w\subseteq I_w$ so that we have $I_w = P_w\rtimes C_w$ (\emph{cf.} Lemma~\ref{lem:Cw}). We write $\theta_{C_w}\coloneqq \theta_w|_{C_w}$. 
    
    Let $k$ be the finite constant field of $X$, and fix a $k$-embedding $k_w\hookrightarrow \bar k$ for each $w\in Z_L^{\ram}$. Finally, fix an embedding $k'\hookrightarrow\bar k$ and set $\overline T_{\psi}\coloneqq \bar k \otimes_{\Ocal_{E'}}T_{\psi}$.
    \begin{enumerate}
        \item\label{cor:psi-part:local} For any $w\in Z_L^{\ram}$ and $j\in \ZZ/(|C_w|)$, we have 
        \[\mrm_{\psi,w}(j)\coloneqq \mrm_{\psi}\big(\Ind_{I_w}^G M_w(j) \big) = \sum_{a=0}^{[k_w:\FF_p]-1} \dim_{\bar k}\big(\overline T_{\psi}[\theta_{C_w}^{jp^a}]\big),\]
        Here, $\overline T_{\psi}[\theta_{C_w}^s]$ for $s\in\ZZ$ is the maximal subspace of $\overline T_{\psi}$ where $C_w$ acts via $\theta_{C_w}^s$.
        \item\label{cor:psi-part:formula} 
    Let $\Ecal$ be a $G$-equivariant vector bundle $X_L$ that satisfies the condition \eqref{eq:Ew} for any $w\in Z_L^{\wild}$. 
    Then we have
    \[  - v_{E'}\left(\rho^\psi\big(\chi^G_p(\Ecal)\big)\right) = v_{E'}(p)\cdot \bigg[ [k:\FF_p]\cdot\dim_{\bar k}\Big(\overline T^\ast_{\psi}\otimes_k \chi_k\big(\Ecal^G\big) \Big) + \ra_\Ecal(\psi) \bigg] \]
    where \[\ra_\Ecal(\psi)\coloneqq \frac{1}{|G|}\sum_{w\in Z_L^{\ram}}\sum_{i=1}^{\rk\Ecal}
    \left(-|P_w|\cdot\sum_{j=1}^{|C_w|-1}j\cdot\mrm_{\psi,w}(-j) + |I_w|\cdot\sum_{j=1}^{l_{w,i}}\mrm_{\psi,w}(j)\right).\]
    Here, $\mrm_{\psi,w}(j)$ is defined in (\ref{cor:psi-part:local}), and $l_{w,i}$'s are as in \eqref{eq:lw}.
    \end{enumerate}
\end{cor}
\begin{proof}
    By transitivity of inductions, we have $\Ind_{I_w}^GM_w(j) \cong \Ind_{C_w}^G\big((\mfr_w^{j}/\mfr^{j+1})|_{C_w}\big)$; \emph{cf.} Lemma~\ref{lem:proj-cover}. Since 
    $\overline T^\ast_{\psi}\otimes_{\FF_p} \Ind_{C_w}^G(\theta_{C_w}^j)$ is a projective $\bar kG$-module,\footnote{One can show projectivity by realising $\overline T^\ast_{\psi}\otimes_{\FF_p} \Ind_{C_w}^G(\theta_{C_w}^j)$ as a direct summand of $\Ind_1^G (k_w\otimes_{\FF_p} \overline T^\ast_{\psi})$.} we have
    \begin{multline}\label{eq:Frob-Rec-psi}
        \overline T^\ast_{\psi}\otimes_{\FF_p G} \Ind_{C_w}^G\big((\mfr_w^{j}/\mfr^{j+1})|_{C_w}\big) \cong \Big(\overline T^\ast_{\psi}\otimes_{\FF_p} \Ind_{C_w}^G\big((\mfr_w^{j}/\mfr^{j+1})|_{C_w}\big)\Big)^G \\
        \cong \Hom_{\bar kG}\Big(\overline T_{\psi},\,\bar k\otimes_{\FF_p}\Ind_{C_w}^G\big((\mfr_w^{j}/\mfr^{j+1})|_{C_w}\big) \Big) \cong \Hom_{\bar k[C_w]}\big(\overline T_{\psi},\,\bar k\otimes_{\FF_p}(\mfr_w^{j}/\mfr^{j+1})\big), 
    \end{multline}
    where the last isomorphism is the Frobenius reciprocity. 
    To conclude, observe that
    \[\bar k\otimes_{\FF_p}(\mfr_w^{j}/\mfr_w^{j+1}) \cong \bigoplus_{i=0}^{[k_w:\FF_p]-1}\bar k \otimes_{\Fr_p^i,k_w}(\mfr_w^{j}/\mfr_w^{j+1}),\] 
     where $\Fr_p$ is the $p$th power map and $C_w$ acts on $\bar k \otimes_{\Fr_p^i,k_w}(\mfr_w^{j}/\mfr_w^{j+1})$ via $\theta_w^{jp^i}$. In particular, the last term in \eqref{eq:Frob-Rec-psi} is the direct sum of the multiplicity spaces for $\theta_w^{jp^i}$, and thus the displayed equation in claim~(\ref{cor:psi-part:local}) follows. 
    
    For any $k$-vector space $V$, we have 
    \[\overline T^\ast_{\psi}\otimes_{\FF_p G}\Ind_1^GV \cong \overline T^\ast_{\psi}\otimes_{\FF_p}V \cong (\overline T^\ast_{\psi}\otimes_{k}V)^{\oplus[k:\FF_p]}.\] 
    Note also that $M_w(j)^\vee \cong M_w(-j)$, where $M_w(j)^\vee$ is the Pontryagin dual with contragredient $G_w$-action.
    Claim~(\ref{cor:psi-part:formula}) now follows from Theorem~\ref{th:EquivRR}(\ref{th:EquivRR:formula}) and Lemma~\ref{lem:psi-part-torsion}.
\end{proof}

We conclude the section with some remarks on Corollary~\ref{cor:psi-part}.
\begin{rmk}
    In Corollary~\ref{cor:psi-part}(\ref{cor:psi-part:local}), we have $\overline T_{\psi}[\theta_{C_w}^s]\cong \overline T^{\ssim}_{\psi}[\theta_{C_w}^s]$ as $k_w[C_w]$ is semi-simple, which shows that $\mrm_{\psi,w}(j)$ depends only on $j\in \ZZ/(|C_w|)$ and $V_\psi$, not on the choice of $T_\psi\subset V_\psi$. In particular, the formula in Corollary~\ref{cor:psi-part}(\ref{cor:psi-part:formula}) is independent of the choice of $T_\psi\subset V_\psi$. 
\end{rmk}

\begin{rmk}\label{rmk:psi-part-intended-case}
    The right hand side of the formula in Corollary~\ref{cor:psi-part}(\ref{cor:psi-part:formula}) can be divided into two parts -- the first terms involves $\deg\psi\coloneqq\dim_{E'}V_\psi$ and the Euler characteristic of $\Ecal^G$, and the second term $\ra_\Ecal(\psi)$ measures the ``local ramification'' of $\overline T_{\psi}^{\ssim}$ and $\Ecal$ (that is, $\overline T_{\psi}^{\ssim}|_{I_w}$ and $l_{w,i}$'s for any $w\in Z^{\ram}_L$).
    
    Let us now consider the special case of $\Ecal = \pi^\ast\Fcal(-Z_L)$ where $\Fcal$ is a vector bundle on $X$ and $Z_L = \pi^{-1}(Z)$ for some closed subset $Z\subset |X|$ containing the ramification locus for $\pi$.  
    Then as in \eqref{eq:EquivRR-intended-case}, the formula in Corollary~\ref{cor:psi-part}(\ref{cor:psi-part:formula}) can be made more explicit using the following formula:
    \begin{align*}
        \dim_{\bar k}\Big(\overline T^\ast_{\psi}\otimes_k \chi_k\big(\Ecal^G\big) \Big) & =(\deg\psi)\cdot\Big( (\rk\Fcal)\cdot\big(1-\gen_K-\deg(Z)\big) + \deg(\Fcal) \Big);\\
        \ra_\Ecal(\psi) &= \frac{\rk\Ecal}{|G|}\sum_{w\in Z_L^{\ram}}\sum_{j=1}^{|I_w/P_w|-1}\sum_{a=0}^{[k_w:\FF_p]-1}j|P_w|\cdot\dim_{\bar k}\big(\overline T_{\psi}[\theta_{C_w}^{jp^a}]\big),
    \end{align*}
    where $\deg(Z) = \sum_{v\in Z}[k_v:k]$.
\end{rmk}

\begin{rmk}\label{rmk:ra-psi}
\setcounter{equation}{\value{equation}-1}
\begin{subequations}
    Let us make $\ra_\Ecal(\psi)$ more explicit in some special cases. Firstly, if we have $I_w = P_w$ for all $w\in Z^{\ram}_L$ (e.g., if $G$ is a $p$-group or $\pi$ is \'etale), then we have $\ra_\Ecal(\psi) = 0$ for any $\psi$ and $\Ecal$ (\emph{cf.} Remark~\ref{rmk:EquivRR-p-group}).

    Now, suppose that $\pi$ is tame everywhere and $\deg\psi=1$. We also specialise to the case when $\Ecal = \pi^\ast\Fcal(-Z_L)$ as in Remark~\ref{rmk:psi-part-intended-case}. For any $w\in Z^{\ram}_L$, let $d_w$ denote the smallest positive integer such that $p^{d_w}\equiv 1\bmod{|I_w|}$.  (Note that $d_w$ divides $[k_w:\FF_p]$.) For each $a=0,\cdots,d_w-1$, let $j^{(a)}_{\psi,w}$ denote the integer in $\{0,\cdots,|I_w|-1\}$ such that $I_w$ acts on $\overline T_{\psi}|_{I_w}$ via $\theta_w^{j^{(a)}_{\psi,w}\cdot p^a}$.  Then we have
    \begin{equation}\label{eq:ra-psi:gen}
        \ra_\Ecal(\psi) = \frac{\rk\Ecal}{|G|}\sum_{w\in Z_L^{\ram}}\sum_{a=0}^{d_w-1}j_{\psi,w}^{(a)}\cdot\frac{[k_w:\FF_p]}{d_w} .
    \end{equation}
    If furthermore $|I_w|$ divides $p-1$ for any $w\in Z^{\ram}_L$ (which follows if $|G|$ divides $p-1$), then setting $j_{\psi,w} = j_{\psi,w}^{(0)}$ we obtain
    \begin{equation}\label{eq:ra-psi:simple}
        \ra_\Ecal(\psi) = \frac{\rk\Ecal}{|G|}\sum_{w\in Z_L^{\ram}}j_{\psi,w}\cdot [k_w:\FF_p].
    \end{equation}
\end{subequations}
\end{rmk}

\section{Review of N\'eron models and base change}\label{sec:Neron}

We use the setting of \S\ref{sec:EquivRR}. Fix an abelian variety $A$ over $K$, and let $\Acal$ denote the N\'eron model of $A$ over $X$. For the pull back $A_L$ of $A$ over $L$, we let $\Acal_L$ denote the N\'eron model, and write $\Acal_{X_L}\coloneqq \Acal\times_X X_L$. The connected N\'eron models are denoted as $\Acal^\circ$, $\Acal^\circ_L$, etc. Let $A^t$ denote the dual abelian variety of $A$, and similarly define $\Acal^t$, $\Acal^{t,\circ}$, etc. We maintain this setting for the rest of the paper.

By the N\'eron mapping property, the $G$-action on $X_L$ lifts to a  $G$-action on $\Acal_L$, and we get a natural $G$-equivariant homomorphism $\Acal_{X_L}\to \Acal_L$ extending the identity map on the generic fibre. Furthermore, we have the following proposition.
\begin{prop}\label{prop:descent-conn-Neron-models}
Let $U'\subset X$ be the \emph{maximal} open subscheme such that the natural map $\Acal^\circ_{X_L}|_{U'_L}\to \Acal^\circ_L|_{U'_L}$ is an isomorphism, where $U'_L\coloneqq \pi^{-1}(U')$. 
\begin{enumerate}
     \item\label{prop:descent-conn-Neron-models:coker} The cokernel of the natural inclusion
    \[\begin{tikzcd}
        \Lie(\Acal_{X_L})\cong \pi^\ast\Lie (\Acal) \arrow[hook]{r} & \Lie(\Acal_L)
    \end{tikzcd}\]
    is supported exactly on $X_L\setminus U_L'$.
     \item\label{prop:descent-conn-Neron-models:sst}  A closed point $v\in |X|$ lies in $U'$ if either $L/K$ is unramified at $v$ or $A$ has semistable reduction at $v$. 
    \item\label{prop:descent-conn-Neron-models:tame} Suppose furthermore that $L/K$ is \emph{tamely ramified} at all places in $X'\setminus U'$. Then the natural map $\Lie(\Acal)\to\Lie(\Acal_L)^G$ is an isomorphism.
\end{enumerate}
\end{prop}
\begin{proof}
    Fix a place $v\in|X|$, and choose a place $w\in |X_L|$ over $v$. We set 
    \[\Acal_{\Ocal_v}\coloneqq \Acal\times_{X}\Spec\Ocal_v,\quad \Acal_{L,\Ocal_w}\coloneqq \Acal_L\times_{X_L}\Spec\Ocal_w, \quad\text{and}\quad \Acal_{\Ocal_w}\coloneqq\Acal\times_X\Spec\Ocal_w.\]
    We similarly define $\Acal_{\Ocal_v}^\circ$, etc. By standard properties of N\'eron models,  we have $v\in U'$ if and only if the natural map $\Acal^\circ_{\Ocal_w}\to\Acal^\circ_{L,\Ocal_w}$ is an isomorphism, and $U'$ contains the unramified locus for $L/K$; \emph{cf.} \S1.2, Proposition~4 and \S7.2, Corollary~2 in \cite{BoschLuetkebohmertRaynaud:NeronModels}. If $\Acal^\circ_{\Ocal_v}$ is a semi-abelian scheme, then $v\in U'$ by \cite[\S7.4, Corollary~4]{BoschLuetkebohmertRaynaud:NeronModels}). This proves (\ref{prop:descent-conn-Neron-models:sst}). 
    
    If $L_w/K_v$ is tamely ramified, then by \cite[Theorem~4.2]{Edixhoven:Tame} we have a natural isomorphism 
    \[\xymatrix@1{\Acal_{\Ocal_v} \ar[r]^-{\cong}& \left(\Res_{\Ocal_w/\Ocal_v}\Acal_{L,\Ocal_w}\right)^{G_w}}\]
    of group schemes over $\Ocal_v$, where $\Res_{\Ocal_w/\Ocal_v}\Acal_{L,\Ocal_w}$ denotes the Weil restriction of scalars. Since the Lie algebra of $\Res_{\Ocal_w/\Ocal_v}\Acal_{L,\Ocal_w}$ coincides with $\Lie(\Acal_L)(\Ocal_w)$ viewed as an $\Ocal_v$-module, it follows that the natural map $\Lie(\Acal)\to\Lie(\Acal_L)^G$ induces an isomorphism on the completed stalks at all \emph{tame} places $v\in|X|$. Since this map induces an isomorphism on the restriction to $U'$ by (\ref{prop:descent-conn-Neron-models:sst}), we obtain  (\ref{prop:descent-conn-Neron-models:tame}) by the standard descent argument.

    To prove (\ref{prop:descent-conn-Neron-models:coker}), we need to show that the natural map $\Acal^\circ_{\Ocal_w}\to \Acal^\circ_{L,\Ocal_w}$ is isomorphic if and only if $\Lie(\Acal_{X_L})(\Ocal_w)\hookrightarrow \Lie(\Acal_L)(\Ocal_w)$ is isomorphic. The ``only if'' direction is clear, so suppose that we have $\Lie(\Acal_{X_L})(\Ocal_w)\riso \Lie(\Acal_L)(\Ocal_w)$. Then the natural map $\Acal^\circ_{\Ocal_w}\to \Acal^\circ_{L,\Ocal_w}$ is \'etale by smoothness of the source and the target (\emph{cf.} \cite[\S2.2, Corollary~10]{BoschLuetkebohmertRaynaud:NeronModels}), so it is an open immersion by Zariski's main theorem (\emph{cf.} \cite[\S2.3, Theorem~2{$'$}]{BoschLuetkebohmertRaynaud:NeronModels}). Since all the fibres of $\Acal^\circ_{L,w}$ over $\Spec\Ocal_w$ is connected, the desired claim now follows. 
\end{proof}

Set $\Bcal_{\Ocal_v}\coloneqq\Res_{\Ocal_w/\Ocal_v}\Acal_{L,\Ocal_w}$, and denote its special fibre by $\Bcal_{k_v}$. Let $\Acal_{L,k_w}$ be the special fibres of $\Acal_{L,\Ocal_w}$. Then have a natural $G_w$-equivariant surjective map
\begin{equation}\label{eq:Edixhoven-Fil}
    \begin{tikzcd}
       \Bcal_{k_v} \arrow[r, two heads] & \Res_{k_w/k_v} \Acal_{L,k_w}
    \end{tikzcd} 
\end{equation}
    with smooth connected unipotent kernel denoted by $\Frm^1\Bcal_{k_v}$; indeed, this can be seen by realising $\Bcal_{k_v}$ as the Weil restriction of scalars for a nilpotent thickening $\Ocal_w\otimes_{\Ocal_v}k_v\twoheadrightarrow k_w$; \emph{cf.} \cite[\S5.1]{Edixhoven:Tame} or \cite[Proposition~A.5.12]{ConradGabberPrasad:PRedGp2}. 
    
    We retain the setting that $X$ is defined over a \emph{perfect} field $k$ of characteristic $p>0$, so $k_v$ is perfect as well. Then $\Frm^1\Bcal_{k_v}$ is a vector group over $k_v$ by \ \cite[Corollary~B.2.7]{ConradGabberPrasad:PRedGp2}; i.e., it is a direct product of copies of $\GG_a$.
\begin{prop}\label{prop:Unip-Rad-Descent}
\setcounter{equation}{\value{equation}-1}
\begin{subequations}
    In the above setting, if $L_w/K_v$ is tame then the short exact sequence \eqref{eq:Edixhoven-Fil} induces the following short exact sequence 
    \begin{equation}\label{eq:Unip-Rad-Descent:Group}
    \begin{tikzcd}[column sep = scriptsize]
            0 \arrow[r] & \big(\Frm^1\Bcal_{k_v}\big)^{G_w} \arrow[r] & \big(\Bcal_{k_v} \big)^{G_w}\cong\Acal_{k_v} \arrow[r]&  \big(\Res_{k_w/k_v}\Acal_{L,k_w}\big)^{G_w}  \arrow[r]& 0
    \end{tikzcd},
    \end{equation}
    which induce the following isomorphism
    \begin{equation}\label{eq:Unip-Rad-Descent:Lie}
         \big(\Lie(\Acal_L)(k_w)\big)^{G_w} \cong \coker\left(
    \big(\Lie(\Acal_L)(\mfr_w/\mfr_w^{|I_w|}) \big)^{G_w} \hookrightarrow \Lie(\Acal)(k_v)
    \right).
    \end{equation}
    Moreover, $\big(\Frm^1\Bcal_{k_v}\big)^{G_w}$ is a \emph{vector group} over $k_v$ and the sequence \eqref{eq:Unip-Rad-Descent:Group} remains exact on $k_v$-points.
\end{subequations}
\end{prop}
\begin{proof}
    Suppose that we know the sequence \eqref{eq:Unip-Rad-Descent:Group} is exact and that $\big(\Frm^1\Bcal_{k_v}\big)^{G_w}$ is a vector group. Then the isomorphism \eqref{eq:Unip-Rad-Descent:Lie} is a direct consequence of the short exact sequence of Lie algebras induced from \eqref{eq:Unip-Rad-Descent:Group}, and the sequence \eqref{eq:Unip-Rad-Descent:Group} remains exact on $k_v$-points since $\Hrm^1(k_v,(\Frm^1\Bcal_{k_w})^{G_w})$ is trivial by the Hilbert normal basis theorem. 

    It remains to show that the sequence \eqref{eq:Unip-Rad-Descent:Group} is exact and that $\big(\Frm^1\Bcal_{k_v}\big)^{G_w}$ is a vector group, both of which can be checked after base change to $k_w$. Set $K_w'\coloneqq (L_w)^{I_w}$, and let $\Ocal_w'$ denote its valuation ring. Then as $\Ocal'_w$ is a finite \'etale extension of $\Ocal_v$, we have
    \[
    \Bcal_{\Ocal_v}\times_{\Spec\Ocal_v}\Spec\Ocal'_w \cong \Res_{(\Ocal_w\otimes_{\Ocal_v}\Ocal'_w)/\Ocal'_w}\Acal_{L,\Ocal_w} \cong \prod_{G_w/I_w}\big( \Res_{\Ocal_w/\Ocal'_w}\Acal_{L,\Ocal_w} \big),
    \]
    where the natural $G_w$-action is the extension of the natural $I_w$-action on $\Res_{\Ocal_w/\Ocal'_w}\Acal_{L,\Ocal_w}$ so that $G_w$ acts transitively on the factors. Therefore, by taking $G_w$-invariants we get 
    \[
    (\Bcal_{\Ocal_v})^{G_w}\times_{\Spec\Ocal_v}\Spec\Ocal'_w \cong \big( \Res_{\Ocal_w/\Ocal'_w}\Acal_{L,\Ocal_w} \big)^{I_w}.
    \]
    If we let $\Bcal_{k_w}$ denote the special fibre of $\Res_{\Ocal_w/\Ocal'_w}\Acal_{L,\Ocal_w}$, then we can also show that
    \[
    (\Frm^1\Bcal_{k_v})^{G_w}\times_{\Spec k_v}\Spec k_w \cong  (\Frm^1\Bcal_{k_w})^{I_w}.
    \]
    Therefore, to prove the proposition we may replace $K_v$ with $K'_w$ and suppose $G_w = I_w$.
    
    Now, suppose that $L_w/K_v$ is totally ramified, so we write $\Bcal_{k_w}$ and $\Acal_{k_w}$ for $\Bcal_{k_v}$ and $\Acal_{k_v}$. Then by tameness, $|I_w|$ acts invertibly on the vector group $\Frm^1\Bcal_{k_w}$, which yields the following short exact sequence of smooth $k_w$-group schemes
    \[
    \begin{tikzcd}[column sep = scriptsize]
        0 \arrow[r] & \big(\Frm^1\Bcal_{k_w}\big)^{I_w} \arrow[r] & \big(\Bcal_{k_w}\big)^{I_w}\cong \Acal_{k_w} \arrow[r] & \big( \Acal_{L,k_w}\big)^{I_w} \arrow[r] & 0
    \end{tikzcd}.
    \]
    Clearly, $\big(\Frm^1\Bcal_{k_w}\big)^{I_w}$ is still a vector group. 
    This concludes the proof.
\end{proof}

\begin{cor}\label{cor:Unip-Rad-Descent}
    In the same setting as in Proposition~\ref{prop:Unip-Rad-Descent}, if $\Acal^\circ_{L,\Ocal}$ is semi-abelian then we have $(\Frm^1\Bcal_{k_v})^{G_w}\cong \Rscr_u(\Acal^\circ_{k_v})$, the unipotent radical of the neutral component of $\Acal_{k_v}$. 
\end{cor}
\begin{proof}
    Since $\Acal^\circ_{L,k_w}$ is semi-abelian (hence, with trivial unipotent radical), the unipotent radical of the neutral component of $\big(\Res_{k_w/k_v}\Acal_{L,k_w}\big)^{G_w}$ is also trivial, so  the exact sequence \eqref{eq:Unip-Rad-Descent:Group} identifies $(\Frm^1\Bcal_{k_v})^{G_w}$ with the unipotent radical of $\Acal_{k_v}^\circ$. 
\end{proof}

Now, choose a dense open subscheme $U\subset X$ contained in both the good reduction locus for $A/K$ and the unramified locus for $\pi$, and set $U_L\coloneqq \pi^{-1}(U)$. Let $Z\coloneqq X\setminus U$ and $Z_L\coloneqq X_L\setminus U_L$ respectively denote the reduced complements. 
\begin{prop}\label{prop:descent-Neron}
\setcounter{equation}{\value{equation}-1}
\begin{subequations}
    We make the same assumption as in Proposition~\ref{prop:descent-conn-Neron-models}(\ref{prop:descent-conn-Neron-models:tame}), and let $Z$ and $Z_L$ be as above. Then we have natural inclusions
\begin{equation}\label{eq:descent-Neron:incl}
\begin{tikzcd}
        \Lie(\Acal)(-Z) \arrow[hook]{r} & \big(\Lie(\Acal_L)(-Z_L)\big)^G \arrow[hook]{r} &\Lie(\Acal)(-(Z\cap U'))
    \end{tikzcd}.
\end{equation}
    that restrict to isomorphisms over $U'$, and the cokernel of the first inclusion is supported exactly on  $X\setminus U'$. Furthermore,  we have
    \begin{equation}\label{eq:descent-Neron:coker}
    \frac{\Lie(\Acal)(-(Z\cap U'))}{\big(\Lie(\Acal_L)(-Z_L)\big)^G} \cong \bigoplus_{v\notin U'}\big(\Lie(\Acal_L)(k_{\tilde v})\big)^{G_{\tilde v}},
    \end{equation}
    where we choose a preimage $\tilde v\in \pi^{-1}(v)$ for each $v\notin U'$.
    \end{subequations}
\end{prop}
\begin{rmk}\label{rmk:Lie-inv}
    The last displayed equation is independent of the choice of $\tilde v$ since we have 
\[\big(\Lie(\Acal_L)(k_{\tilde v})\big)^{G_{\tilde v}}\cong \Big(\bigoplus_{w\mid v}\Lie(\Acal_L)(k_w)\Big)^G.\]
\end{rmk}
\begin{proof}[Proof of {Proposition~\ref{prop:descent-Neron}}]
    We  use the notation from Proposition~\ref{prop:Unip-Rad-Descent} and its proof. By Proposition~\ref{prop:descent-conn-Neron-models}(\ref{prop:descent-conn-Neron-models:tame}), we have a natural isomorphism $\Lie(\Acal)(-Z)|_{U'}\cong\big(\Lie(\Acal_L)(-Z_L)\big)^G|_{U'}$. 
    For any $v\notin U'$ we have
    \[
    \frac{\big(\Lie(\Acal)(-(Z\cap U')\big)\widehat{_v}}{\big(\Lie(\Acal)(-Z))\big)\widehat{_v}} \cong\frac{\Lie(\Acal)(\Ocal_v)}{\Lie(\Acal)(\mfr_v)} \cong \Lie(\Acal)(k_v),
    \]
    and the preimage of $\Lie(\Frm^1\Bcal_{k_{\tilde v}})^{G_{\tilde v}}$ in $\Lie(\Acal)(\Ocal_v)$ can be naturally identified with the completed stalk of $\big(\Lie(\Acal_L)(-Z_L)\big)^G$ at $v$. Thus we get the desired inclusions of vector bundles. Furthermore, we have 
        \[\frac{\Lie(\Acal)(-(Z\cap U'))}{\big(\Lie(\Acal_L)(-Z_L)\big)^G} \cong \bigoplus_{v\notin U'}\frac{\Lie(\Acal)(k_v)}{\Lie(\Frm^1\Bcal_{k_{\tilde v}})^{G_{\tilde v}}}\cong \bigoplus_{v\notin U'}\big(\Lie(\Acal_L)(k_{\tilde v})\big)^{G_{\tilde v}},\]
        where the last isomorphism follows from Proposition~\ref{prop:Unip-Rad-Descent}. 
        
        Lastly, we show that $\Lie(\Acal)(k_v)\not\cong\big(\Lie(\Acal_L)(k_{\tilde v})\big)^{G_{\tilde v}}$ for any $v\notin U'$ and $\tilde v \mid v$; i.e., the cokernel of $\Lie(\Acal)(-Z) \hookrightarrow \big(\Lie(\Acal_L)(-Z_L)\big)^G$ is supported exactly on  $X\setminus U'$. Indeed, if we have $\Lie(\Acal)(k_v)\cong\big(\Lie(\Acal_L)(k_{\tilde v})\big)^{G_{\tilde v}}$, then it follows that the following composition
       \[
       \Lie(\Acal)(k_v)\otimes_{k_v}k_{\tilde v} \cong \big(\Lie(\Acal_L)(k_{\tilde v})\big)^{G_{\tilde v}}\otimes_{k_v}k_{\tilde v} \hookrightarrow \Lie(\Acal_L)(k_{\tilde v})
       \]
        is an isomorphism for the dimension reason. Hence, by the Nakayama lemma, the natural map $\pi^\ast\Lie(\Acal)\to\Lie(\Acal_L)$ is isomorphic at $\tilde v$, so $v\in U'$ by Proposition~\ref{prop:descent-conn-Neron-models}(\ref{prop:descent-conn-Neron-models:coker}). This concludes the proof.
\end{proof}

     Under the same setting as in Proposition~\ref{prop:descent-conn-Neron-models}(\ref{prop:descent-conn-Neron-models:tame}), choose integers $r_{w,i}\in\{0,\cdots,|I_w|-1\}$ for any $w\in Z'_L$ so that $\Lie(\Acal_L)(k_w)\cong\bigoplus_i (\mfr_w^{-r_{w,i}}/\mfr_w^{-r_{w,i}+1})$; \emph{cf.} Lemma~\ref{lem:tame-inertial-action}. Then one can shows that
\begin{equation}\label{eq:rw}
\dim_{k_v}\big(\Lie(\Acal_L)(k_{\tilde v})\big)^{G_{\tilde v}} = d_{\tilde v}'= \left|\{i \text{ such that } r_{w,i}=0\} \right|. 
\end{equation}

Let us record the following immediate corollary.
\begin{cor}\label{cor:descent-Neron}
Under the same assumption as in Proposition~\ref{prop:descent-Neron}, the following properties hold.
    \begin{enumerate}
        \item\label{cor:descent-Neron:EC} If $A$ is an elliptic curve, then we have  $\big(\Lie(\Acal_L)(-Z_L)\big)^G \riso \Lie(\Acal)(-(Z\cap U'))$.
        \item\label{cor:descent-Neron:deg} We have
        \begin{multline*}
            \deg\Big(\big(\Lie(\Acal_L)(-Z_L)\big)^G\Big) = \deg\Big(\Lie(\Acal)(-(Z\cap U'))\Big) \\
            -\sum_{v\notin U'}[k_v:k]\cdot\dim_{k_v}\big(\Lie(\Acal_L)(k_{\tilde v})\big)^{G_{\tilde v}},
        \end{multline*}
        where we choose $\tilde v\in \pi^{-1}(v)$ for each $v\notin U'$.
        \item\label{cor:descent-Neron:perf} Suppose additionally that $L/K$ is weakly ramified everywhere. Then the $G$-equivariant vector bundle $\Lie(\Acal_L)(-Z_L)$ satisfies the condition  \eqref{eq:Ew}, so Theorem~\ref{th:EquivRR} holds. Furthermore, we can compute  $l_{w,i}$'s defined in \eqref{eq:lw} as follows:
        \[
        l_{w,i} = \left\{
        \begin{array}{ll}
         |I_w/P_w|-1,    & \forall\,i \text{ if } w\in Z^{\ram}_L\setminus Z'_L; \\
         r_{w,i} -1      & \text{ if } w\in Z'_L\text{ and }r_{w,i}\ne0;\\
         |I_w/P_w|-1     & \text{ if } w\in Z'_L\text{ and }r_{w,i}=0,
        \end{array}
        \right.
        \]
        where $r_{w,i}$'s are defined in \eqref{eq:rw}.
    \end{enumerate}
\end{cor}
\begin{proof}   
    Claim~(\ref{cor:descent-Neron:EC}) can be deduced from the statement on the cokernels of the inclusions in  \eqref{eq:descent-Neron:incl} in  Proposition~\ref{prop:descent-Neron}, noting that all the vector bundles involved are line bundles. Claim~(\ref{cor:descent-Neron:deg}) is clear from Proposition~\ref{prop:descent-Neron}. 

    Suppose that $L/K$ is weakly ramified everywhere and we want to verify \eqref{eq:Ew} for $\Ecal = \Lie(\Acal_L)(-Z_L)$, which is a local condition at each $w\in Z_L^{\wild}$. Since the natural inclusion $\Lie(\Acal_{X_L})(-Z_L) \hookrightarrow \Lie(\Acal_L)(-Z_L)$ restricts to an isomorphism over $U'_L$, the completed stalk of $\Lie(\Acal_L)(-Z_L)$ at each $w\in U'_L$ is $G_w$-equivariantly isomorphic to $\mfr_w^{\oplus d}$ with $d=\dim(A)$; \emph{cf.} \eqref{eq:nw-div-twist}. Now, the condition \eqref{eq:Ew} follows since the assumption of Proposition~\ref{prop:descent-Neron} implies that $Z_L^{\wild}\subseteq Z_L\cap U'_L$. It also implies that $l_{w,i}\equiv -1 \bmod{|P_w|}$ for any $w\in Z^{\ram}_L\cap U'_L$. The computation of $l_{w,i}$ is clear for $w\notin U'_L$, so Claim~(\ref{cor:descent-Neron:perf}) now follows.
\end{proof}

\begin{exa}\label{exa:descent-semistable}
    The assumption in Corollary~\ref{cor:descent-Neron}(\ref{cor:descent-Neron:perf}) is satisfied if $L/K$ is weakly ramified everywhere, $A$ has semistable reduction at all places of $L$, and $L/K$ is tame at all places in $K$ where $A$ does \emph{not} have semistable reduction. In that case, $|X\setminus U'|$ is precisely the set of places of non-semistable reduction for $A$. As a special case, if $A$ has semistable reduction at all places in $K$ then by Proposition~\ref{prop:descent-Neron} we have $\Lie(A)(-Z)\riso\big(\Lie(\Acal_L)(-Z_L)\big)^G$.
\end{exa}

\section{Review of Equivariant BSD and Hasse--Weil--Artin $L$-values}\label{sec:BKK}
We introduce a certain perfect $\widehat\ZZ G$-complex encoding the integral Galois module structure of the arithmetic invariants of $A/L$, and review the main result of \cite{BurnsKakdeKim:EquivBSDTame} on the equivariant BSD conjecture.
We maintain the setting of \S\ref{sec:Neron}, and additionally assume that $k$ is a \emph{finite field} of characteristic $p$. In particular, $L/K$ is an arbitrary finite Galois extension of global function fields. 

For each $w\in Z_L$, we set
    \[\Acal_L^\circ(\mfr_w)\coloneqq\ker\left(\Acal_L^\circ(\Ocal_w) \to \Acal_L^\circ(k_w)\right),\]
which is a $G_w$-stable pro-$p$ open subgroup of $A(L_w)$. Following \cite[\S2.2]{KatoTrihan:BSD}, we let $\RGamma_{\ari,Z_L}(U_L,\Acal_{L,\tors})\in D(\ZZ G)$ denote the mapping fibre of
\begin{equation}\label{eq:arith-coho}
    \begin{tikzcd}[column sep = small]
        \RGamma_{\fl}(U_L, \Acal_{L,\tors})\oplus \left( \stackbin[w\in Z_L]{}{\bigoplus} \Acal_L^\circ(\mfr_w)\stackbin{\Lrm}{\otimes}\QQ/\ZZ\right)[-1] \arrow{r}& \stackbin[w\in Z_L]{}{\bigoplus}  \RGamma_{\fl} (\Spec L_w, \Acal_{L,\tors})
    \end{tikzcd}.
\end{equation}

 We will often write $\whSC_{Z_L} = \whSC_{Z_L}(A,L/K)$.

\begin{defn}\label{def:SC}
    We set 
    \[\whSC_{Z_L}(A,L/K)\coloneqq\left(\RGamma_{\ari,Z_L}(U_L,\Acal_{L,\tors})\right)^\vee[-2] \in \Drm(\widehat\ZZ G),\] 
    where $(-)^\vee$ is the Pontryagin dual.
    For any prime $\ell$ (allowing $\ell = p$), we write 
    \[\SC_{Z_L,\ell}(A,L/K)\coloneqq \whSC_{Z_L}(A,L/K)\otimes_{\widehat\ZZ}\ZZ_\ell \in D(\ZZ_\ell G).\]
    We will often write $\whSC_{Z_L} = \whSC_{Z_L}(A,L/K)$ and $\SC_{Z_L,\ell} = \SC_{Z_L,\ell}(A,L/K)$ for simplicity.
\end{defn}

\begin{prop}[{\emph{cf.} \cite[\S2.5]{KatoTrihan:BSD}}]\label{prop:SC}
    We have $\Hrm^i(\whSC_{Z_L})=0$ for $i\notin [0,2]$. Furthermore, if $\Sha(A/L)$ is finite, then we have 
    \[\Hrm^0(\whSC_{Z_L}) \cong A^t(L)\otimes\widehat\ZZ\] 
    and a long exact sequence
    \[\xymatrix@1@C=1em{0\ar[r]&\Sel_{\QQ/\ZZ}(A/L)^\vee \ar[r]& \Hrm^1(\whSC_{Z_L})\ar[r]& \stackbin[w\in Z_L]{}{\bigoplus}\left(\Acal_L(k_w)\right)^\vee \ar[r]& \left(A(L)_{\tors}\right)^\vee\ar[r]&\Hrm^2(\whSC_{Z_L})\ar[r]&0}.\]
\end{prop}
\begin{proof}
    Apply \cite[\S2.5, \S2.3]{KatoTrihan:BSD} to $V = (\Acal_L^\circ(\mfr_w))_{w\in Z_L}$, noting that $A(L_w)/\Acal_L^\circ(\mfr_w) \cong \Acal_L(\Ocal_w)/\Acal_L^\circ(\mfr_w) \cong \Acal_L(k_w)$.
\end{proof} 

By Schneider's result \cite[p~509]{Schneider:BSD}, we have a non-degenerate $G$-equivariant pairing
\begin{equation}\label{eq:NT-height}
    \langle\, ,\rangle_{A/L}\coloneqq(\log p)^{-1}\cdot\langle\, ,\rangle_{A/L,\NT}\colon A(L)\times A^t(L)\to \QQ,
\end{equation}
where $\langle\, ,\rangle_{A/L,\NT}$ is the N\'eron--Tate height pairing.
\begin{defn} 
    For any prime $\ell$ (allowing $\ell = p$), we write $h_{\ell}\colon A^t(L)\otimes\QQ_\ell \to \left(A(L)\otimes\QQ_\ell\right)^\ast$ for the $\QQ_\ell G$-isomorphism induced by $\langle\,,\rangle_{A/L}$. If in addition $\Sha(A/L)$ is finite, then we interpret $h_\ell$ as a $\QQ_\ell G$-isomorphism \[h_\ell\colon\xymatrix@1{\Hrm^0(\SC_{Z_L,\ell}\otimes_{\ZZ_\ell}\QQ_\ell)\ar[r]^-{\cong}& \Hrm^1(\SC_{Z_L,\ell}\otimes_{\ZZ_\ell}\QQ_\ell)}.\]
\end{defn}
Therefore, $h_\ell$ defines the $\QQ_\ell$-trivialisation in the sense of \eqref{eq:ht-pairing} if $\Sha(A/L)$ is finite \emph{and} $\SC_{Z_L,\ell}$ is a perfect $\ZZ_\ell G$-complex. The following proposition gives a sufficient condition for the $\ZZ_\ell G$-perfectness.
\begin{prop}\label{prop:c-t-SC} \begin{enumerate}
    \item\label{prop:c-t-SC:ell} If $\ell\ne p$ then we have $\SC_{Z_L,\ell}\in\Dperf(\ZZ_\ell G)$.
    \item\label{prop:c-t-SC:p} Suppose that $L/K$ is weakly ramified everywhere, and if $L/K$ is wildly ramified at $v\in Z$ then the natural map $\Acal^\circ\times_{X}\Spec\Ocal_w\to \Acal^\circ_L\times_{X_L}\Spec\Ocal_w$ is an isomorphism. (In other words, the condition for Corollary~\ref{cor:descent-Neron}(\ref{cor:descent-Neron:perf}) is valid.) Then we have $\whSC_{Z_L}(A,L/K)\in\Dperf(\widehat\ZZ G)$, and hence, $\SC_{Z_L,p}\in\Dperf(\ZZ_p G)$.
\end{enumerate} 
\end{prop}

Before we prove the proposition, let us make the following remark.

\begin{rmk}\label{rmk:c-t-SC}
    Proposition~\ref{prop:c-t-SC} is a special case of \cite[Proposition~3.7(i)]{BurnsKakdeKim:EquivBSDTame}, built upon the argument in \cite[\S6]{KatoTrihan:BSD}. To explain, by \cite[Proposition~3.4, Proposition~3.7(i)]{BurnsKakdeKim:EquivBSDTame} one constructs a \emph{perfect} $\widehat\ZZ G$-complex $\whSC_{V_L}$ using carefully chosen family of $G_w$-stable open compact subgroups $V_L\coloneqq (V_w\subset \Acal^\circ_L(\mfr_w))_{w\in Z_L}$, equipped with a distinguished triangle in $D(\widehat\ZZ G)$
    \[
    \xymatrix@1{ \whSC_{V_L} \ar[r] & \whSC_{Z_L} \ar[r] & \stackbin[w\in Z_L]{}{\bigoplus} \big(\Acal^\circ_L(\mfr_w)/V_w\big)^\vee[-1] \ar[r]& (+1)}.
    \]
    (In \cite[Proposition~3.7(i)]{BurnsKakdeKim:EquivBSDTame}, $\whSC_{V_L}$ is denoted as $\RGamma_{\ari,V_L}(U_L,\Acal_{L,\tors})^\vee[-2]$.) As $\Acal^\circ_L(\mfr_w)/V_w$ is a $p$-group for each $w\in Z_L$, it easily follows that the natural map $\whSC_{V_L}\otimes_{\widehat\ZZ}\ZZ_\ell\to \SC_{Z_L,\ell}$ is a quasi-isomorphism, proving Proposition~\ref{prop:c-t-SC}(\ref{prop:c-t-SC:ell}). To prove Proposition~\ref{prop:c-t-SC}(\ref{prop:c-t-SC:p}), we have to show that the choice $V_L=(\Acal^\circ_L(\mfr_w))_{w\in Z_L}$ makes $\whSC_{V_L}$ a perfect $\widehat\ZZ G$-complex under the additional assumption in the statement, which we explain in the proof.

    Without the additional assumption in Proposition~\ref{prop:c-t-SC}(\ref{prop:c-t-SC:p}), the choice of $V_L$ that make $\whSC_{V_L}$ perfect is quite \emph{inexplicit} and hard to work with. Therefore, we state Proposition~\ref{prop:c-t-SC} in a restrictive setting where there is a preferred explicit choice of $V_L$. See Remark~\ref{rmk:necessity-weak-ram} for further discussions.
\end{rmk}

\begin{proof}[Proof of Proposition~\ref{prop:c-t-SC}]
    Claim~(\ref{prop:c-t-SC:ell}) is already proven in Remark~\ref{rmk:c-t-SC}, so let us prove claim~(\ref{prop:c-t-SC:p}) in the setting of (\ref{prop:c-t-SC:p}).
    By the proof of Proposition~3.7(i) in \cite{BurnsKakdeKim:EquivBSDTame}, we need to show that the continuous $G_w$-action on $\Acal^\circ_L(\mfr_w)$ is cohomologically trivial, which is equivalent to the cohomological triviality of  $\Lie(\Acal_L)(\mfr_w)$ for $G_w$ by the proof of Lemma~6.1 and Lemma~6.2 in \cite{KatoTrihan:BSD}. (See also the proof of Lemma~3.4 in \cite{BurnsKakdeKim:EquivBSDTame}.) But by  Corollary~\ref{cor:descent-Neron}(\ref{cor:descent-Neron:perf}), we have $G_w$-equivariant isomorphisms of $\Ocal_w$-modules
    \[\Lie(\Acal_L)(\mfr_w) \cong \Lie(\Acal)(\Ocal_v)\otimes_{\Ocal_v}\mfr_w \cong \mfr_w^{\oplus \dim(A)},\]
    for any $w\in Z_L$ where $L/K$ is weakly and wildly ramified, so the desired cohomological triviality follows from K\"ock's local integral normal basis theorem (\emph{cf.} Theorem~\ref{th:int-normal-basis}).
%
\end{proof}

\begin{notn}\label{not:psi}
Let $\Ir(G)$ denote the set of isomorphism classes of (complex) irreducible $G$-representations, and choose a number field $E\subset \CC$ over which any $\psi\in\Ir(G)$ is defined. For each $\psi\in\Ir(G)$, let $V_\psi$ be the corresponding $EG$-module, and choose a $G$-stable $\Ocal_{E}$-lattice $T_\psi\subset V_\psi$. For any place $\lambda$ of $E$, we write 
\[
    V_{\psi,\lambda}\coloneqq V_\psi\otimes_{E} E_{\lambda} \quad \& \quad T_{\psi,\lambda}\coloneqq T_\psi\otimes_{\Ocal_E}\Ocal_{E,\lambda}.
\]
We obviously extend the above notation for any $G$-representation $\psi$.
\end{notn}

For any $G$-representation $\psi$, let $L_U(A,\psi,s)$ be the Hasse--Weil--Artin $L$-series for $(A,\psi)$ without Euler factors away from $U$;\footnote{Sometimes it can be convenient to allow $\psi$ to be reducible as we do (such as the regular representation), which is harmless as we have $L_U(A,\psi'\oplus\psi',s) = L_U(A,\psi',s)\cdot L_U(A,\psi'',s)$.} i.e., choosing a place $\lambda\mid\ell$ of $E$ with $\ell\ne p$ we have
\begin{equation}\label{eq:L-ftn}
    L_U(A,\psi,s) \coloneqq\prod_{v\in|U|} \det{}_{E_\lambda}\left(1 - |k_v|^{1-s}\Frob_v\mid V_\ell(A)\otimes_{\QQ_\ell} V_{\psi,\lambda}\right)^{-1},
\end{equation}
where $\Frob_v$ is the \emph{geometric} Frobenius at $v$.
It is sometimes useful to apply the change of variable $t=p^{-s}$ and set $Z_U(A,\psi,t)=L_U(A,\psi,s)$. 

Recall that by the Lefschetz trace formula we have
\[Z_U(A,\psi,t) = \prod_{i=0}^2 \det{}_{E_\lambda}\left(1-pt\cdot\Frob_p \mid \Hrm^i_{\et,c} (U\times_{\Spec\FF_p}\Spec\overline\FF_p, V_\ell(A)\otimes_{\QQ_\ell} V_{\psi,\lambda})\right)^{(-1)^{i+1}},
\]
where $\Frob_p$ is the \emph{geometric} $p$-Frobenius. Moreover,  we have $Z_U(A,\psi,t)\in E(t)$ that is independent of the choice of $\lambda$, and there is an analogous formula for $\lambda\mid p$ recovering $Z_U(A,\psi,t)$ via rigid cohomology. (For more details, see  \cite[Theorem~8.2]{BurnsKakdeKim:EquivBSDTame}.) 

For any $G$-representation $\psi$, set 
\begin{align}\label{eq:rk-LT}
    r_{\an}(\psi) &\coloneqq \ord_{s=1}L_U(A,\psi,s) = \ord_{t=p^{-1}}Z_U(A,\psi, t)\qquad\text{and}\\
    \Lscr_U(A,\psi)&\coloneqq \frac{L^\ast_U(A,\psi,1)}{(\log p)^{r_{\an}(\psi)}} = \lim_{t\to p^{-1}}\frac{Z_U(A,\psi,t)}{(1-pt)^{r_{\an}(\psi)}} \in E^\times.\notag
\end{align}

We recall the following result; \emph{cf.} \cite[Proposition~5.6]{BurnsKakdeKim:EquivBSDTame}.
\begin{prop}\label{prop:equiv-L-value}
\begin{enumerate}
    \item \label{lem:equiv-L-value:L-series}
    For any field automorphism $\tau$ of $\CC$, we have 
    \[\tau(Z_U(A,\psi,t)) = Z_U(A,\tau\circ\psi,t)\quad \text{and} \quad \tau(\Lscr_U(A,\psi)) = \Lscr_U(A,\tau\circ\psi),\] 
    where we view $Z_U(A,\psi,t)\in \CC(t)$. 
    \item \label{lem:equiv-L-value:L-value}
    There exists an element $\Lscr_U(A,L/K) \in \Krm_1(\QQ G)$ interpolating $\Lscr_U(A,\psi)$'s in the following sense: for any $G$-representation $\psi$, we have
    \[\Nrd^\psi(\Lscr_U(A,L/K)) = \Lscr_U(A,\psi).\]
\end{enumerate}
\end{prop}
\begin{proof}
    Claim~(\ref{lem:equiv-L-value:L-series}) follows from Eq~(2) in the proof of Proposition~2.2 in \cite{BurnsKakdeKim:EquivBSDTame} (or alternatively, see the proof of Proposition~5.6 in \cite{BurnsKakdeKim:EquivBSDTame}). Identifying $\zeta(\CC G) \cong \prod_{\psi\in\Ir(G)} \CC$, we set 
    \[\Lscr_U(A,L/K) \coloneqq \big(\Lscr_U(A,\psi)\big)_{\psi\in\Ir(G)}\in \zeta(\CC G)^\times.\] 
    It follows from \cite[Proposition~5.6]{BurnsKakdeKim:EquivBSDTame} that $\Lscr_U(A,L/K)\in \Krm_1(\QQ G) = \zeta(\QQ G)^\times\cap\Krm_1(\RR G)$. This element clearly satisfies the interpolation property for any $\psi\in\Ig(G)$, hence for any $G$- representation $\psi$.
 \end{proof}

For $\psi\in\Ir(G)$, the \emph{algebraic $\psi$-rank} of $A$ is defined as follows:
\begin{equation}\label{eq:psi-MW-rk}
    r_{\alg}(\psi) \coloneqq \dim_E\Big(\Hom_{EG} (V_{\psi}, E\otimes A^t(L) )\Big) .
\end{equation}
Recall the following standard result.
\begin{thm}[{\emph{cf.} \cite{KatoTrihan:BSD}, \cite[Theorem~8.2]{BurnsKakdeKim:EquivBSDTame}}]
    The following are equivalent.
    \begin{enumerate}
        \item The $\ell_0$-primary part of $\Sha(A/L)$ is finite for some prime $\ell_0$. 
        \item $\Sha(A/L)$ is finite.
        \item We have $r_{\an}(\psi) = r_{\alg}(\psi)$ for any $\psi\in\Ir(G)$.
    \end{enumerate}
\end{thm}

We conclude this section by recalling the following theorem from \cite{BurnsKakdeKim:EquivBSDTame}. For any prime $\ell$, we write $\partial^G_\ell\coloneqq \partial_{\ZZ_\ell G,\QQ_\ell}\colon \Krm_1(\QQ_\ell G)\to\Krm_0(\ZZ_\ell G,\QQ_\ell G)$ and   $\chi^G_\ell(C^\bullet,h)\coloneqq \chi_{\ZZ_\ell G,\QQ_\ell}(C^\bullet, h)\in \Krm_0(\ZZ_\ell G,\QQ_\ell G)$; \emph{cf.} \eqref{eq:conn-hom} and Def~\ref{def:Euler-Char-Rel-K0}.
\begin{thm}[{\emph{cf.}~\cite[Theorem~4.9]{BurnsKakdeKim:EquivBSDTame}}]\label{th:BKK}
    Suppose that the $\ell_0$-primary part of $\Sha(A/L)$ is finite for some prime $\ell_0$. 
    \begin{enumerate}
        \item\label{th:BKK:ell} For any prime $\ell\ne p$, the following formula 
    \[
    \partial^G_\ell\big(\Lscr_U(A,L/K)\big) - \chi^G_\ell\big(\SC_{Z_L,\ell},h_\ell\big) 
    \]
     defines a torsion element in $\Krm_0(\ZZ_\ell G,\QQ_\ell G)$.
     \item\label{th:BKK:p}   Under the assumption as in  Corollary~\ref{cor:descent-Neron}(\ref{cor:descent-Neron:perf}), the following formula
     \[  \partial^G_p\big(\Lscr_U(A,L/K)\big) - \chi^G_p\big(\SC_{Z_L,\ell},h_p\big) + \chi^G_p\big(\Lie(\Acal_L)(-Z_L)\big)  \]
    defines a torsion element in $\Krm_0(\ZZ_p G,\QQ_p G)$. Here, $\chi^G_p(\Ecal)$ for a $G$-equivariant vector bundle $\Ecal$ is defined in \eqref{eq:Euler-Char-Rel-K0-VB}. 
    \end{enumerate}
\end{thm}
\begin{proof}
This theorem essentially follows from \cite[Theorem~4.9, Proposition~5.6]{BurnsKakdeKim:EquivBSDTame}, which we now explain in details. 
Using the notation of \cite[Proposition~5.6]{BurnsKakdeKim:EquivBSDTame}, the image of $\chi^{\BSD}_{G,\QQ}(A,V_L)$ in $\Krm_0(\ZZ_\ell G,\QQ_\ell G)$ is equal to $\chi^G_\ell\big(\SC_{Z_L,\ell},h_\ell\big)$ by Proposition~\ref{prop:c-t-SC} and Remark~\ref{rmk:c-t-SC}. (We allow $\ell = p$ under the addition assumption as in the statement.) The image of $\chi^{\coh}_G(A,V_L)$ in $\Krm_0(\ZZ_\ell G,\QQ_\ell G)$ is equal to $\chi^G_\ell\big(\Lie(\Acal_L)(-Z_L)\big)$ since the vector bundle $\Lcal_L$ attached to $V_L = (\Acal^\circ_L(\mfr_w))_{w\in Z_L}$ in \cite[\S3.5]{BurnsKakdeKim:EquivBSDTame} is exactly $\Lie(\Acal_L)(-Z_L)$. (In \emph{loc.~cit.} we assumed that $V_w\subseteq \Acal^\circ_{X_L}(\mfr_w)$ due to the way we construct cohomologically trivial $V_w$'s, but the same proof can be extended \emph{verbatim} to $V_w = \Acal_L^\circ(\mfr_w)$ and $\Lcal_L = \Lie(\Acal_L)(-Z_L)$.) Lastly,  $\chi^{\sgn}_G(A)$ is $2$-torsion by definition. By \cite[Theorem~4.9]{BurnsKakdeKim:EquivBSDTame}, the formula in \cite[Proposition~5.6(ii)]{BurnsKakdeKim:EquivBSDTame} holds up to torsion, and hence the theorem follows.
\end{proof}

We apply Theorem~\ref{th:BKK} to deduce a result on the normalised leading term $\Lscr_U(A,\psi)$. To state it, let us introduce some notation, which is a slight adaptation of \S\ref{sec:K-thy}.  Fix a place $\lambda\mid\ell$ of $E$ and a $G$-representation $\psi$. Set
\[
 \rho^\psi_\lambda \colon \xymatrix@1{\Krm_0(\ZZ_\ell G,\QQ_\ell G)\ar[r]& \Krm_0(\Ocal_{E,\lambda},E_\lambda) \ar[r]^-{v_\lambda}& \ZZ},
\]
where the first map is induced by $\Afr = \ZZ_\ell G \xrightarrow{\psi}\End_{\Ocal_{E,\lambda}}(T_{\psi,\lambda})$. 

We also introduce the following integers
\[
\chi^{\BSD}_{Z_L,\lambda}(A,\psi)\coloneqq \rho^\psi_\lambda\big(\chi^G_\ell(\SC_{Z_L,\ell},h_\ell) \big)\quad \&
\quad \chi^{\coh}_{Z_L,\lambda}(A,\psi)\coloneqq \rho^\psi_\lambda \big(\chi^G_{\ell}\big(\Lie(\Acal_L)(-Z_L)\big)\big),
\]
where we set $\chi^G_{\ell}\big(\Lie(\Acal_L)(-Z_L)\big)=0$ for $\ell\ne p$.

\begin{cor}\label{cor:BKK}
    Suppose that the $\ell_0$-primary part of $\Sha(A/L)$ is finite for some prime $\ell_0$. 
    Fix a place $\lambda \mid \ell$ of $E$, and let $\psi$ be a $G$-representation. Then we have
    \[
    v_\lambda\big(\Lscr_U(A,\psi)\big) = \chi^{\BSD}_{Z_L,\lambda}(A,\psi) -  \chi^{\coh}_{Z_L,\lambda}(A,\psi)
    \]
    if either $\ell \ne p$, or $\ell = p$ and the same assumption as in Corollary~\ref{cor:descent-Neron}(\ref{cor:descent-Neron:perf}) holds.
\end{cor}
\begin{proof}
    By construction of $\Lscr_U(A,L/K)$ and Lemma~\ref{lem:rho-psi}(\ref{lem:rho-psi:conn-hom}), we have
    \[(\rho^{\psi}_\lambda\circ\partial^G_\ell)\big(\Lscr_U(A,L/K)\big) = v_\lambda\big(\Lscr_U(A,\psi)\big). \]
    The corollary can now be obtained by applying $\rho^{\psi}_\lambda$ to the formulae in Theorem~\ref{th:BKK}.
\end{proof}

\begin{rmk}\label{rmk:HS-spec}
In the above setting, $\chi^{\coh}_{Z_L,\lambda}(A,\psi)$ has already been computed in Corollary~\ref{cor:psi-part}(\ref{cor:psi-part:formula}), so to make $v_\lambda\big(\Lscr_U(A,\psi)\big)$ explicit, it remains to compute $ \chi^{\BSD}_{Z_L,\lambda}(A,\psi)$ in terms of \emph{arithmetic invariants} of $A/K$ (as integral Galois modules). This step turns out to be quite subtle, and in the next section we carry it out under certain \emph{simplifying assumptions} depending on $\ell$; \emph{cf.} Assumptions~\ref{ass:BC}. In particular, even though a more general version of Theorem~\ref{th:BKK}(\ref{th:BKK:p}) is obtained in \cite[Theorem~4.9, Proposition~5.6]{BurnsKakdeKim:EquivBSDTame} involving some inexplicit choice $V_L$ as in Remark~\ref{rmk:c-t-SC}, the resulting general formula for $v_\lambda\big(\Lscr_U(A,\psi)\big)$ seems difficult to make explicit. See Remark~\ref{rmk:necessity-weak-ram} for further discussions.
\end{rmk}

\section{The BSD-like formula for Hasse--Weil--Artin $L$-values}\label{sec:main}
Assuming the finiteness of $\Sha(A/L)$, we shall express $\chi^{\BSD}_{Z_L,\lambda}(A,\psi)$ in terms of the Galois module structure of $A(L)$, $A^t(L)$ and $\Sha(A/L)$ under a certain set of assumptions satisfied for almost all primes $\ell$ under $\lambda$, and thereby obtain the formula for $v_\lambda\big(\Lscr_U(A,\psi)\big)$. We also introduce a stronger assumption to handle $\ell = p$. We closely follow the proof in Burns--Castillo \cite[Proposition~7.3]{BurnsMaciasCastillo:EquivBSD}, which proves an analogous result over a number field.

\begin{notn}
\setcounter{equation}{\value{equation}-1}
\begin{subequations}
   For a $\ZZ_\ell G$-module $M$ we set
\begin{align}\label{eq:psi-part:coinv}
    [M]_{\psi,\lambda} &\coloneqq \Hom_{\Ocal_{E,\lambda}}(T_{\psi,\lambda},\Ocal_{E,\lambda}\otimes_{\ZZ_\ell} M)_G \cong T^\ast_{\psi,\lambda}\otimes_{\ZZ_\ell G}M, \qquad\text{and}\\
    [M]^\psi_\lambda &\coloneqq \Hom_{\Ocal_{E,\lambda}}(T_{\psi,\lambda},\Ocal_{E,\lambda}\otimes_{\ZZ_\ell} M)^G = 
    \Hom_{\Ocal_{E,\lambda}[G]}(T_{\psi,\lambda},\Ocal_{E,\lambda}\otimes_{\ZZ_\ell} M).
\end{align}
We extend these definitions to $\ZZ_\ell G$-complexes.
If $V$ is a $\QQ_\ell G$-module and $\psi\in\Ir(G)$ then we have $[V]_{\psi,\lambda}\cong [V]^\psi_\lambda$ and its $E_\lambda$-dimension is the multiplicity of $\psi$ in $V$. 

If $M$ is a finitely generated $\ZZ G$-module, then we abusively write $[M]_{\psi,\lambda}$ for $[M\otimes\ZZ_\ell]_{\psi,\lambda}$ and similarly define $[M]^\psi_\lambda$. 

Lastly, we set
\begin{equation}\label{eq:Sha-psi}
    \Sha^\vee_{\psi,\lambda}(A/L)\coloneqq \ker \Big([\Sel_{\QQ_\ell/\ZZ_\ell}(A/L)^\vee]_{\psi,\lambda} \twoheadrightarrow [(A(L)\otimes\QQ_\ell/\ZZ_\ell)^\vee]_{\psi,\lambda} \Big).
\end{equation}
To motivate the notation, note that by right exactness of $[-]_{\psi,\lambda}$ we have a natural right exact sequence
\[
\begin{tikzcd}[column sep=scriptsize]
    {[\Sha(A/L)^\vee]_{\psi,\lambda}} \arrow[r]& {[\Sel_{\QQ_\ell/\ZZ_\ell}(A/L)^\vee]_{\psi,\lambda}} \arrow[r] & {[(A(L)\otimes\QQ_\ell/\ZZ_\ell)^\vee]_{\psi,\lambda}}  \arrow[r] & 0
\end{tikzcd},
\]
and $ \Sha^\vee_{\psi,\lambda}(A/L)$ is the image of $[\Sha(A/L)^\vee]_{\psi,\lambda}$ in $[\Sel_{\QQ_\ell/\ZZ_\ell}(A/L)^\vee]_{\psi,\lambda}$. In particular, if $\ell$ is prime to $|G|$ then we have $[\Sha(A/L)^\vee]_{\psi,\lambda} \riso \Sha^\vee_{\psi,\lambda}(A/L)$. Note also that $\Sha^\vee_{\psi,\lambda}(A/L)$ is finite if $\Sha(A/L)$ is finite, in which case $\Sha^\vee_{\psi,\lambda}(A/L)$ is the torsion part of $[\Sel_{\QQ_\ell/\ZZ_\ell}(A/L)^\vee]_{\psi,\lambda}$.
\end{subequations}
\end{notn}

To compute $\chi^{\BSD}_{Z_L,\lambda}(A,\psi)$, we need to compute the cohomology of $[\SC_{Z_L,\ell}]_{\psi,\lambda}$ in terms of the arithmetic invariants of $A/L$, which naturally involves some Hochschield--Serre-type spectral sequence. To make the spectral sequence \emph{sufficiently degenerate}, we introduce the following conditions for $(A,L/K,Z_L)$ and a prime $\ell$.

\begin{ass}\label{ass:BC} 
For $(A,L/K,Z_L)$ and $\ell$ as above, suppose the following conditions hold.
\begin{enumerate}
    \item\label{ass:BC:MW} Neither $A(L)$ nor $A^t(L)$ have any non-trivial $\ell$ -torsion.
    \item\label{ass:BC:loc-vol} For any $w\in Z_L$, there is no non-trivial $\ell$-torsion in $\Acal_L(k_w)\cong A(L_w)/\Acal^\circ_L(\mfr_w)$. 
    \item\label{ass:BC:p} If $\ell = p$, then we assume that $L/K$ is weakly ramified everywhere, $A$ has semistable reduction at all places of $L$, and $L/K$ is tamely ramified at all places $v\in Z$ where $A$ does \emph{not} have semistable reduction. (Cf. Example~\ref{exa:descent-semistable}.)
\end{enumerate}
Condition~(\ref{ass:BC:loc-vol}) can be rephrased as $\ell$ dividing neither $|\Acal^\circ_L(k_w)|$ nor the local Tamagawa number $|A(L_w)/\Acal^\circ_L(\Ocal_w)|$.
\end{ass}

Given $(A,L/K,Z_L)$, Assumption~\ref{ass:BC} is clearly satisfied for all but finitely many primes $\ell$, but it is most interesting for $\ell = p$, especially when $p$ divides $[L:K]$. In \S\ref{sec:exa} we will present some non-trivial examples of $(A,L/K,Z_L)$ where Assumption~\ref{ass:BC} is satisfied for $\ell = p$ and $\Sha(A/L)$ is finite.

Condition~(\ref{ass:BC:p}) of Assumption~\ref{ass:BC} may look stronger than the assumption to ensure $\SC_{Z_L,p}\in\Dperf(\ZZ_p G)$ in Proposition~\ref{prop:c-t-SC}(\ref{prop:c-t-SC:p}), but the following lemma shows that these two conditions are equivalent under Assumption~\ref{ass:BC}(\ref{ass:BC:loc-vol}) for $\ell=p$.
\begin{lem}\label{lem:ass}
    \begin{enumerate}
        \item\label{lem:ass:sst} Assumption~\ref{ass:BC}(\ref{ass:BC:loc-vol}) for $\ell = p$ implies that $A_L$ has semistable reduction at all places of $L$.
        \item\label{lem:ass:torus}  For any torus $T$ over a finite field $k'$ of characteristic~$p$, the order of $T(k')$ is prime to $p$. In particular, if $A_L$ has totally toric degeneration at all places in $Z_L$, then Assumption~\ref{ass:BC}(\ref{ass:BC:loc-vol}) is satisfied for $\ell = p$ if $p$ does not divide $|A(L_w)/\Acal^\circ_L(\Ocal_w)|$ for any $w\in Z_L$.
    \end{enumerate}
\end{lem}
\begin{proof}
    Set $\Acal^\circ_{L,k_w}\coloneqq \Acal^\circ_L\times_{X_K}\Spec k_w$, which is a semi-abelian variety if and only if the unipotent radical $\Rscr_u(\Acal^\circ_{L,k_w})$ is trivial. Since any connected commutative unipotent algebraic group over a perfect field is a vector group (\emph{cf.} \cite[Corollary~B.2.7]{ConradGabberPrasad:PRedGp2}), $\Rscr_u(\Acal^\circ_{L,k_w})(k_w)$ is a non-trivial $p$-group whenever $\Rscr_u(\Acal^\circ_{L,k_w})$ is non-trivial. This shows Claim (\ref{lem:ass:sst}). 
    
    To prove (\ref{lem:ass:torus}), recall that for any $k'$-torus $T$ we have a short exact sequence of $k'$-tori
    \[
    1\to T'\to S \to T\to 1
    \]
    where $S = \Res_{k''/k'}\GG_m^d$ for some finite extension $k''/k'$. (This is a standard fact; see \cite[pp~8--9]{HamacherKim:Gisoc} for the proof.) It now follows that $T(k')$ is of prime-to-$p$ order since it is a quotient of $S(k') = (k''^\times)^d$ by surjectivity of the Lang isogeny. If $A_L$ has totally toric degeneration at $w\in Z_L$, then we just showed that $p \nmid |\Acal^\circ_{L}(k_w)|$ since $\Acal^\circ_{L,k_w}$ is a torus. 
\end{proof}

Let us now record the effect of Assumption~\ref{ass:BC} on the cohomology of $[\SC_{Z_L,\ell}]_{\psi,\lambda}$.
\begin{lem}\label{lem:SC-two-term}
    Suppose that the $\ell_0$-primary part of $\Sha(A/L)$ is finite for some $\ell_0$, and Assumption~\ref{ass:BC} is satisfied for $\ell$. 
    Then $\SC_{Z_L,\ell}$ can be represented by a two-term complex $[P^0\xrightarrow{d} P^1]$ of finitely generated projective $\ZZ_\ell G$-modules concentrated in degrees $[0,1]$. Furthermore, the following properties are valid for any $G$-representation $\psi$.
    \begin{enumerate}
        \item\label{lem:SC-two-term:0} $\Hrm^0([\SC_{Z_L,\ell}]_{\psi,\lambda}) \cong [A^t(L)]^\psi_\lambda$, which is torsion-free.
        \item\label{lem:SC-two-term:1} $\Hrm^1([\SC_{Z_L,\ell}]_{\psi,\lambda}) \cong [\Sel_{\QQ_\ell/\ZZ_\ell}(A/L)^\vee]_{\psi,\lambda}$,
         whose  torsion part and maximal torsion-free quotient are respectively $\Sha^\vee_{\psi,\lambda}(A/L)$ and $ [(A(L)\otimes\QQ_\ell/\ZZ_\ell)^\vee]_{\psi,\lambda}$; \emph{cf.} \eqref{eq:Sha-psi}.
    \end{enumerate}
\end{lem}

\begin{proof} \emph{(Compare with the proof of Proposition~7.3(ii) in \cite{BurnsMaciasCastillo:EquivBSD}.)}
    By Proposition~\ref{prop:SC} and Assumption~\ref{ass:BC} for $\ell$, we have $\Hrm^0(\SC_{Z_L,\ell}) \cong A^t(L)\otimes\ZZ_\ell$, which is torsion-free, and $\Hrm^i(\SC_{Z_L,\ell})=0$ for $i\ne 0,1$. The $\ZZ_\ell G$-perfectness (\emph{cf.} Prop.~\ref{prop:c-t-SC}) now implies that $\SC_{Z_L,\ell}$ can be represented by a two-term perfect $\ZZ_\ell G$-complex $[P^0\xrightarrow{d} P^1]$.

    As $T_{\psi,\lambda}^\ast\otimes_{\ZZ_\ell} P^i$ is also a cohomologically trivial $\Ocal_{E,\lambda}G$-module, the norm map induces a natural $\Ocal_{E,\lambda}$-linear isomorphism $N_G\colon [P^i]_{\psi,\lambda} \riso [P^i]^\psi_\lambda$ for $i=0,1$. Therefore, we have the following commutative diagram with exact rows
    \begin{equation}\label{eq:SC-two-term}
        \xymatrix@C=1.5em{
        &&[P^0]_{\psi,\lambda} \ar[rr]^-{[d]_{\psi,\lambda}} \ar[d]_-{N_G}^-{\cong} & &  [P^1]_{\psi,\lambda} \ar[r] \ar[d]_-{N_G}^-{\cong}& [\Hrm^1]_{\psi,\lambda} \ar[r] &0\\
        0\ar[r]& [\Hrm^0]^\psi_\lambda\ar[r]& [P^0]^\psi_\lambda\ar[rr]_-{[d]^\psi_\lambda} & &[P^1]^\psi_\lambda& & & &
        },
    \end{equation}
    where $\Hrm^i\coloneqq \Hrm^i(\SC_{Z_L,\ell})$. Now the lemma follows, noting that $[(A(L)\otimes\QQ_\ell/\ZZ_\ell)^\vee]_{\psi,\lambda} \cong \big([A(L)\otimes\ZZ_\ell]^{\check\psi}_\lambda\big)^\ast$ is torsion-free.
\end{proof}

\begin{rmk}\label{rmk:norm-triv}
    For any $E_\lambda G$-module $V_\lambda$ and $\psi\in \Ir(G)$, we have an isomorphism
\begin{equation}\label{eq:isom-inv-coinv}
    \begin{tikzcd}
         (V_{\psi,\lambda}^\ast\otimes_{E_\lambda}V_{\lambda})^G  \arrow[hook]{r} \arrow[bend left=15,"\cong"]{rr} &  V_{\psi,\lambda}^\ast\otimes_{E_\lambda}V_{\lambda} \arrow[two heads]{r} & (V_{\psi,\lambda}^\ast\otimes_{E_\lambda}V_{\lambda})_G,
   \end{tikzcd}
\end{equation}   
where the first map is the natural inclusion and the second the natural projection. Then the norm map $N_G\colon V_\lambda\to V_\lambda$ induces 
    \[
    \begin{tikzcd}
        (V_{\psi,\lambda}^\ast\otimes_{E_\lambda}V_{\lambda})_G \arrow["N_G"']{dr} & & \arrow["\cong", "\text{\eqref{eq:isom-inv-coinv}}"']{ll}   (V_{\psi,\lambda}^\ast\otimes_{E_\lambda}V_{\lambda})^G  \arrow["|G|"]{dl}\\
       &  (V_{\psi,\lambda}^\ast\otimes_{E_\lambda}V_{\lambda})^G , & 
    \end{tikzcd}
    \]
    where the right diagonal map is multiplication by $|G|$.

    Applying this observation to $V_\lambda = \Hrm^i(\SC_{Z_L,\ell})\otimes_{\ZZ_\ell}E_\lambda$, we obtain the following commutative diagram
    \begin{equation}\label{eq:norm-triv}
            \begin{tikzcd}
       \Hrm^i([\SC_{Z_L,\ell}]_{\psi,\lambda}) \arrow["\cong"',"N_G"]{d} \arrow[hook]{r} &  
       \Hrm^i([\SC_{Z_L,\ell}]_{\psi,\lambda})\otimes E_\lambda   \arrow["N_G","\cong"']{d} \arrow["\cong"', "\text{\eqref{eq:isom-inv-coinv}}"]{r} &  \Hrm^i([\SC_{Z_L,\ell}]^\psi_\lambda) \otimes E_\lambda \arrow["|G|"]{dl}\\
        \Hrm^i([\SC_{Z_L,\ell}]^\psi_\lambda) \arrow[hook]{r} &  
         \Hrm^i([\SC_{Z_L,\ell}]^\psi_\lambda) \otimes E_\lambda. & 
    \end{tikzcd}
    \end{equation}
    where the left vertical isomorphism is induced by the isomorphism
    \[
    \begin{tikzcd}
        {[\SC_{Z_L,\ell}]}_{\psi,\lambda} \arrow["\cong"', "N_G"]{r} & {[\SC_{Z_L,\ell}]}^\psi_\lambda
    \end{tikzcd}
    \]
     given by \eqref{eq:SC-two-term}. For $i=0$ the left horizontal arrow in \eqref{eq:norm-triv} coincides with the isomorphism 
     \[N_G\colon 
     \begin{tikzcd}
         \Hrm^0([\SC_{Z_L,\ell}]_{\psi,\lambda}) \arrow["\cong"]{r} & {[A^t(L)]}^\psi_\lambda
     \end{tikzcd}
     \]
     in Lemma~\ref{lem:SC-two-term}(\ref{lem:SC-two-term:0}).
     We use this observation in the computation of $\chi^{\BSD}_{Z_L,\lambda}(A,\psi)$; \emph{cf.} Proposition~\ref{prop:BC}. 
\end{rmk}

We now introduce the $\psi$-twisted regulator, following \cite[\S7.2.2]{BurnsMaciasCastillo:EquivBSD}.
\begin{defn}\label{def:psi-ht}
    We maintain the setting of Lemma~\ref{lem:SC-two-term}, and fix a place $\lambda\mid\ell$ of $E$. Given $\psi\in\Ir(G)$, choose $\Ocal_{E,\lambda}$-bases $(e_i)_{i = 1,\cdots,r_{\alg}(\psi)}$ of $[A(L)]^{\check\psi}_\lambda$, and $(\check e_j)_{j = 1,\cdots,r_{\alg}(\psi)}$ $[A^t(L)]^{\psi}_\lambda$, respectively. (We refer to \S\ref{not:psi} for the abuse of notation $[M]_{\psi,\lambda}$ and $[M]^\psi_\lambda$ when $M$ is a finitely generated $\ZZ G$-module.)
    
    we define the \emph{$\psi$-twisted regulator} to be 
    \[\Reg^\psi_\lambda\coloneqq \det\left( \langle e_i,\check  e_j \rangle_{A/L}\right).\]
    Note that $\Reg^\psi_\lambda$ is independent of the choice of $\Ocal_{E,\lambda}$-bases only up to $\Ocal_{E,\lambda}^\times$-multiple, so $v_\lambda(\Reg^\psi_\lambda)$ is a well-defined integer.

    Consider an $E_\lambda$-linear isomorphism  
    \[h^\psi\coloneqq [A^t(L)]^\psi_\lambda \otimes E_\lambda \riso \big([A(L)]^{\check\psi}_\lambda\big)^\ast \otimes E_\lambda \cong [A(L)^\ast]_{\psi,\lambda}\otimes E_\lambda\]
    by sending $\check e_j$ to the functional $\langle -,\check e_j\rangle_{A/L}$. If the $\ell$-primary part of $\Sha(A/L)$ is finite, then we can interpret $h^\psi$ as an $E_\lambda$-trivialisation of $[\SC_{Z_L,\ell}]_{\psi,\lambda}$ by Lemma~\ref{lem:SC-two-term}.
\end{defn}

\begin{subequations}
\setcounter{equation}{\value{equation}-1}
\begin{prop}\label{prop:BC}
    Suppose that the $\ell_0$-primary part of $\Sha(A/L)$ is finite for some $\ell_0$, and fix a place $\lambda$ of $E$ above a prime number $\ell$ that satisfies Assumption~\ref{ass:BC}. Then for any  $\psi\in\Ir(G)$, we have 
    \[
        \chi^{\BSD}_{Z_L,\lambda}(A,\psi) = v_\lambda\big(\Reg^\psi_\lambda/|G|^{r_{\alg}(\psi)}\big)+\lth_{\Ocal_{E,\lambda}}\big(\Sha^\vee_{\psi,\lambda}(A/L)\big).
    \]
\end{prop}
\begin{proof}
    Recall that $\chi^{\BSD}_{Z_L,\lambda}(A,\psi) = v_\lambda\big( \chi_{\Ocal_{E,\lambda},E_\lambda}([\SC_{Z_L,\ell}]_{\psi,\lambda},[h_\ell]_{\psi,\lambda}) \big)$. So we proceed by making explicit $[\SC_{Z_L,\ell}]_{\psi,\lambda}$ and $[h_\ell]_{\psi,\lambda}$.
    
    By Lemma~\ref{lem:SC-two-term} we represent $\SC_{Z_L,\ell} \cong [P^0\xrightarrow{d}P^1]$,
    and we have 
    \[[\SC_{Z_L,\ell}]_{\psi,\lambda}\cong \bigg[\xymatrix@1{[P^0]_{\psi,\lambda} \ar[r]^-{d_\psi} & [P^1]_{\psi,\lambda}}\bigg],\]
    where $d_\psi = [d]_{\psi,\lambda}$.     
    Write $\Hrm^i_\psi\coloneqq \Hrm^i([\SC_{Z_L,\ell}]_{\psi,\lambda})$, and set $\Hrm^{1}_{\psi,\tf}$ to be the maximal torsion-free quotient of $\Hrm^1_\psi$. Choose a decomposition
    \[
    [P^0]_{\psi,\lambda}  = \Hrm^0_\psi \oplus\, Q^0_\psi \quad \text{and} \quad [P^1]_{\psi,\lambda}  = \Hrm^1_\tf \oplus\, Q^1_\psi, 
    \]
    so that $d_\psi$ factorises as follows:
    \[d_\psi\colon
    \begin{tikzcd}
        {[P^0]}_{\psi,\lambda} \arrow[two heads]{r} & \, Q^0_\psi\,  \arrow["d_{Q_\psi}"]{r}& \,Q^1_\psi\, \arrow[hook]{r} & {[P^1]}_{\psi,\lambda},
    \end{tikzcd}\]
    where $d_Q$ is an injective map with $\coker(d_{Q_\psi})\cong \Sha^\vee_{\psi,\lambda}(A/L)$. Now, we can express $[h_\ell]_{\psi,\lambda}$ as follows
    \[
    [h_\ell]_{\psi,\lambda}\colon 
    \begin{tikzcd}
        \Hrm^0_{\psi,E_\lambda}\oplus\, Q^0_{\psi,E_\lambda} \arrow["{(\tilde h_\psi,d_{Q_\psi})}"]{rr} & &\Hrm^1_{\psi,E_\lambda}\oplus\, Q^1_{\psi,E_\lambda}
    \end{tikzcd}
    \]
    for some $E_\lambda$-isomorphism $\tilde h_\psi\colon \Hrm^0_{\psi,E_\lambda} \riso \Hrm^1_{\psi,E_\lambda}$, where the subscript $E_\lambda$ stands for the scalar extension. Therefore, we have
    \[
    \chi^{\BSD}_{Z_L,\lambda}(A,\psi) = [\Hrm^0_\psi, \tilde h_\psi, \Hrm^1_{\psi,\tf}] + [Q^0_\psi,d_{Q_\psi},Q^1_\psi].
    \]
    
    Recall that by Example~\ref{exa:Euler-char-tor} we have 
    \[v_\lambda([Q^0_\psi, d_{Q_\psi},Q^1_\psi]) = \lth_{\Ocal_{E,\lambda}}\big(\coker(d_{Q_\lambda})\big) = \lth_{\Ocal_{E,\lambda}}\big(\Sha^\vee_{\psi,\lambda}(A/L) \big),\] 
    so to prove the proposition it remains to compute $v_\lambda([\Hrm^0_\psi, \tilde h_\psi, \Hrm^1_{\psi,\tf}])$. By Lemma~\ref{lem:SC-two-term} and Remark~\ref{rmk:norm-triv}, we have the following commutative diagram of isomorphisms
    \[
    \begin{tikzcd}
       {[A^t(L)]}^\psi_\lambda \otimes E_\lambda \arrow["|G|"']{dr} 
       & \Hrm^0_{\psi, E_\lambda}   \arrow["N_G"]{d} \arrow["\text{\eqref{eq:isom-inv-coinv}}"']{l} \arrow["\tilde h_\psi"]{r} &  \Hrm^1_{\psi, E_\lambda} \arrow["\cong"]{d}\\
       &  {[A^t(L)]}^\psi_\lambda \otimes E_\lambda \arrow["h_\psi"]{r} &  {[A(L)^\ast]}_{\psi,\lambda}\otimes E_\lambda.
    \end{tikzcd}
    \]
    Since \eqref{eq:isom-inv-coinv} identifies $\Hrm^0_\psi$ with $|G|^{-1}\cdot [A^t(L)]_{\psi,\lambda}$ in $[A^t(L)]_{\psi,\lambda}\otimes E_\lambda$, we have
    \begin{align*}
        \big[\Hrm^0_\psi, \tilde h_\psi, \Hrm^1_{\psi,\tf}\big] 
        & = \big[ [A^t(L)]^\psi_\lambda,\, |G|^{-1}\cdot h_\psi,\, [A(L)^\ast]_{\psi,\lambda} \big] \\
        & = \big[ \Ocal_{E,\lambda},\, |G|^{-r_\alg(\psi)}\cdot \Reg^\psi_\lambda,\, \Ocal_{E,\lambda}\big],
    \end{align*}
    where the second equality uses the choice of $\Ocal_{E,\lambda}$-bases as in Def~\ref{def:psi-ht}.
\end{proof}    
\end{subequations}

For any finite torsion $\Ocal_{E,\lambda}$-module $M_\lambda$, we let $\Char_\lambda(M_\lambda)$ denote the characteristic ideal of $M_\lambda$. One can show that
\begin{equation}\label{eq:char-ideal}
    \Char_\lambda(M_\lambda) = \pfr_\lambda^{\lth_{\Ocal_{E,\lambda}}(M_\lambda)},
\end{equation}
where $\pfr_\lambda$ is the maximal ideal of $\Ocal_{E,\lambda}$.

Let us write $Z = Z_1\sqcup Z_2$ where $Z_2$ is the reduced complement of $U'$ defined in Proposition~\ref{prop:descent-conn-Neron-models}; i.e., $v\notin Z_2$ if and only if $\Acal^\circ_{X_L}$ and $\Acal^\circ_L$ are isomorphic at $v$. If $A$ has semistable reduction at all places of $L$ then $Z_2$ is the set of places of $K$ where $A$ has non-semistable reduction.  

We are now ready to state our main result.  
\begin{thm}\label{th:main} 
\setcounter{equation}{\value{equation}-1}
\begin{subequations}
  Fix a place $\lambda$ of $E$ above a prime number $\ell$, and suppose that the same assumption as in Corollary~\ref{cor:BKK} is valid (that is,  we assume that the $\ell_0$-primary part of $\Sha(A/L)$ is finite for some $\ell_0$, and if $\ell = p$ then we assume that $L/K$ is weakly ramified everywhere and tamely ramified over $Z_2$). Then for any  $\psi\in\Ir(G)$, we have the following equality of fractional ideals:
    \begin{equation}\label{eq:main:trivial}
        \Lscr_U(A,\psi)\cdot \Ocal_{E,\lambda} = 
        \left(\frac{\vol_{Z_1}(A/K)}{\prod_{v\in Z_2}\big|\Lie(\Acal_L)(k_{\tilde v})^{G_{\tilde v}} \big|} \right)^{\deg\psi}\cdot \lo_{Z_L}(A,\psi)\cdot \pfr_\lambda^{\chi^{\BSD}_{Z_L,\lambda}(A,\psi)}
    \end{equation}
    where we choose $\tilde v\in \pi^{-1}(v)$ for each $v\in Z_2$. Here, $\lo_{Z_L}(A,\psi),\vol_{Z_1}(A/K)\in p^\ZZ$ are respectively defined by 
    \begin{align}
        \log_p\big(\lo_{Z_L}(A,\psi)\big) &\coloneqq \ra_{\Lie(\Acal_L)(-Z_L)}(\psi)\quad\text{and} \label{eq:main:loc}\\
        \vol_{Z_1}(A/K)&\coloneqq \bmu\big(\Lie(A)(\AA_K)/\Lie(A)(K)\big)^{-1}\cdot \prod_{v\in Z_1}\bmu_v\big(\Acal^\circ(\mfr_v)\big) \label{eq:main:volume}
    \end{align}
    with respect to the Haar measure $\bmu_v$ on $\Lie(A)(K_v)$ and $\bmu\coloneqq\stackbin[v]{}{\prod}\bmu_v$ as in \cite[\S1.6, \S1.7]{KatoTrihan:BSD}.
    
    In particular, if Assumption~\ref{ass:BC} is valid for $\ell$ then we have
    \begin{multline}\label{eq:main:formula}
        \Lscr_U(A,\psi)\cdot \Ocal_{E,\lambda} = \\
        \left(\frac{\vol_{Z_1}(A/K)}{\prod_{v\in Z_2}\big|\Lie(\Acal_L)(k_{\tilde v})^{G_{\tilde v}} \big|} \right)^{\deg\psi}\cdot \lo_{Z_L}(A,\psi)\cdot \frac{\Reg^\psi_\lambda}{|G|^{r_{\alg}(\psi)}}
        \cdot\Char_{\lambda}(\Sha^\vee_{\psi,\lambda}(A/L)).
    \end{multline}
\end{subequations}
\end{thm}
\begin{proof}
    By Corollary~\ref{cor:BKK} we have $ \Lscr_U(A,\psi)\cdot \Ocal_{E,\lambda} = \pfr_\lambda^{\chi^{\BSD}_{Z_L,\lambda}(A,\psi)-\chi^{\coh}_{Z_L,\lambda}(A,\psi)}$. Since we have $\vol_{Z_1}(A/K) = \frac{|\Hrm^0(X,\Lie\Acal(-Z_1))|}{|\Hrm^1(X,\Lie\Acal(-Z_1))|}$ by \cite[\S3.7]{KatoTrihan:BSD}, it follows from Corollary~\ref{cor:psi-part} and Corollary~\ref{cor:descent-Neron}(\ref{cor:descent-Neron:deg}) that 
    \[-\chi^{\coh}_{Z_L,\lambda}\big(A,\psi\big)  = v_\lambda\left(\left(\frac{\vol_{Z_1}(A/K)}{\prod_{v\in Z_2}\big|\Lie(\Acal_L)(k_{\tilde v})^{G_{\tilde v}} \big|} \right)^{\deg\psi}\cdot \lo_{Z_L}(A,\psi)\right)\]
    for any place $\lambda$ of $E$ over $\ell$. If Assumption~\ref{ass:BC} is valid for $\ell$ then the formula \eqref{eq:main:formula}  immediately follows from the computation of $\chi^{\BSD}_{Z_L,\lambda}(A,\psi)$ in Proposition~\ref{prop:BC}.
\end{proof}

If $\ell$ does not divide $|G|$ then any $\ZZ_\ell G$-module is cohomologically trivial so the functor $[-]_{\psi,\lambda}$ is exact for any place $\lambda$ above $\ell$. In particular, we have 
\[
\Hrm^i\big([\SC_{Z_L,\ell}]_{\psi,\lambda}\big) \cong [\Hrm^i(\SC_{Z_L,\ell})]_{\psi,\lambda}\quad\text{and}\quad  \Sha^\vee_{\psi,\lambda}(A/L)\coloneqq [\Sha(A/L)^\vee]_{\psi,\lambda}.
\]
Therefore, we immediately obtain the following proposition even when Assumption~\ref{ass:BC} does not hold for $\ell$.
\begin{prop}\label{prop:main-c-t}
\setcounter{equation}{\value{equation}-1}
\begin{subequations}
   Suppose that the $\ell_0$-primary part of $\Sha(A/L)$ is finite for some $\ell_0$, and we fix a place $\lambda$ of $E$ over a prime $\ell$ not dividing $|G|$. Then for any  $\psi\in\Ir(G)$, we have 
    \begin{multline}\label{eq:main-c-t:BC}
    \chi^{\BSD}_{Z_L,\lambda}(A,\psi) = v_\lambda\big(\Reg^\psi_\lambda/|G|^{r_{\alg}(\psi)}\big)+\lth_{\Ocal_{E,\lambda}}\big(\Sha^\vee_{\psi,\lambda}(A/L)\big)  \\ - \lth_{\Ocal_{E,\lambda}}\Big([A(L)^\vee_{\tors}]_{\psi,\lambda}\Big)-\lth_{\Ocal_{E,\lambda}}\Big([A^t(L)_{\tors}]_{\psi,\lambda}\Big) \\+\sum_{v\in Z}\lth_{\Ocal_{E,\lambda}}\Big(\big[\bigoplus_{w \mid v}\Acal_L(k_w)^\vee\big]_{\psi,\lambda}\Big).
\end{multline}
    Furthermore, we have 
    \begin{multline}\label{eq:main-c-t:formula}
        \Lscr_U(A,\psi)\cdot \Ocal_{E,\lambda} = 
        \left(\frac{\vol_{Z_1}(A/K)}{\prod_{v\in Z_2}\big|\Lie(\Acal_L)(k_{\tilde v})^{G_{\tilde v}} \big|} \right)^{\deg\psi}\cdot \lo_{Z_L}(A,\psi)\cdot \frac{\Reg^\psi_\lambda}{|G|^{r_{\alg}(\psi)}}
        \\
        \times\frac{\Char_{\lambda}(\Sha^\vee_{\psi,\lambda}(A/L))\cdot\prod_{v\in Z}\Char_\lambda\Big(\big[\bigoplus_{w \mid v}\Acal_L(k_w)^\vee\big]_{\psi,\lambda}\Big)}{\Char_\lambda([A(L)^\vee_{\tors}]_{\psi,\lambda})\cdot\Char_\lambda([A^t(L)_{\tors}]_{\psi,\lambda})},
    \end{multline}
\end{subequations}
\end{prop}
\begin{proof}
    The formula \eqref{eq:main-c-t:BC} is immediate from Proposition~\ref{prop:SC} by the exactness of $[-]_{\psi,\lambda}$, and \eqref{eq:main-c-t:formula} follows from \eqref{eq:main:trivial} and \eqref{eq:main-c-t:BC}. Note that if $p$ does not divide $|G|$ then $L/K$ is tame at all places so \eqref{eq:main:trivial} holds for any place $\lambda\mid p$ of $E$.
\end{proof}

\begin{rmk}\label{rmk:main-thm-simplified}
    We can make $\lo_{Z_L}(A,\psi)$ more explicit in some cases; \emph{cf.} Remark~\ref{rmk:ra-psi}. For example, if $L/K$ is either a $p$-extension or unramified everywhere, then $\lo_{Z_L}(A,\psi) = 1$ for any $\psi\in\Ir(G)$. If $L/K$ is cyclic and tame everywhere $A$ has semistable reduction at all places of $K$, then one gets a simpler formula for $\lo_{Z_L}(A,\psi)$. If $A$ has semistable reduction at all places of $L$ but admits non-semistable reduction at some place of $K$ (so $Z_2\ne\emptyset$), then we need to compute $r_{\tilde v,i}$'s as in \eqref{eq:rw} for a preimage $\tilde v\in\pi^{-1}(v)$ of each $v\in Z_2$. In principle, $\lo_{Z_L}(A,\psi)$ should be computable in any explicit examples. 
\end{rmk}

Let us make a few remarks on the formulae in Theorem~\ref{th:main} and Proposition~\ref{prop:main-c-t}.    
\begin{rmk}\label{rmk:vol}
\setcounter{equation}{\value{equation}-1}
\begin{subequations}
    In the proof of Theorem~\ref{th:main} we used the interpretation of $\vol_{Z_1}(A/K)$ in terms of the Euler characteristic of $\Lie(\Acal)(-Z_1)$, so we have
    \begin{multline}\label{rmk:vol:EC}
    \log_{|k|}\big(\vol_{Z_1}(A/K)\big) = \\ \dim_k\chi_k\big(\Lie\Acal(-Z_1)\big)=\dim(A)\cdot \big(1-\gen_K-\deg(Z_1)\big)+\deg(\Lie\Acal),
    \end{multline}
    where $\gen_K$ is the genus of $X$. The same equality holds for $\vol_Z(A/K)$ with $Z$ in place of $Z_1$. If $A$ is an elliptic curve, then we have $\deg(\Lie\Acal) = -\deg(\Delta)/12$ where $\Delta$ is the global discriminant; \emph{cf.} \cite[p~325, eq~(9)]{Tan:RefindBSD}.

    Next, let us show
    \begin{equation}\label{eq:1-BSD:vol}
        \frac{\vol_{Z_1}(A/K)}{{\stackbin[v\in Z_2]{}{\prod}\big|\Lie(\Acal_L)(k_{\tilde v})^{G_{\tilde v}} \big|}} = \bmu\bigg(\frac{\Lie(A)(\AA_K)}{\Lie(A)(K)}\bigg)^{-1}\cdot \prod_{v\in Z}\bmu_v\left(\Acal^\circ_L(\mfr_{\tilde v})^{G_{\tilde v}}\right)
    \end{equation}
    for the Haar measure as in \cite[\S1.6, \S1.7]{KatoTrihan:BSD}; in other words, this expression roughly measures the volume of $\big(\prod_{w\in Z_L}\Acal^\circ_L(\mfr_{w})\big)^G$. 
    In fact, observe that
    \begin{equation}\label{eq:vol:easy}
    \frac{\vol_{Z_1}(A/K)}{\prod_{v\in Z_2}\big|\Lie(\Acal_L)(k_{\tilde v})^{G_{\tilde v}} \big|} = \vol_{Z}(A/K)\prod_{v\in Z_2}\cdot\frac{\big|\Lie(\Acal)(k_v) \big|}{\big|\Lie(\Acal_L)(k_{\tilde v})^{G_{\tilde v}} \big|}.
    \end{equation}
    Now, for any place $v$ of $K$ where $L/K$ is tamely ramified at worst,  Proposition~\ref{prop:Unip-Rad-Descent} yields
    \begin{equation}\label{eq:1-BSD:index}
           \frac{\big|\Lie(\Acal)(k_v) \big|}{\big|\Lie(\Acal_L)(k_{\tilde v})^{G_{\tilde v}} \big|} = \big|\big(\Lie(\Frm^1\Bcal_{k_v})(k_v)\big)^{G_{\tilde v}}\big| = \big|\big(\Frm^1\Bcal_{k_v}(k_v)\big)^{G_{\tilde v}}\big| = \frac{\big|\Acal(k_v) \big|}{\big|\big(\Acal_L(k_{\tilde v})\big)^{G_{\tilde v}} \big|},
   \end{equation}
   where the first and last equalities follow from the short exact sequences induced by \eqref{eq:Unip-Rad-Descent:Group} on the Lie algebras and $k_v$-rational points respectively, and the second equality holds since $(\Frm^1\Bcal_{k_v})^{G_{\tilde v}}$ is a vector group; \emph{cf.} Proposition~\ref{prop:Unip-Rad-Descent}. Note that \eqref{eq:1-BSD:index} applies to any $v\in Z_2$ in the setting where the formula \eqref{eq:main:formula} or \eqref{eq:main-c-t:formula} can be applied. If $v\notin Z_2$ then the left-most ratio in \eqref{eq:1-BSD:index} is equal to $1$ by Proposition~\ref{prop:descent-conn-Neron-models}. From this observation together with \eqref{eq:vol:easy} and \eqref{eq:1-BSD:index}, we get the following equality
      \begin{align}\label{eq:vol:formula}
       \vol_Z(A/K)\cdot\prod_{v\in Z}\big|\Acal(k_v) \big| &=  \frac{\vol_{Z_1}(A/K)}{\stackbin[v\in Z_2]{}{\prod}\big|\Lie(\Acal_L)(k_{\tilde v})^{G_{\tilde v}} \big|}\cdot \prod_{v\in Z_1}\big|\Acal(k_v) \big|\cdot\prod_{v\in Z_2}\big|\Acal_L(k_{\tilde v})^{G_{\tilde v}} \big| \\
       &=   \frac{\vol_{Z_1}(A/K)}{\stackbin[v\in Z_2]{}{\prod}\big|\Lie(\Acal_L)(k_{\tilde v})^{G_{\tilde v}} \big|}\cdot \prod_{v\in Z} \left|\frac{A(K_v)}{\Acal^\circ_L(\mfr_{\tilde v})^{G_{\tilde v}}} \right|; \notag
   \end{align}
   To see the last equality note that $\Acal^\circ_L(\mfr_{\tilde v})^{G_{\tilde v}} = \Acal(\mfr_v)$ for $v\in Z_1$, as $\Acal^\circ_{L,\Ocal_{\tilde v}}$ is the base change of $\Acal^\circ_{\Ocal_v}$. Now the formula \eqref{eq:1-BSD:vol} follows.
    \end{subequations}
\end{rmk}

\begin{rmk}\label{rmk:1-BSD}
\setcounter{equation}{\value{equation}-1}
\begin{subequations}
    Let us compare the formula \eqref{eq:main:formula} for $\Lscr_U(A,\triv_G)\cdot\ZZ_p$ (with $\triv_G$ denoting the trivial character of $G$)
    with the $p$-part of the classical BSD formula \cite[(1.8.1)]{KatoTrihan:BSD} when $p$ satisfies Assumption~\ref{ass:BC} for $(A,L/K,Z_L)$. Since for any $P\in A(K)$ and $P\in A^t(K)$ we have 
    \begin{equation}
        \langle P, \check P\rangle_{A/L} = |G|\cdot\langle P, \check P\rangle_{A/K},
    \end{equation}
     the discriminant of $\langle\,,\rangle_{A/K}$ coincides with $\Reg^{\triv}_p/|G|^{r_{\alg}(\triv)}$ up to $\ZZ_p^\times$-multiple.  
    
    Recall that $[\SC_{Z_L,p}(A,L/K)]_{\triv_G,p}\cong \big(\SC_{Z_L,p}(A,L/K)\big)_G$, and we have the following distinguished triangle
    \begin{equation}\label{eq:1-BSD:SC-descent}
        \begin{tikzcd}[column sep = small]
    \SC_{Z,p}(A, K/K) \arrow[r] & \big(\SC_{Z_L,p}(A,L/K)\big)_G \arrow{r} &  \bigoplus_{v\in Z} \left(\frac{\Acal_L^\circ(\mfr_{\tilde v})^{G_{\tilde v}}}{\Acal^\circ(\mfr_v)}\right)^\vee{[-1]} \arrow{r} & +1        
    \end{tikzcd},
    \end{equation}
    where we choose $\tilde v\in\pi^{-1}(v)$ for each $v\in Z$. In fact, by \cite[Lemma~6.1]{KatoTrihan:BSD} and its proof $\big(\big(\SC_{Z_L,p}(A,L/K)\big)_G\big)^\vee[-2]$ is the mapping fibre of
    \[
        \begin{tikzcd}[column sep = small, font = \small]
        \RGamma_{\fl}(U, \Acal{[p^\infty]})\oplus \left( \stackbin[v\in Z]{}{\bigoplus} \Acal_L^\circ(\mfr_{\tilde v})^{G_{\tilde v}}\stackbin{\Lrm}{\otimes}\QQ_p/\ZZ_p\right)[-1] \arrow{r}& \stackbin[v\in Z]{}{\bigoplus}  \RGamma_{\fl} (\Spec K_v, \Acal{[p^\infty]})
    \end{tikzcd},
    \]
    so by comparing it with the $p$-primary part of \eqref{eq:arith-coho} for $L=K$ we obtain \eqref{eq:1-BSD:SC-descent}.

    Next we turn to the index $[\Acal^\circ_L(\mfr_{\tilde v})^{G_{\tilde v}}:\Acal^\circ(\mfr_v)]$ for each $v\in Z$. If $v\in Z_1$ then this index is $1$ as explained below \eqref{eq:vol:formula}. Since $L/K$ is tame at all places $v\in Z_2$ in the setting where the formula \eqref{eq:main:formula} or \eqref{eq:main-c-t:formula} can be applied, we have $\frac{\Acal_L^\circ(\mfr_{\tilde v})^{G_{\tilde v}}}{\Acal^\circ(\mfr_v)} \cong \Frm^1\Bcal_{k_v}(k_v)^{G_{\tilde v}}$  for $v\in Z_2$ by the short exact sequence 
    \[
    \begin{tikzcd}[column sep = scriptsize]
        0 \arrow{r} & \Frm^1\Bcal_{k_v}(k_v)^{G_{\tilde v}} \arrow{r} & \Acal(k_v)\arrow{r} & \Acal_L(k_{\tilde v})^{G_{\tilde v}} \arrow{r} &0
    \end{tikzcd}
    \]
    induced by $k_v$-points of \eqref{eq:Unip-Rad-Descent:Group}. Therefore we get
    \begin{equation}
        \prod_{v\in Z}[\Acal^\circ_L(\mfr_{\tilde v})^{G_{\tilde v}}:\Acal^\circ(\mfr_v)] = \prod_{v\in Z_2}[\Acal(k_v):\Acal_L(k_{\tilde v})^{G_{\tilde v}}].
    \end{equation}
    From this together with \eqref{eq:1-BSD:SC-descent} and \eqref{eq:vol:formula}, the formula \eqref{eq:main:formula} for $\Lscr_U(A,\triv_G)\cdot\ZZ_p$ can be reduced to the $p$-part of the the classical BSD formula \cite[(1.8.1)]{KatoTrihan:BSD}. 
\end{subequations}
\end{rmk}


\begin{rmk}\label{rmk:necessity-weak-ram}
    This remark is a continuation of Remark~\ref{rmk:HS-spec}.
    Even if $L/K$ is \emph{not} weakly ramified at some place, one can still obtain a formula for $v_\lambda\big(\Lscr_U(A,\psi)\big)$ at $\lambda\mid p$ analogous to Corollary~\ref{cor:BKK}, at the cost of replacing $\big(\Acal^\circ_L(\mfr_w)\big)_{w\in Z_L}$ with some (usually inexplicit) family $G_w$-stable open compact subgroups $V_L\coloneqq (V_w)_{w\in Z_L}$ where each $V_w$ as ``cohomologically trivial'' $G_w$-action. But then, it would be quite unlikely that $\widehat\SC_{V_L}\otimes_{\widehat\ZZ}\ZZ_p$, introduced in Remark~\ref{rmk:c-t-SC}, can be represented by a perfect $2$-term complex.
    Indeed, at any place $w\in Z_L$ where $\Lie\Acal_L(\mfr_w)$ is not cohomologically trivial for $G_w$ (with $v=\pi(w)$), we should choose $V_w$ to be a proper subgroup of $\Acal_L^\circ(\mfr_w)$, and $\Acal^\circ_L(\mfr_w)/V_w$ is \emph{not} cohomologically trivial as a $\ZZ_p[G_w]$-module; \emph{cf.} \cite[Lemma~3.4, Proposition~3.7]{BurnsKakdeKim:EquivBSDTame}. In particular, it is difficult to control 
    \[\rho^\psi_\lambda\big(\chi^G_p(\widehat\SC_{V_L}\otimes_{\widehat\ZZ}\ZZ_p,h_p)\big),\]
    as $A(L_w)/V_w$ is neither cohomologically trivial for $G_w$ nor prime to $p$. One should also note that the formula for $v_\lambda(\Lscr_U(A,\psi))$ also involves the $G$-equivariant Euler characteristic of some (usually inexplicit) proper $G$-stable subbundle of $\Lie(\Acal_L)(-Z_L)$.
\end{rmk}

\section{Examples}\label{sec:exa}
In Theorem~\ref{th:main} we computed the $\ell$-part of the normalised leading term $\Lscr_U(A,\psi)$ under Assumption~\ref{ass:BC}, which was imposed to simplify the homological algebra especially when $\ell$ divides $|G|$. 
The main novelty lies in obtaining the $p$-part of the normalised leading term when $p$ divides $|G|$. We also have to assume the finiteness of $\Sha(A/L)$, which is still wide open in the general setting.

In this section, we  present examples of $(A,L/K,Z_L)$ that satisfy Assumption~\ref{ass:BC} for $\ell = p$. For all but the last example, $\Sha(A/L)$ is known to be finite so Theorem~\ref{th:main} can be applied unconditionally.  
\begin{exa}
    Suppose that $L/K$ is weakly ramified everywhere, and $A$ is a \emph{constant} abelian variety over $K$; that is, there is an abelian variety $A_0$ over a finite subfield $k_0$ of $K$ such that $A = A_0\times_{\Spec k_0}\Spec K$. Then, $\Sha(A/L)$ is finite by \cite{Milne:Sha} and Assumption~\ref{ass:BC}(\ref{ass:BC:p}) is automatic for $\ell = p$. Since the torsion points of $A(L)$ and $A^t(L)$ are defined over a finite subfield of $L$, to check Assumption~\ref{ass:BC} for $\ell = p$ it suffices to show that there is no place $w\in Z_L$ where $k_w$ contains the field of definition of any non-trivial point in $A_0[p](\bar k_0)$. In particular, Assumption~\ref{ass:BC} holds for $\ell = p$ (with any $Z_L$) if $A$ is a constant supersingular abelian variety.
\end{exa}

Let us now focus on the case where $A$ is a non-constant elliptic curve over $K$. Observe that Assumption~\ref{ass:BC} can be check locally at places $w\in Z_L$ \emph{except} Assumption~\ref{ass:BC}(\ref{ass:BC:MW}), which is on the $\ell$-torsion of the Mordell--Weil groups of $A$ and $A^t$. Let us give a convenient sufficient condition for Assumption~\ref{ass:BC}(\ref{ass:BC:MW}) for $\ell = p$.
\begin{lem}\label{lem:ass-MW}
    Let $A$ be an \emph{ordinary} elliptic curve over $L$. If there is a non-trivial $p$-torsion point of $A$ defined over a finite \emph{separable} extension of $L$, then $A$ can be defined over $L^p$ so the $j$-invariant of $A$ lies in $L^p$. In particular, if the $j$-invariant of $A$ does not lie in $L^p$ then $A(L)$ has no non-trivial $p$-torsion.
\end{lem}
\begin{proof}
    If there is a non-trivial $p$-torsion point of $A$ defined over a separable extension of $L$, then one can split the connected-\'etale sequence $0\to A[p]^\circ\to A[p]\to A[p]^{\et} \to 0$, which in turn enables one to factorise $[p]\colon A\to A$ as
    \[\begin{tikzcd}
        A \arrow[r,"\rho^{\et}"] & B \arrow[r, "\sigma_B"] & A
    \end{tikzcd},\]
    where $\rho^{\et}$ is a degree-$p$ \'etale isogeny and $\sigma_B$ is a degree-$p$ purely inseparable isogeny. Therefore, $\sigma_B$ can be identified with the Frobenius isogeny $B\to B^{(p)}\cong A$, so $A$ can be defined over $L^p$.
\end{proof}

\begin{rmk}
    Indeed, the converse of Lemma~\ref{lem:ass-MW} also holds since the connected-\'etale sequence for $A[p]$ splits after the Frobenius pullback. We do not need this property.
\end{rmk}

\begin{exa}\label{exa:Ulmer}
\setcounter{equation}{\value{equation}-1}
\begin{subequations}
    Let $L = \FF_q(t)$ be a rational function field of characteristic $p>3$, and let $A$ be an \emph{Ulmer elliptic curve}; i.e., the elliptic curve over $L$ given by the following equation:
    \begin{equation}\label{eq:Ulmer}
        y^2+xy = x^3 - t^{d}
    \end{equation}
    for some $d$ coprime to $p$. This elliptic curve has been studied by Ulmer \cite{Ulmer:EllCurveLargeRank}; namely, he showed that $\Sha(A/L)$ is finite \cite[Proposition~6.4]{Ulmer:EllCurveLargeRank} and computed the rank of $A(L)$ \cite[Theorems~1.5, 9.2]{Ulmer:EllCurveLargeRank}.
    
    Since $A$ is defined over $\FF_p(t^d)$, we can choose an intermediate extension $K = \FF_{q'}(t^{d'})$ so that $L/K$ is Galois. (For example, we choose $q'$ and $q$ large enough so that $e\coloneqq d/d'$ divides $q'-1$.) Let $w_0$ and $w_\infty$ respectively denote the places of $L$ corresponding to $t=0$ and $t=\infty$, and write $v_0=\pi(w_0)$ and $v_\infty = \pi(w_\infty)$. Then $L/K$ is unramified away from $\{v_0,v_\infty\}$, and admits tame ramification at worst. We now show that Assumption~\ref{ass:BC} for $\ell = p$ is satisfied for \emph{many} Ulmer elliptic curves $A$ and $L/K$ with ``minimal'' choice of $Z_L$. 
    
    Let us first recall various invariants and local properties of $A/L$, following \cite[\S2]{Ulmer:EllCurveLargeRank}. The discriminant of the model \eqref{eq:Ulmer} is 
    \begin{equation}
        \Delta\coloneqq t^d(1-2^43^3t^d)
    \end{equation}
    and at each place dividing $\Delta$ the reduction of $A$ is semistable with prime-to-$p$ local Tamagawa number (as $p>3$). Lastly, $A$ has good reduction at $w_\infty$ if $6\mid d$, and additive reduction at $w_\infty$ otherwise.

    Since the $j$-invariant of $A$ is $1/\Delta \notin L^p$, Assumption~\ref{ass:BC}(\ref{ass:BC:MW}) is satisfied for $\ell = p$ by Lemma~\ref{lem:ass-MW}. From now on, assume $d \mid 6$ so that $A$ has semistable reduction at all places of $L$, and hence Assumption~\ref{ass:BC}(\ref{ass:BC:p}) is satisfied. Let $Z_L$ be the union of the zeroes of $\Delta$ and $w_\infty$, which is the \emph{minimal} choice if $d'\ne d$. Since the local Tamagawa number at each place of $L$ is prime to $p$, to verify Assumption~\ref{ass:BC}(\ref{ass:BC:loc-vol}) for $\ell = p$ it remains to show that $\Acal_{L}(k_{w_\infty})$ is $p$-torsion free (\emph{cf.} Lemma~\ref{lem:ass}). Indeed, the fibre of $\Acal_L$ at $w_\infty$ is given by the following Weierstra\ss{} equation
    \begin{equation}\label{eq:Ulmer:infty}
       y^2 = x^3-1,
    \end{equation}
    which is the mod~$p$ reduction of an elliptic curve over $\QQ$ with complex multiplication by $\QQ(\sqrt{-3})$. In particular,  $\Acal_{L,k_{w_\infty}}$ is supersingular if $p\equiv 2 \bmod 3$ (and $p>3$), in which case $\Acal_{L}(\overline k_{w_\infty})$ is trivial.

    To summarise, suppose that $p>3$ and $p\equiv 2\bmod 3$. Let $A$ be an elliptic curve defined by the equation \eqref{eq:Ulmer} with $6\mid d$ and $p\nmid d$. We choose a finite extension $\FF_q/\FF_{q'}$ of finite fields of characteristic $p$ and a positive integer  $e$ dividing $\gcd(d,q'-1)$, and set $L\coloneqq \FF_q(t)$ and $K\coloneqq \FF_{q'}(t^{\frac{d}{e}})$. (We can arrange so that $p$ divides $[L:K]$ by manipulating $\FF_q/\FF_{q'}$.) Let $Z_L$ be the disjoint union of the zeroes of $\Delta\coloneqq t^d(1-2^43^3t^d)$ and $w_\infty$. For such $(A,L/K,Z_L)$, Assumption~\ref{ass:BC} holds for $\ell = p$. We also note that $\Sha(A/K)$ is finite, and we can arrange so that $A(L)$ has large rank by the work of Ulmer \cite{Ulmer:EllCurveLargeRank}.
\end{subequations}
\end{exa}

\begin{rmk}
    Let $\Acal_\infty$ be the elliptic curve over $\ZZ[1/6]$ defined by the equation \eqref{eq:Ulmer:infty}. We want to show that there is no non-trivial $p$-torsion in $\Acal_\infty(\FF_p)$ for any $p\geqslant 5$. If $p\equiv 2\bmod 3$, then we already observed in Example~\ref{exa:Ulmer} that $\Acal_\infty$ has supersingular good reduction so the assertion is obvious. So we may assume that $p\equiv 1\bmod 3$, in which case $\Acal_\infty$ has good \emph{ordinary} reduction at $p$. Then there exists $\omega\in\FF_p^\times\setminus\{1\}$ such that $\omega^3=1$. Now over $\FF_p$, we can rewrite \eqref{eq:Ulmer:infty} as follows
    \[y^2 = (x-1)(x-\omega)(x-\omega^2), \]
    so $\Acal_\infty(\FF_p)[2]$ has order $4$. Now suppose by contrary that  $\Acal_\infty(\FF_p)[p]$ is non-trivial, so $4p$ divides $|\Acal_\infty(\FF_p)|$. Then we have
    \[
    1 + p - |\Acal_\infty(\FF_p)| \leqslant 1 - 3p,
    \]
    which clearly violates the Weil bound. Therefore,  $\Acal_\infty(\FF_p)[p]$ should be trivial.

    Let $p$ be a prime satisfying $p\equiv 1\bmod 3$. Then we just showed that the field of definition of the $p$-torsion points of $\Acal_{\infty,\FF_p}$ is a non-trivial extension of $\FF_p$. Returning to the setting of Example~\ref{exa:Ulmer}, choose a positive integer $d$ such that $p\nmid d$ and $6\mid d$. Set $L=\FF_q(t)$ for some finite field $\FF_q/\FF_p$ and consider $A$ and $L/K$ as in Example~\ref{exa:Ulmer}. Then $(A,L/K,Z_L)$ as in Example~\ref{exa:Ulmer} satisfies Assumption~\ref{ass:BC} for $\ell = p$ provided that 
    $[k_0:\FF_p]$ does not divide $[\FF_q:\FF_p]$. Since $[k_0:\FF_p]$ divides $p-1$, one can still produce examples where $p$ divides $[L:K]$ and $p\equiv 1\bmod 3$.
\end{rmk}

\begin{exa}
    Let us now give an example of a  non-constant elliptic curve $A/K$ and a finite Galois extension $L/K$ with \emph{wild} ramification where Assumption~\ref{ass:BC} holds for $\ell = p$.
    Let $K=\FF_q(t)$ be a rational function field of characteristic of characteristic $p$, and consider an Artin--Schreier extension $L\coloneqq K(u)$ with $u^p-u = t$. Then $L/K$ is a cyclic extension of degree $p$ ramified only at the place $ v_\infty$ corresponding to $t=\infty$. For the unique place $w_\infty$ above $v_\infty$ one can check that $L_{w_\infty}/K_{v_\infty}$ is weakly wildly ramified.

    Let $A$ be an elliptic curve defined by the equation \eqref{eq:Ulmer} with $6\mid d$ and $p\nmid d$, and suppose that $p>3$ with $p\equiv 2\bmod 3$. Let $Z_L$ be the disjoint union of the bad reduction places for $A$ and $\{w_\infty\}$. Let us now verify Assumption~\ref{ass:BC}  for $(A,L/K,Z_L)$ and $\ell = p$, using the properties of $A/K$ obtained in Example~\ref{exa:Ulmer}.

    Since $A$ has semistable reduction at all places of $K$ and $L/K$ is weakly ramified everywhere, Assumption~\ref{ass:BC}(\ref{ass:BC:p}) holds. Assumption~\ref{ass:BC}(\ref{ass:BC:MW}) follows from Lemma~\ref{lem:ass-MW} since the $j$-invariant of $A/K$ does not lie in $K^p$. Since $A$ has \emph{good supersingular reduction} at $v_\infty$, $\Acal_L(k_{w_\infty})$ is trivial. 
    And since the fibre of $\Acal$ at each bad reduction place for $A/K$ is either N\'eron $d$-gon (at $t=0$) or of type $I_1$ (away of $t=0$), each of its components is rationally defined and the local Tamagawa number remains prime to $p$ under any unramified extension; \emph{cf.} \cite[\S2.2]{Ulmer:EllCurveLargeRank}. By Lemma~\ref{lem:ass}, it follows that $\Acal_L(k_w)$ has no non-trivial $p$-torsion for any $w\in Z_L$, verifying Assumption~\ref{ass:BC}(\ref{ass:BC:loc-vol}). We are not able to check if $\Sha(A/L)$ is finite.
\end{exa}

\subsection*{Acknowledgement} We thank David~Burns and Daniel~Macias~Castillo for explaining their result \cite{BurnsMaciasCastillo:EquivBSD} to us and clarified on the $\psi$-twisted regulators. The first named author (WK) was supported by the National Research Foundation of Korea(NRF) grant funded by the Korea government(MSIT). (No. RS-2023-00208018). The third named author (FT) was supported by the Japanese Society for Promotion of Sciences (JSPS) (Research grant C/21K03186).

\bibliographystyle{amsalpha}
\bibliography{bib}
\end{document}

%% file: macro-KTTT.tex
\usepackage{latexsym}

\usepackage[active]{srcltx}
\usepackage{amsfonts}
\usepackage{amscd}
\usepackage{amsxtra}
\usepackage{amstext}
\usepackage{amssymb}
\usepackage{pifont}
\usepackage{leftidx}
\usepackage{stmaryrd}

\usepackage{mathrsfs}
\usepackage{mathtools}
\usepackage{fancyhdr}
\usepackage{graphicx}
\usepackage{calc}
\usepackage{url}

\usepackage{pb-diagram,pb-xy}

\usepackage{verbatim}

\usepackage{ifthen}
\usepackage{epsfig}
\usepackage{textcomp}

\usepackage{psfrag}
\usepackage{manfnt}

\usepackage{setspace}
\usepackage{changepage}
\usepackage{layout}

\usepackage[normalem]{ulem}

\usepackage[all,cmtip]{xy}
\xyoption{matrix}

\usepackage{tikz-cd}
\usepackage{stackrel}

\usepackage{xcolor}
\usepackage{hyperref}
\hypersetup{%
  colorlinks=true,
}

\usepackage{xspace}

\usepackage[OT2,T1]{fontenc}

\usepackage{libertine}
\usepackage[libertine]{newtxmath}

\usepackage{bbm}

\numberwithin{equation}{section}

\theoremstyle{plain}
\newtheorem{thm}[equation]{Theorem}
\newtheorem*{thm*}{Theorem}

\newtheorem{lem}[equation]{Lemma}
\newtheorem*{lem*}{Lemma}

\newtheorem{prop}[equation]{Proposition}

\newtheorem*{prop*}{Proposition}

\newtheorem{cor}[equation]{Corollary}
\newtheorem*{cor*}{Corollary}

\newtheorem*{claim*}{Claim}

\newtheorem*{conj*}{Conjecture}

\theoremstyle{definition}
\newtheorem{defn}[equation]{Definition}
\newtheorem*{defn*}{Definition}

\newtheorem{notn}[equation]{Notation}
\newtheorem{notconv}[equation]{Notation and Conventions}
\newtheorem{ass}[equation]{Assumption}
\newtheorem*{ass*}{Assumption}

\newtheorem{ques}[equation]{Question}

\theoremstyle{remark}
\newtheorem{rmk}[equation]{Remark}
\newtheorem*{rmk*}{Remark}

\newtheorem{exa}[equation]{Example}
\newtheorem*{exa*}{Example}

\numberwithin{figure}{equation}
\numberwithin{table}{equation}

\def \AA {\mathbb{A}}

\def \CC {\mathbb{C}}

\def \FF {\mathbb{F}}
\def \GG {\mathbb{G}}

\def \QQ {\mathbb{Q}}
\def \RR {\mathbb{R}}

\def \ZZ {\mathbb{Z}}

\def \Acal {\mathcal{A}}
\def \Bcal {\mathcal{B}}

\def \Ecal {\mathcal{E}}
\def \Fcal {\mathcal{F}}

\def \Lcal {\mathcal{L}}

\def \Ocal {\mathcal{O}}

\def \Afr {\mathfrak{A}}

\def \afr {\mathfrak{a}}

\def \mfr {\mathfrak{m}}

\def \pfr {\mathfrak{p}}

\def \Lscr {\mathscr{L}}

\def \Rscr {\mathscr{R}}

\def \hbar {\bar{h}}

\def \Gbf {\mathbf{G}}

\newcommand{\Lrm}{\mathrm{L}}
\DeclareMathOperator{\Hrm}{H}
\DeclareMathOperator{\Krm}{K}

\newcommand{\Frm}{\mathrm{F}}

\newcommand{\RGamma}{\mathrm{R}\Gamma}
\DeclareMathOperator{\Drm}{D}
\DeclareMathOperator{\perf}{perf}
\newcommand{\Dperf}{\Drm^{\perf}}
\DeclareMathOperator{\mrm}{m}
\DeclareMathOperator{\ra}{ram}

\DeclareMathOperator{\lth}{length}

\DeclareMathOperator{\vol}{vol}

\def \riso  {\stackrel{\sim}{\rightarrow}}

\DeclareMathOperator{\Cl}{Cl}

\DeclareMathOperator{\Ig}{Ig}

\DeclareMathOperator{\coker}{coker}

\DeclareMathOperator{\Reg}{Reg}

\DeclareMathOperator{\Res}{Res}
\DeclareMathOperator{\Ind}{Ind}

\DeclareMathOperator{\gen}{gen}
\DeclareMathOperator{\Bch}{Bch}
\DeclareMathOperator{\ch}{ch}

\DeclareMathOperator{\dreg}{-reg}
\newcommand{\preg}{p\dreg}

\DeclareMathOperator{\rk}{rk}

\DeclareMathOperator{\Sel}{Sel}

\newcommand{\BSD}{\mathrm{BSD}}
\DeclareMathOperator{\id}{id}

\newcommand{\coh}{\mathrm{coh}}
\newcommand{\sgn}{\mathrm{sgn}}
\DeclareMathOperator{\Char}{Char}

\DeclareSymbolFont{cyrletters}{OT2}{wncyr}{m}{n}
\DeclareMathSymbol{\Sha}{\mathalpha}{cyrletters}{"58}

\newcommand{\et}{\mathrm{\acute{e}t}}

\newcommand{\an}{\mathrm{an}}

\newcommand{\triv}{\mathbf{1}}

\DeclareMathOperator{\Spec}{Spec}

\newcommand{\alg}{\mathrm{alg}}

\DeclareMathOperator{\Nrd}{Nrd}

\newcommand{\ram}{\mathrm{ram}}
\newcommand{\wild}{\mathrm{wild}}

\newcommand{\bmu}{\boldsymbol{\mu}}

\newcommand{\tors}{\mathrm{tors}}
\DeclareMathOperator{\Frac}{Frac}

\newcommand{\QQp}{\QQ_p}

\newcommand{\ZZp}{\ZZ_p}

\newcommand{\Fp}{\FF_p}

\DeclareMathOperator{\lo}{loc}

\DeclareMathOperator{\Frob}{Frob}
\DeclareMathOperator{\Fr}{Fr}

\newcommand{\ord}{\mathrm{ord}}

\newcommand{\ssim}{\mathrm{ss}}

\DeclareMathOperator{\Gal}{Gal}

\newcommand{\fl}{\mathrm{fl}}

\newcommand{\ari}{\mathrm{ar}}
\DeclareMathOperator{\SC}{SC}
\newcommand{\whSC}{\widehat{\SC}}
\newcommand{\NT}{\mathrm{NT}}

\DeclareMathOperator{\Hom}{Hom}

\DeclareMathOperator{\End}{End}

\DeclareMathOperator{\Lie}{Lie}

\newcommand{\rad}{\mathbf{rad}}

\newcommand{\tf}{\mathrm{tf}}

\DeclareMathOperator{\Ir}{Ir}
